\documentclass[a4paper,11pt]{article}
\usepackage{amssymb,amsmath,amsthm}
\usepackage{authblk}
\usepackage{booktabs}
\usepackage{cancel}
\usepackage{colortbl}
\usepackage{comment}
\usepackage[hmargin={26mm,26mm},vmargin={30mm,35mm}]{geometry}
\usepackage{ifthen}
\usepackage{hyperref}
\usepackage{multirow}
\usepackage{nicefrac}
\usepackage{paralist}
\usepackage{tikz-cd}
\usepackage{xcolor}
\usepackage{pgfplotstable}
\usepackage[font={footnotesize}]{subcaption}

\usepackage[color,notref,notcite,final]{showkeys}

\usepackage{makeidx}
\usepackage{graphicx}
\graphicspath{{./data/meshes/}}

\usepackage[
backend=biber,
style=numeric,
sorting=ynt,
url=false,
isbn=false
]{biblatex}
\usepackage{accents}
\numberwithin{equation}{section}
\definecolor{darkbrown}{HTML}{996633}
\newcommand{\logLogSlopeTriangle}[5]
{

    \pgfplotsextra
    {
        \pgfkeysgetvalue{/pgfplots/xmin}{\xmin}
        \pgfkeysgetvalue{/pgfplots/xmax}{\xmax}
        \pgfkeysgetvalue{/pgfplots/ymin}{\ymin}
        \pgfkeysgetvalue{/pgfplots/ymax}{\ymax}

        \pgfmathsetmacro{\xArel}{#1}
        \pgfmathsetmacro{\yArel}{#3}
        \pgfmathsetmacro{\xBrel}{#1-#2}
        \pgfmathsetmacro{\yBrel}{\yArel}
        \pgfmathsetmacro{\xCrel}{\xArel}

        \pgfmathsetmacro{\lnxB}{\xmin*(1-(#1-#2))+\xmax*(#1-#2)} 
        \pgfmathsetmacro{\lnxA}{\xmin*(1-#1)+\xmax*#1} 
        \pgfmathsetmacro{\lnyA}{\ymin*(1-#3)+\ymax*#3} 
        \pgfmathsetmacro{\lnyC}{\lnyA+#4*(\lnxA-\lnxB)}
        \pgfmathsetmacro{\yCrel}{\lnyC-\ymin)/(\ymax-\ymin)} 

        \coordinate (A) at (rel axis cs:\xArel,\yArel);
        \coordinate (B) at (rel axis cs:\xBrel,\yBrel);
        \coordinate (C) at (rel axis cs:\xCrel,\yCrel);

        \draw[black]   (A)-- node[pos=0.5,anchor=north] {\scriptsize{1}}
                    (B)-- 
                    (C)-- node[pos=0.,anchor=west] {\scriptsize{\color{#5}#4}} 
                    (A);
    }
}

\newtheorem{theorem}{Theorem}

\newtheorem{lemma}[theorem]{Lemma}

\theoremstyle{remark}
\newtheorem{remark}[theorem]{Remark}
\theoremstyle{definition}

\newtheorem{definition}[theorem]{Definition}

\newcommand{\st}{\,:\,}
\newcommand{\Real}{\mathbb{R}}

\newcommand{\Ball}[2]{B(#1,#2)}
\newcommand{\bigO}{\mathcal{O}}

\DeclareRobustCommand{\bvec}[1]{\boldsymbol{#1}}
\newcommand{\uvec}[1]{\underline{\bvec{#1}}}
\newcommand{\cvec}[1]{\bvec{\mathcal{#1}}}
\newcommand{\ul}[1]{\underline{#1}}

\newcommand{\rdofE}[1]{\bvec{R}_{#1,E}}
\newcommand{\srdofE}[1]{{R}_{#1,E}} 
\newcommand{\rdofV}[1]{\bvec{R}_{#1,V}}
\newcommand{\rdofFg}[1]{\bvec{R}_{#1,\cvec{G},F}}
\newcommand{\rdofFgc}[1]{\bvec{R}^c_{#1,\cvec{G},F}}
\newcommand{\gsdofF}[1]{{G}_{#1,F}}
\newcommand{\gdofV}[1]{\bvec{G}_{#1,V}}
\newcommand{\gdofE}[1]{\bvec{G}_{#1,E}}

\newcommand{\dotp}[1]{{#1}'}

\DeclareMathOperator{\TGRAD}{\bf \nabla}
\DeclareMathOperator{\TDIV}{\bf \nabla \cdot}
\DeclareMathOperator{\TLAPLACIAN}{\bf \Delta}
\DeclareMathOperator{\DIV}{div}
\DeclareMathOperator{\ROT}{rot}
\DeclareMathOperator{\GRAD}{\bf grad}
\DeclareMathOperator{\VROT}{\bf rot}
\DeclareMathOperator{\CURL}{\bf curl}
\DeclareMathOperator{\ID}{{\bf I}_{3,3}}
\DeclareMathOperator{\Tr}{Tr}
\DeclareMathOperator{\InvDivGrad}{\boldsymbol{P}_{\Tr}}
\DeclareMathOperator{\tdot}{\bf :}

\newcommand{\Hv}[1]{\bvec{H}^1(#1)}
\newcommand{\Ls}[1]{L^2(#1)}
\newcommand{\Lt}[1]{\bvec{L}^2(#1)}
\newcommand{\Besov}[1]{\bvec{B}^{1/2,2}(#1)}
\newcommand{\ball}[2]{B(#1,#2)}

\newcommand{\uHgrad}{\underline{X}_{\GRAD,T}^k}
\newcommand{\uHgradF}{\underline{X}_{\GRAD,F}^k}
\newcommand{\uHgradE}{\underline{X}_{\GRAD,E}^k}
\newcommand{\uHgradh}{\underline{X}_{\GRAD,h}^k}
\newcommand{\uHrot}{\underline{X}_{\VROT,F}^k} 

\newcommand{\uHroth}{\underline{X}_{\VROT,h}^k} 

\newcommand{\uHcurl}{\underline{\bvec{X}}_{\CURL,T}^k}
\newcommand{\uHcurlF}{\underline{\bvec{X}}_{\CURL,F}^k}
\newcommand{\uHcurlE}{\underline{\bvec{X}}_{\CURL,E}^k}
\newcommand{\uHcurlh}{\underline{\bvec{X}}_{\CURL,h}^k}
\newcommand{\uHv}{\underline{\bvec{X}}_{\TGRAD,T}^k}
\newcommand{\uHvF}{\underline{\bvec{X}}_{\TGRAD,F}^k}
\newcommand{\uHvE}{\underline{\bvec{X}}_{\TGRAD,E}^k}
\newcommand{\uHvh}{\underline{\bvec{X}}_{\TGRAD,h}^k}
\newcommand{\uHvhstar}{\underline{\bvec{X}}_{\TGRAD,h,\star}^k}
\newcommand{\uHvhzero}{\underline{\bvec{X}}_{\TGRAD,h,0}^k}
\newcommand{\uHvhzeroD}{\underline{\bvec{X}}_{\TGRAD,h,D}^k}
\newcommand{\uLsF}{\underline{\bvec{X}}_{L^2,F}^k}
\newcommand{\uLs}{\underline{\bvec{X}}_{L^2,T}^k}
\newcommand{\uLsh}{\underline{\bvec{X}}_{L^2,h}^k}
\newcommand{\uLshstar}{\underline{\bvec{X}}_{L^2,h,\star}^k}

\newcommand{\uLt}{\underline{\bvec{X}}_{\bvec{L}^2,T}^{k+1}}
\newcommand{\uLth}{\underline{\bvec{X}}_{\bvec{L}^2,h}^{k+1}}

\newcommand{\Poly}[2][]{\mathcal{P}_{#1}^{#2}}

\newcommand{\bPoly}[2][]{\boldsymbol{\mathcal{P}}_{#1}^{#2}}
\newcommand{\Roly}[1]{\boldsymbol{\mathcal{R}}^{#1}}
\newcommand{\Rolyb}[1]{\overline{\boldsymbol{\mathcal{R}}}^{#1}}
\newcommand{\Goly}[1]{\boldsymbol{\mathcal{G}}^{#1}}
\newcommand{\Golyb}[1]{{\boldsymbol{\mathcal{G}}}^{#1}}
\newcommand{\EPoly}[1]{\widetilde{\boldsymbol{\mathcal{P}}}_{n,E}^{#1}(E)}

\newcommand{\NE}[1]{\boldsymbol{\mathcal{N}}^{#1}}

\newcommand{\RT}[1]{\boldsymbol{\mathcal{RT}}^{#1}}
\newcommand{\RTb}[1]{\overline{\boldsymbol{\mathcal{RT}}}^{#1}}
\newcommand{\FRTb}[1]{\widetilde{\boldsymbol{\mathcal{P}}}^{#1}(F)}

\newcommand{\vertiii}[1]{{\left\vert\kern-0.25ex\left\vert\kern-0.25ex\left\vert #1 
    \right\vert\kern-0.25ex\right\vert\kern-0.25ex\right\vert}}
\newcommand{\opnHgrad}[2][T]{\vertiii{#2}_{\GRAD,#1}}
\newcommand{\opnHcurl}[2][T]{\vertiii{#2}_{\CURL,#1}}
\newcommand{\opnNa}[2][T]{\vertiii{#2}_{\TGRAD,#1}}
\newcommand{\normNa}[2][T]{\Vert #2 \Vert_{\TGRAD,#1}}
\newcommand{\opnLt}[2][T]{\vertiii{#2}_{\bvec{L}^2,#1}}
\newcommand{\normLt}[2][T]{\Vert #2 \Vert_{\bvec{L}^2,#1}}
\newcommand{\normLs}[2][T]{\Vert #2 \Vert_{L^2,#1}}
\newcommand{\norm}[2][]{\left \|#2 \right \|_{#1}}
\newcommand{\seminorm}[2][]{\left |#2 \right |_{#1}}
\newcommand{\spLt}[3][T]{\left (#2, #3 \right )_{\bvec{L}^2,#1}}
\newcommand{\stLt}[3][T]{\text{s}_{\bvec{L}^2,#1} \left (#2, #3 \right )}
\newcommand{\spNa}[2]{\left (#1, #2 \right )_{\TGRAD,T}}
\newcommand{\stNa}[2]{\text{s}_{\TGRAD,T} \left (#1, #2 \right )}
\newcommand{\semiBtr}[2][\partial \Omega]{\vert #2 \vert^{\Diamond}_{\Besov{#1}}}
\newcommand{\normHNa}[1]{\Vert #1 \Vert_{\mu, \TGRAD, 1, h}}

\newcommand{\uIgrad}[1][T]{\ul{I}_{\GRAD,#1}^k}
 
\newcommand{\uIcurl}[1][T]{\ul{I}_{\CURL,#1}^k}
\newcommand{\uIH}[1][T]{\uvec{I}_{\TGRAD,#1}^k}
\newcommand{\uIL}[1][T]{\uvec{I}_{\bvec{L}^2(#1)}^k}
\newcommand{\uIgradh}{\ul{I}_{\GRAD,h}^k}
\newcommand{\uIroth}{\ul{I}_{\VROT,h}^k}
\newcommand{\uIcurlh}{\uvec{I}_{\CURL,h}^k}
\newcommand{\uIHh}{\uvec{I}_{\TGRAD,h}^k}
\newcommand{\uILh}{\uvec{I}_{\bvec{L}^2,h}^k}
\newcommand{\uILsh}{\ul{I}_{{L}^2,h}^k}

\newcommand{\lproj}[2]{\pi_{\mathcal{P},#2}^{#1}}
\newcommand{\vlproj}[2]{\boldsymbol{\pi}_{\cvec{P},#2}^{#1}}

\newcommand{\Rproj}[2][T]{\bvec{\pi}_{\cvec{R},#1}^{#2}}
\newcommand{\Rcproj}[2][T]{\bvec{\pi}_{\cvec{R},#1}^{c,#2}}

\newcommand{\Gproj}[2][T]{\bvec{\pi}_{\cvec{G},#1}^{#2}}
\newcommand{\Gcproj}[2][T]{\bvec{\pi}_{\cvec{G},#1}^{c,#2}}

\newcommand{\Gbcproj}[2][T]{\bvec{\pi}_{\Golyb{},#1}^{c,#2}}

\newcommand{\RTproj}[2][T]{\bvec{\pi}_{\RT{},#1}^{#2}}
\newcommand{\RTbproj}[2][T]{\bvec{\pi}_{\RTb{},#1}^{#2}}
\newcommand{\FRTbproj}[1]{\bvec{\pi}_{\widetilde{\boldsymbol{\mathcal{P}}},F}^{#1}}

\newcommand{\bgvec}[2][T]{\bvec{#2}_{\Goly{},#1}}
\newcommand{\bgcvec}[2][T]{\bvec{#2}_{\Golyb{},#1}^c}
\newcommand{\rkvec}[2][T]{\bvec{#2}_{\Roly{},#1}}
\newcommand{\rkcvec}[2][T]{\bvec{#2}_{\Roly{},#1}^c}
\newcommand{\brvec}[2][T]{\bvec{#2}_{\Rolyb{},#1}}
\newcommand{\brcvec}[2][T]{\bvec{#2}_{\Rolyb{},#1}^c}
\newcommand{\FRTnvec}[1]{\bvec{#1}_{\boldsymbol{\mathcal{P}},\nF}}

\newcommand{\LiftNaF}[1]{\bvec{R}_{#1}}

\newcommand{\GF}{\bvec{G}_{F}^{k}}
\newcommand{\GFp}{\bvec{G}_{F}^{\perp k}}
\newcommand{\GT}{\bvec{G}_{T}^{k}}
\newcommand{\uGE}{\uvec{G}_E^{k}}
\newcommand{\uGF}{\uvec{G}_F^{k}}
\newcommand{\uGT}{\uvec{G}_T^{k}}
\newcommand{\uGh}{\uvec{G}_h^{k}}

\newcommand{\CE}{\bvec{C}_{E}^{k}}
\newcommand{\CF}{{C}_{F}^{k}}
\newcommand{\bCF}{\bvec{C}_{F}^{k}}
\newcommand{\CT}{\bvec{C}_{T}^{k}}
\newcommand{\uCF}{\uvec{C}_F^{k}}
\newcommand{\uCT}{\uvec{C}_T^{k}}
\newcommand{\uCh}{\uvec{C}_h^{k}}

\newcommand{\NaE}{\TGRAD_E^{k+2}}
\newcommand{\NaF}{\TGRAD_F^{k+1}}
\newcommand{\NaT}{\TGRAD_T^{k+1}}
\newcommand{\uNaF}{\underline{\TGRAD}_F^{k+1}}
\newcommand{\uNaT}{\underline{\TGRAD}_T^{k+1}}
\newcommand{\uNah}{\underline{\TGRAD}_h^{k+1}}

\newcommand{\DT}{D_T^{k}}
\newcommand{\Dh}{\underline{D}_h^{k}}

\newcommand{\trna}{\gamma_{\TGRAD,F}^{k+2}}
\newcommand{\trnac}{\widetilde{\gamma_{\TGRAD,F}^{k+2}}^c}
\newcommand{\trgrad}{\gamma_{\GRAD,F}^{k+1}}
\newcommand{\trcurl}{\gamma_{t,\ROT,F}^{k}}
\newcommand{\pna}{P_{\TGRAD,T}^{k+1}}

\newcommand{\AdjG}{\tilde{\mathcal{E}}_{\TGRAD,h}}
\newcommand{\AdjL}{\tilde{\mathcal{E}}_{\TLAPLACIAN,h}}

\newcommand{\Hh}{\mathcal{H}}
\newcommand{\Mh}{\mathcal{M}_h}
\newcommand{\Th}{\mathcal{T}_h}
\newcommand{\Fh}{\mathcal{F}_h}
\newcommand{\Eh}{\mathcal{E}_h}
\newcommand{\Xh}{\mathcal{X}_h}
\newcommand{\Vh}{\mathcal{V}_h}
\newcommand{\FT}{\mathcal{F}_T}
\newcommand{\ET}{\mathcal{E}_T}
\newcommand{\EF}{\mathcal{E}_F}

\newcommand{\VE}{\mathcal{V}_E}
\newcommand{\Ech}{\mathcal{E}_h}

\newcommand{\nE}{\bvec{t}_E}
\newcommand{\nFE}{\bvec{n}_{FE}}
\newcommand{\nF}{\bvec{n}_F}
\newcommand{\nOmega}{\bvec{n}_{\Omega}}

\newcommand{\wFE}{\omega_{FE}}
\newcommand{\wTF}{\omega_{TF}}

\newcommand{\ttr}[1]{\bvec{#1}_{t,F}} 
\newcommand{\Tttr}[1]{\bvec{#1}_{\otimes t,F}} 
\newcommand{\Tttrnb}[1]{{#1}_{\otimes t,F}} 

\newcommand{\aSk}{\text{a}_h}
\newcommand{\bSk}{\text{b}_h}
\newcommand{\ASk}{\mathcal{A}_h}
\newcommand{\LSk}{\mathcal{L}_h}
\newcommand{\Aerr}[1][h]{\mathcal{E}_{#1}}

\DeclareMathOperator{\Ker}{Ker}
\DeclareMathOperator{\Image}{Im}

\title{An arbitrary-order fully discrete Stokes complex on general polyhedral meshes.}	
\author[1]{Marien-Lorenzo Hanot \footnote{\href{mailto:marien-lorenzo.hanot@umontpellier.fr}{marien-lorenzo.hanot@umontpellier.fr}}}
\affil[1]{IMAG, UMR CNRS 5149 and Université de Montpellier, Montpellier, France}

\begin{document}
\maketitle

\begin{abstract}
In this paper we present an arbitrary-order fully discrete Stokes complex on general polyhedral meshes.
We enriche the fully discrete de Rham complex with the addition of a full gradient operator defined on vector fields 
and fitting into the complex.
We show a complete set of results on the novelties of this complex:
exactness properties, uniform Poincaré inequalities and primal and adjoint consistency.
The Stokes complex is especially well suited for problem involving Jacobian, divergence and curl,
like the Stokes problem or magnetohydrodynamic systems.
The framework developed here eases the design and analysis of scheme for such problems. 
Schemes built that way are nonconforming and benefit from the exactness of the complex.
We illustrate with the design and study of a scheme to solve the Stokes equations 
and validate the convergence rates with various numerical tests.
\end{abstract}

\section{Introduction.}
The exactness of the divergence free condition plays an important role in the numerical resolution
of incompressible fluid equations, \cite{2017DivConstraint} provides a detailed review.
This kind of conservation requires the discrete spaces to reproduce relevant algebraic properties of the continuous spaces.
Let $\Omega$ be a domain of $\Real^3$. This exactness can be expressed by the following differential complex:
\begin{equation} \label{cd:L2deRham}
\begin{tikzcd}
\Real \arrow[r,"i_\Omega"] & H^1(\Omega) \arrow[r,"\GRAD"] & \bvec{H}(\text{curl},\Omega) \arrow[r,"\CURL"] & \bvec{H}(\text{div}, \Omega) \arrow[r,"\DIV"] & L^2(\Omega) \arrow[r,"0"] & \lbrace 0 \rbrace.
\end{tikzcd}
\end{equation}
Many discrete counterparts of the complex \eqref{cd:L2deRham} have been developed.
See \cite{feec-cbms} for a thorough exposition and an extensive bibliography.
Although many partial differential equations can be expressed using the de Rham complex,
the lack of smoothness causes issues for some equations,
in particular for the Stokes equations (see \cite{ARNOLD2012}).
So a smoother variant more suited to the Stokes equations (hence called Stokes complex) has been considered.
In three dimensions the Stokes complex is written:
\begin{equation} \label{cd:Stokescomplex}
\begin{tikzcd}
\Real \arrow[r,"i_\Omega"] & H^2(\Omega) \arrow[r,"\GRAD"] & \bvec{H}^2(\Omega) \arrow[r,"\CURL"] & \bvec{H}^1(\Omega) \arrow[r,"\DIV"] & L^2(\Omega) \arrow[r,"0"] & \lbrace 0 \rbrace.
\end{tikzcd}
\end{equation}
The development of discrete counterparts of this smoother complex is much more complicated.
See \cite[Chapter~8.7]{feec-cbms} for a history.
Although such constructions exist (for example \cite{Nei2015})
they often have drawbacks.
Recurrent problems can be a large minimal degree and thus numerous unknowns 
as well as difficulties to enforce Dirichlet boundary conditions.
The subject is very active with many recent advances: \cite{hu2020family,huang2020nonconforming}.
Another issue of these constructions is that they are frequently constrained to conformal simplicial meshes,
which is limiting for some geometries as well as for the possibility of refinement or agglomeration.
A construction of the Stokes complex in virtual finite elements on general meshes has also been recently developed (see \cite{2020VEM}).

Our construction works on general polyhedral meshes and for arbitrary polynomial degrees.
The discrete spaces consist of polynomial spaces on the elements of all geometric dimensions: cells, faces, edges and vertices.
Compared to the virtual finite element method, the basis functions are explicitly known but do not live in a subspace of continuous functions.
The discrete differential operators are therefore necessarily different from the continuous operators.
They are constructed according to integration by parts formulae and in a sense converge with the discrete spaces to the continuous operators 
(see the consistency results of Section \ref{Consistencyresults}).
A discretization of the de Rham complex \eqref{cd:L2deRham} has been developed in detail by D. A. Di Pietro and J. Droniou \cite{ddr}.
One can find in the introduction a very complete comparison of the different methods leading to discrete de Rham complex on polytopal meshes.
Our paper is a continuation of \cite{ddr}:
Our construction is based upon it, and we add the necessary basis functions required for the increased smoothness of the Stokes complex.
We define and analyze in detail the Jacobian operator while checking its compatibility with the complex.

More precisely we show the exactness of the complex, the existence of uniform Poincaré inequalities and many consistency results
as well as a discrete version of the right inverse for the divergence for the discrete norm $\bvec{H}^1$.
Finally, we apply this to the Stokes equations: we show well-posedness, give an error estimate and find an optimal convergence rate of order $\bigO(h^{k+1})$,
$h$ being the size of the mesh and $k \geq 0$ the chosen polynomial degree.
We also explore other choices of boundary conditions and validate numerically every result.

The remaining of the paper is organized as follows.
In Section \ref{Setting} we introduce the general setting.
We define the discrete spaces and operators (interpolators, differential operators and norms) in Section \ref{Discretecomplex}.
In Section \ref{Complexproperty} we show that our construction is indeed a complex which is exact for contractible domains.
In Section \ref{Consistencyresults} we establish consistency properties, including primal and dual consistencies.
The Stokes equations are defined in Section \ref{Stokes} and other boundary conditions are studied in Section \ref{Alternateboundaryconditions}.
We display our numerical results in Section \ref{Numericaltests}.
Finally we prove technical propositions in the appendices: on polynomial spaces in appendix \ref{Resultsonpolynomialspaces} and on various lifts in appendix \ref{Tracelifting}.

\section{Setting.} \label{Setting}
This section is dedicated to the introduction of the setting and various notations that will be used throughout the paper.
We follow the conventions of \cite{ddr}.

\subsection{Mesh and orientation.}
In the following we consider a polyhedral domain $\Omega \subset \Real^3$ and keeping the notation of \cite{ddr}, 
for any set $Y \subset \Real^3$, we write $h_Y := \sup \lbrace \vert \bvec{x} - \bvec{y} \vert \st \bvec{x}, \bvec{y} \in Y \rbrace$ and $\vert Y \vert$ its Hausdorff measure.
We consider on this domain a mesh sequence $\Mh = \Th \cup \Fh \cup \Eh \cup \Vh$ parameterized by a positive real parameter $h \in \Hh$.
Here $\Th$ is a finite collection of open convex polyhedra such that $\overline{\Omega} = \cup_{T \in \Th} \overline{T}$ and $h = \max_{T \in \Th} h_T > 0$,
$\Fh$ is the collection of open polygonal faces of the cells, $\Eh$ is the collection of open polygonal edges, and $\Vh$ the collection of vertices.
This sequence must be regular in the sense of \cite[Definition~1.9]{hho} with the regularity constant $\rho$.
For any cell $T \in \Th$, we write $\FT$ the set of faces of this cell.
Likewise for any face $F \in \Fh$, we write $\EF$ the set of edges of this face.

We take $k \ge 0$ a fixed polynomial degree.
In the following most inequalities hold up to a positive constant. 
This constant depends only on some parameters, here on the chosen polynomial degree $k$, on the regularity parameter of the mesh sequence $\rho$ and on the domain $\Omega$.

We denote the inequality up to a positive constant by
\[ A \lesssim B 
\]
meaning there exists $C \in \Real^*_+$ depending only on some parameters (here usually only on $k$, $\rho$ and $\Omega$) such that $A \leq C B$.
We also write 
\[ A \approx B 
\]
meaning that $A \lesssim B$ and $B \lesssim A$.

For any $h$, we set the orientation of any face $F \in \Fh$ and any edge $E \in \Eh$ 
by prescribing a unit normal vector $\nF$ and unit tangent vector $\nE$.
For any face $F \in \Fh$ and any $E \in \EF$ we also define the unit vector $\nFE$ normal to $E$ lying in the plane tangent to $F$, 
and such that $(\nE,\nFE,\nF)$ is right-handed in $\Real^3$.
To keep track of the relative orientation we define for any $T \in \Th$ and $F \in \FT$, $\wTF \in \lbrace -1,1\rbrace$ 
such that $\wTF \nF$ points out of $T$, 
and for any $F \in \Fh$, $E \in \EF$ we define $\wFE \in \lbrace -1,1\rbrace$ such that $\wFE \nFE$ points out of $F$.
We also define $\nOmega$ as the outward pointing unit normal vector on the boundary $\partial \Omega$.
We note by ${}^\perp$ the rotation of angle $\pi/2$ in the oriented plane $F$.

\subsection{Polynomial spaces.}
For any entity $X \in \lbrace E, F, T\rbrace$, we denote by 
$\Poly{k}(X)$ the set of polynomials of total degree at most $k$ on $X$, 
by $\bPoly{k}(X)$ the set of vector valued polynomials,
and by $(\bPoly{k}(X)^\intercal)^3$ the set of triples of polynomials on $X$ forming the rows of a matrix valued polynomial.
We use the conventions
$\Poly{-1}(X) := \lbrace 0 \rbrace$ and 
$\Poly{0,k}(X) := \lbrace P \in \Poly{k}(X) \st \int_X P = 0 \rbrace$.
We also define the broken polynomial space
\begin{equation}
\Poly{k}(\Xh) := \lbrace P_h \in L^2(\Xh) \st \forall X \in \Xh, P_{h \vert X} \in \Poly{k}(X) \rbrace,
\end{equation}
as well as its continuous counterpart
\begin{equation}
\Poly[c]{k}(\Xh) := \lbrace P_h \in C^0(\Xh) \st \forall X \in \Xh, P_{h \vert X} \in \Poly{k}(X) \rbrace.
\end{equation}

\begin{remark} \label{rem:continuouspoly}
Continuous polynomials can be characterized by their values at the interface and their lower order moments on the elements.
An explicit construction is deduced from Lemma \ref{lemma:normpcz}. 
In the context of edges we can see the isomorphism between 
$\Poly[c]{k+2}(\Eh)$ and $\Poly{k}(\Eh) \times \Real^{\Vh}$.
\end{remark}

For the sake of readability we quote two lemmas on discrete spaces: \cite[Lemma~1.28 and Lemma~1.32]{hho} (in a slightly more restrictive setting):
\begin{lemma}[Discrete inverse Poincaré] \label{lemma:discretepoincare}
Let $X$ be an element of $\Th \cup \Fh \cup \Eh$. Let a positive integer $l$ and a real number $p \in [1,\infty]$ be fixed. Then, the following inequality holds: For all $v \in \Poly{l}(X)$,
\begin{equation}
\Vert \nabla v \Vert_{L^p(X)} \lesssim h_X^{-1} \Vert v \Vert_{L^p(X)},
\end{equation}
with hidden constant depending only on $\rho$, $l$ and $p$.
\end{lemma}
\begin{lemma} \label{lemma:discretetrace}
Let $p \in [1,\infty]$ be a fixed real number and $l \ge 0$ be a fixed integer. 
Then for all $h \in \Hh$, all $T \in \Th$ (resp. $F \in \Fh$), all $F \in \Fh$ (resp. $E \in \Eh$), all $v \in \Poly{l}(F)$, 
\begin{equation}
\Vert v \Vert_{L^p(F)} \lesssim h^{-\frac{1}{p}}_T \Vert v \Vert_{L^p(T)} 
\end{equation}
with hidden constant depending only on $\rho$, $l$ and $p$.
\end{lemma}

We will also use Koszul complements (see \cite[Section~2.4]{ddr}).
We consider for any face $T \in \Th$ a point $\bvec{x}_T$ such that $\Ball{\bvec{x}_T}{\rho h_T} \subset T$.
Then we define the following subspaces of $\bPoly{k}(T)$:
\begin{equation} \label{eq:defkoszul}
\begin{aligned}
&\Goly{k}(T) := \GRAD \Poly{k+1}(T), &\Goly{c,k}(T) := (\bvec{x} - \bvec{x}_T) \times \bPoly{k-1}(T),\\
&\Roly{k}(T) := \CURL \bPoly{k+1}(T), &\Roly{c,k}(T) := (\bvec{x} - \bvec{x}_T) \Poly{k-1}(T).
\end{aligned}
\end{equation}
These spaces are such that:
\begin{equation} \label{eq:compkoszul}
\bPoly{k}(T) = \Goly{k}(T) \oplus \Goly{c,k}(T) = \Roly{k}(T) \oplus \Roly{c,k}(T),
\end{equation}
however the sum is not orthogonal for the $L^2$ scalar product.

Similarly, in $2$ dimensions for any face $F \in \Fh$, we define:
\begin{equation} \label{eq:defkoszul2d}
\begin{aligned}
&\Goly{k}(F) := \GRAD \Poly{k+1}(F), &\Goly{c,k}(F) := (\bvec{x} - \bvec{x}_F)^\perp \Poly{k-1}(F),\\
&\Roly{k}(F) := \VROT \Poly{k+1}(F), &\Roly{c,k}(F) := (\bvec{x} - \bvec{x}_F) \Poly{k-1}(F).
\end{aligned}
\end{equation}
We also have the following isomorphisms:
\begin{align}
\VROT \st& \Poly{0,k}(F) \rightarrow \Roly{k-1}(F),& \label{eq:isovrot}\\
\DIV \st& \Roly{c,k}(F) \rightarrow \Poly{k-1}(F),& \DIV \st& \Roly{c,k}(T) \rightarrow \Poly{k-1}(T), \label{eq:isodiv}\\
&& \CURL \st& \Goly{c,k}(T) \rightarrow \Roly{k-1}(T). \label{eq:isocurl}
\end{align}
We can deduce from Lemma \ref{lemma:discretepoincare} that
$\vertiii{\VROT} \lesssim h^{-1}$, $\vertiii{\DIV} \lesssim h^{-1}$, $\vertiii{\CURL} \lesssim h^{-1}$ 
and from \cite[Lemma~A.9]{ddr} that 
$\vertiii{(\VROT)^{-1}} \lesssim h$,
$\vertiii{(\DIV)^{-1}} \lesssim h$,
$\vertiii{(\CURL)^{-1}} \lesssim h$.

For $X = \lbrace T,F \rbrace$ we define the local spaces of Nedelec and of Raviart-Thomas respectively by:
\begin{equation} \label{eq:defNEb}
\NE{k}(X) := \Goly{k-1}(X) \oplus \Goly{c,k}(X),\quad \RT{k}(X) := \Roly{k-1}(X) \oplus \Roly{c,k}(X).
\end{equation}
These spaces are strictly contained between $\bPoly{k-1}(X)$ and $\bPoly{k}(X)$.
Another important property given in \cite[Proposition~A.8]{ddr} is that for any cell $T \in \Th$ (resp. $F \in \Fh$)
and any face of this cell $F \in \FT$ (resp. $E \in \EF$):
\begin{equation} \label{eq:trNRT}
\begin{aligned}
\forall \bvec{v}_F \in \NE{k}(F),&\; (\bvec{v}_F)_{\vert E} \cdot \nE \in \Poly{k-1}(E),\\
\forall \bvec{w}_F \in \RT{k}(F),&\; (\bvec{w}_F)_{\vert E} \cdot \nFE \in \Poly{k-1}(E),\\
\forall \bvec{v}_T \in \NE{k}(T),&\; (\bvec{v}_T)_{\vert E} \cdot \nE \in \Poly{k-1}(E),\\
\forall \bvec{w}_T \in \RT{k}(T),&\; (\bvec{w}_T)_{\vert F} \cdot \nF \in \Poly{k-1}(F),\\
\forall \bvec{v}_T \in \NE{k}(T),&\; (\bvec{v}_T)_{\vert F} \times \nF \in \RT{k}(F).
\end{aligned}
\end{equation}

In order to fix the notation we write
\begin{equation}
(\Roly{c,k}(F)^\intercal)^2 = \begin{pmatrix} \Roly{c,k}(F)^\intercal \\ \Roly{c,k}(F)^\intercal \end{pmatrix},\;
(\Roly{c,k}(T)^\intercal)^3 = \begin{pmatrix} \Roly{c,k}(T)^\intercal \\ \Roly{c,k}(T)^\intercal \\ \Roly{c,k}(T)^\intercal \end{pmatrix}.
\end{equation}
We take differential operators to be acting row-wise on matrix valued functions, and we use the convention 
\begin{equation*}
\TGRAD \begin{pmatrix} v_1 \\ v_2 \\ v_3 \end{pmatrix} := 
\begin{pmatrix}
\partial_1 v_1 & \partial_2 v_1 & \partial_3 v_1 \\
\partial_1 v_2 & \partial_2 v_2 & \partial_3 v_2 \\
\partial_1 v_3 & \partial_2 v_3 & \partial_3 v_3 \end{pmatrix} .
\end{equation*}
We define the space $\Rolyb{c,k}(T)$ by
\begin{equation} \label{eq:defRbc}
\Rolyb{c,k}(T) := \lbrace W \in (\Roly{c,k}(T)^\intercal)^3 \st \Tr{W} = 0 \rbrace .
\end{equation}
An explicit description of this space is given by Lemma \ref{lemma:explicitRbCompl}.
Let us now construct a complement to this space.
First noticing that $\Tr((\Roly{c,k}(T)^\intercal)^3) = \Poly{0,k}(T)$,
we can consider the inverse operator
$\InvDivGrad^k : \Poly{0,k}(T) \rightarrow (\Roly{c,k}(T)^\intercal)^3$:
\begin{equation} \label{eq:defInvDivGrad}
\InvDivGrad^k := \begin{pmatrix} \DIV^{-1} \\ \DIV^{-1} \\ \DIV^{-1} \end{pmatrix} \circ \GRAD ,
\end{equation}
where $\DIV$ is the isomorphism from $\Roly{c,k}(T)$ into $\Poly{k-1}(T)$ given by \eqref{eq:isodiv}.
Then we define the space:
\begin{equation} \label{eq:defRb}
\Rolyb{k}(T) := \InvDivGrad^k \Poly{0,k}(T) .
\end{equation}
Lemma \ref{lemma:Rcdec} shows that the spaces $\Rolyb{c,k}(T)$ and $\Rolyb{k}(T)$ are complementary.
A similar construction of the spaces $\Rolyb{c,k}(F)$ and $\Rolyb{k}(F)$ holds in $2$ dimensions, and is used in Appendix \ref{2DComplex}.

\begin{remark} \label{rem:divInvDivGrad}
By construction, we have:
$\TDIV \Rolyb{k}(T) = \TDIV \InvDivGrad^k \Poly{0,k}(T) = \GRAD \Poly{k}(T) = \Goly{k-1}(T)$.
\end{remark}

\begin{remark} \label{rem:sequential}
These spaces are hierarchical since
$\Rolyb{c,k} \subset \Rolyb{c,k+1}, \Rolyb{k} \subset \Rolyb{k+1}$.
\end{remark}

We define a matrix valued equivalent to the Raviart-Thomas space as follows
\begin{equation} \label{eq:defRTb}
\RTb{k}(T) := \Rolyb{c,k}(T) \oplus \Rolyb{k-1}(T) \oplus (\Roly{k-1}(T)^\intercal)^3 .
\end{equation}

\begin{remark} \label{rem:qIinRT}
For $q \in \Poly{k}(T)$, we have $qI \in \Rolyb{k}(T) \oplus (\Roly{k}(T)^\intercal)^3$.
Indeed $\TDIV \left ( \InvDivGrad q - qI \right ) = 0$ so
$\InvDivGrad q - qI \in (\Roly{k}(T)^\intercal)^3$ by the isomorphism \eqref{eq:isodiv} and \eqref{eq:compkoszul}.
\end{remark}

\begin{lemma} \label{lemma:defGbc}
For $X \in \lbrace F,T \rbrace$, $\TDIV$ is an isomorphism from $\Rolyb{c,k+1}(X)$ to $\Goly{c,k}(X)$.
\end{lemma}
\begin{proof}
The proof for $X = F$ is given by Lemma \ref{lemma:isoRbG} for $X = T$ in the appendix.
The case $X = F$ is far easier and is provable with the same arguments.
\end{proof}

We will often need to view $2$-dimensional spaces as subspace of $\Real^3$. 
In particular we introduce two spaces related to the normal plane of an edge and to the tangent plane of a face.
\begin{definition} 
For any edge $E \in \Eh$ a natural $3$-dimensional vector is $\nE$.
We can arbitrarily complete it in an orthonormal basis of $\Real^3$
$(\nE,\bvec{n}_1,\bvec{n}_2)$. 
Assume that $\bvec{n}_1$ and $\bvec{n}_2$ are fixed once and for all on each edge.
We define the space 
\begin{equation} \label{eq:defEPoly}
\EPoly{k} = \lbrace p_1 \bvec{n}_1 + p_2 \bvec{n}_2 \st\; p_1, p_2 \in \Poly{k}(E) \rbrace .
\end{equation}

Likewise for any face $F \in \Fh$ assume there is a $F$-dependent fixed basis $(\nF,\bvec{n}_1,\bvec{n}_2)$ of $\Real^3$.
We define the space
\begin{equation} \label{eq:defFRTb}
\begin{aligned}
\FRTb{k} =&\; \left \lbrace \sum_{i,j = \lbrace 1,2 \rbrace} V_{i,j} \, \bvec{n}_i \otimes \bvec{n}_j + 
\sum_{i = \lbrace 1,2 \rbrace} w_i \, \nF \otimes \bvec{n}_i \st \right .\\
&\quad \left . \bvec{V} = (V)_{i,j} \in \Rolyb{c,k}(F) \oplus \Rolyb{k}(F) \oplus (\Roly{k}(F)^\intercal)^2 , \bvec{w} = (w)_i \in \bPoly{k}(F) \right \rbrace .
\end{aligned}
\end{equation}
\end{definition}
We will implicitly write $\FRTb{k} = (\Rolyb{c,k}(F) \oplus \Rolyb{k}(F) \oplus (\Roly{k}(F)^\intercal)^2) \oplus (\nF \otimes \bPoly{k}(F))$ 
to decompose it into its subcomponents. 
The last direct sum here is $L^2$-orthogonal, hence this will not cause any ambiguity in the scalar products.
The space $\FRTb{k}$ is isomorph to $M_{3,2}(\Poly{k}(F))$.
When embedded in the space of $3$ by $3$ tensor, elements of $\FRTb{k}$ are all orthogonal to $\nF$ on the right.

\section{Discrete complex.} \label{Discretecomplex}
We can now define the discrete complex. We start by giving the degree of freedom and the interpolator of the discrete spaces.
Then we define the discrete differential operators and give some basic properties on them.
To distinguish operators acting on scalar from those acting on vector we use the notation 
$\GRAD$ for the operator acting on scalar fields and giving vector fields, and $\TGRAD$ for the operator acting on vector fields and giving a tensor field.
\subsection{Complex definition.}
We define five discrete spaces $\uHgradh$, $\uHcurlh$, $\uHvh$, $\uLth$ and $\uLsh$.
Diagram \eqref{cd:Discretecomplex} summarizes their connection with each other and with their continuous counterpart.
Throughout the paper we will use the notations introduced in this section to refers to the components of discrete vectors.
\begin{equation} \label{cd:Discretecomplex}
\begin{tikzcd}  & & \bvec{L}^2(\Omega) \arrow[ddd, rounded corners, to path={ -- ([xshift=17ex]\tikztostart.east) 
 --node[right]{\scriptsize$\uILh$} ([xshift=17ex]\tikztotarget.east) 
 -- (\tikztotarget.east)}]& \\
H^2(\Omega) \arrow[r,"\GRAD"] \arrow[d,"\uIgradh"] & \bvec{H}^2 \arrow[r,"\CURL"] \arrow[d,"\uIcurlh"] & 
\bvec{H}^1(\Omega) \arrow[r,"\DIV"] \arrow[d,"\uIHh"] \arrow[u,"\TGRAD"] & 
L^2(\Omega) \arrow[d,"\uILsh"] \\
\uHgradh \arrow[r,"\uGh"] & \uHcurlh \arrow[r,"\uCh"] &
\uHvh \arrow[r,"\Dh"] \arrow[d,"\uNah"] &
\uLsh \\
&  & \uLth &
\end{tikzcd}
\end{equation}
Notice that the interpolators (defined in Section \ref{Interpolators}) require more smoothness than the spaces shown in \eqref{cd:Discretecomplex}.
Discrete spaces are defined by:
\begin{align}
\uHgradh :=& \lbrace \ul{q}_h = ((\gdofV{q})_{V \in \Vh}, (q_E,\gdofE{q})_{E \in \Eh}, (q_F, \gsdofF{q})_{F \in \Fh}, (q_T)_{T \in \Th}) 
\st \nonumber \\ 
&\quad
\begin{aligned}[t]
& \gdofV{q} \in \Real^3(V), \forall V \in \Vh,\\
&q_{E} \in \Poly[c]{k+1}(\Ech), 
\gdofE{q} \in \EPoly{k},
\forall E \in \Eh, \\
& q_F \in \Poly{k-1}(F), 
\gsdofF{q} \in \Poly{k-1}(F),
\forall F \in \Fh,\\
& q_T \in \Poly{k-1}(T), \forall T \in \Th
\rbrace , 
\end{aligned}\displaybreak[1]\\ 
\uHcurlh :=& \lbrace \uvec{v}_h = ((\rdofV{v})_{V \in \Vh},(\bvec{v}_E ,\rdofE{v})_{E \in \Eh}, (\rkvec[F]{v},\rkcvec[F]{v},v_F, \rdofFg{v}, \rdofFgc{v})_{F \in \Fh}, \\
&\ (\rkvec{v},\rkcvec{v})_{T \in \Th}) 
\st \nonumber
\begin{aligned}[t]
&\rdofV{v} \in \Real^3(V), \forall V \in \Vh,\\
&\bvec{v}_{E} \in \bPoly[c]{k+2}(\Ech), 
\rdofE{v} \in \bPoly{k+1}(E), \forall E \in \Eh, \\
& \rkvec[F]{v} \in \Roly{k-1}(F), 
\rkcvec[F]{v} \in \Roly{c,k}(F),
v_F \in \Poly{k-1}(F), \\
& \rdofFg{v} \in \Goly{k}(F), 
\rdofFgc{v} \in \Goly{c,k}(F), 
\forall F \in \Fh, \\
& \rkvec{v} \in \Roly{k-1}(T),
\rkcvec{v} \in \Roly{c,k}(T), \forall T \in \Th
\rbrace , 
\end{aligned}\displaybreak[1]\\
\uHvh :=& \lbrace \uvec{w}_h =  ((\bvec{w}_{E})_{E \in \Eh}, (w_F,\bgvec[F]{w},\bgcvec[F]{w})_{F \in \Fh}, (\bgvec{w},\bgcvec{w})_{T \in \Th})
\st \nonumber \\
& \quad
\begin{aligned}[t]
&\bvec{w}_{E} \in \bPoly[c]{k+3}(E), \forall E \in \Eh \\
&w_F \in \Poly{k}(F), \bgvec[F]{w} \in \Goly{k}(F), \bgcvec[F]{w} \in \Golyb{c,k}(F), \forall F \in \Fh,\\
&\bgvec{w} \in \Goly{k-1}(T), \bgcvec{w} \in \Golyb{c,k}(T), \forall T \in \Th 
\rbrace ,
\end{aligned}\displaybreak[1]\\
\uLth :=& \lbrace \uvec{W}_h = ((\bvec{W}_E)_{E \in \Eh}, (\bvec{W}_F)_{F \in \Fh}, (\bvec{W}_T)_{T \in \Th}) \st \nonumber \\
& \quad 
\begin{aligned}[t]
& \bvec{W}_E \in \bPoly{k+2}(E), \forall E \in \Eh,\\
& \bvec{W}_F \in \FRTb{k+1}, \forall F \in \Fh,
 \bvec{W}_T \in \RTb{k+1}(T), \forall T \in \Th
\rbrace , 
\end{aligned}\displaybreak[1]\\
\uLsh :=& \lbrace \ul{q}_h = ( (q_T)_{T \in \Th}) \st 
\begin{aligned}[t]
q_T \in \Poly{k}(T), \forall T \in \Th \rbrace .
\end{aligned}
\end{align}
Figure \ref{fig:tikzcddiff} summarizes the involvement of the various degrees of freedom with the differential operators.

\begin{figure}
\begin{tikzcd}[row sep=small]
T: &[-3em] \Poly{k-1}(T) \arrow[r,"\GRAD"] & \RT{k}(T) \arrow[r,"\CURL"] & \NE{k}(T) \arrow[r,"\DIV"] \arrow[dr,blue, "\TGRAD"] & \Poly{k}(T) \\
F: &[-3em] \Poly{k-1}(F) \arrow[r, "\GRAD"] & \RT{k}(F) \arrow[r, "\ROT"] & \Poly{k}(F) \arrow[dr,swap,blue,start anchor={[yshift=1ex]},"\GRAD"]& {\color{blue}\RTb{k+1}(T)}\\
& \Poly{k-1}(F) \arrow[r,"\GRAD^\perp"] \arrow[dr,"\text{Id}"]& |[alias=Y]| \Goly{k}(F) \oplus \Goly{c,k}(F)  & & {\color{blue} \FRTb{k+1}}\\
&  & \Poly{k-1}(F) \arrow[r,"\VROT"{name=UF}] & \Goly{k}(F) \oplus \Goly{c,k}(F) \arrow[ur,blue,start anchor={[yshift=-2ex,xshift=0.5em]},"\TGRAD"] \\
E:&[-3em] \EPoly{k} \arrow[r,"\GRAD^\perp"{name=Gt}] 
\arrow[dr,"\text{Id}"]
& |[alias=Z]| \bPoly{k+1}(E)  & & \\
 &  \Poly{k-1}(E) \arrow[r,end anchor={[yshift=-0.5ex]},pos=0.9,swap,"\GRAD"{name=G}] 
 &\begin{matrix}(\Poly{k}(E))^2 \\ \Poly{k}(E) \end{matrix} \arrow[r,"\CURL"{name=U}] & \bPoly{k+1}(E)  \arrow[r,blue,"\TGRAD"{name=GE}] & {\color{blue}\bPoly{k+2}(E)}\\
V:&[-3em] \Real = \Poly{k+1}(V) \arrow[dash, to=G, end anchor={[yshift=2.5ex,xshift=-0.5em]}] &\bPoly{k+2}(V) \arrow[dash, to=U] & \bPoly{k+3}(V) \arrow[blue, dash, to=GE]\\
& \Real^3 = \bPoly{k+2}(V) 
\arrow[ur, white, "\text{Id}"{name=GI}]
  \arrow[dash, crossing over, crossing over clearance=0.5ex,from=GI, to=Gt, start anchor={[xshift=-0.9em,yshift=-3.2ex]}, end anchor={[xshift=-0.6em,yshift=-0.1ex]}]
\arrow[ur, "\text{Id}"] & \bPoly{k+3}(V) \arrow[ur, "\text{Id}"]
 \arrow[dash, from=Z, to=U,bend left=10,end anchor={[yshift=-2.25ex,xshift=-0.35em]}, "\text{Id}"]
\arrow[dash, from=Y, to=UF, end anchor={[yshift=-2.15ex,xshift=-0.1em]}, "\text{Id}"]
\end{tikzcd}
\caption{Usage of the local degrees of freedom for the discrete differential operators.}
\label{fig:tikzcddiff}
\end{figure}
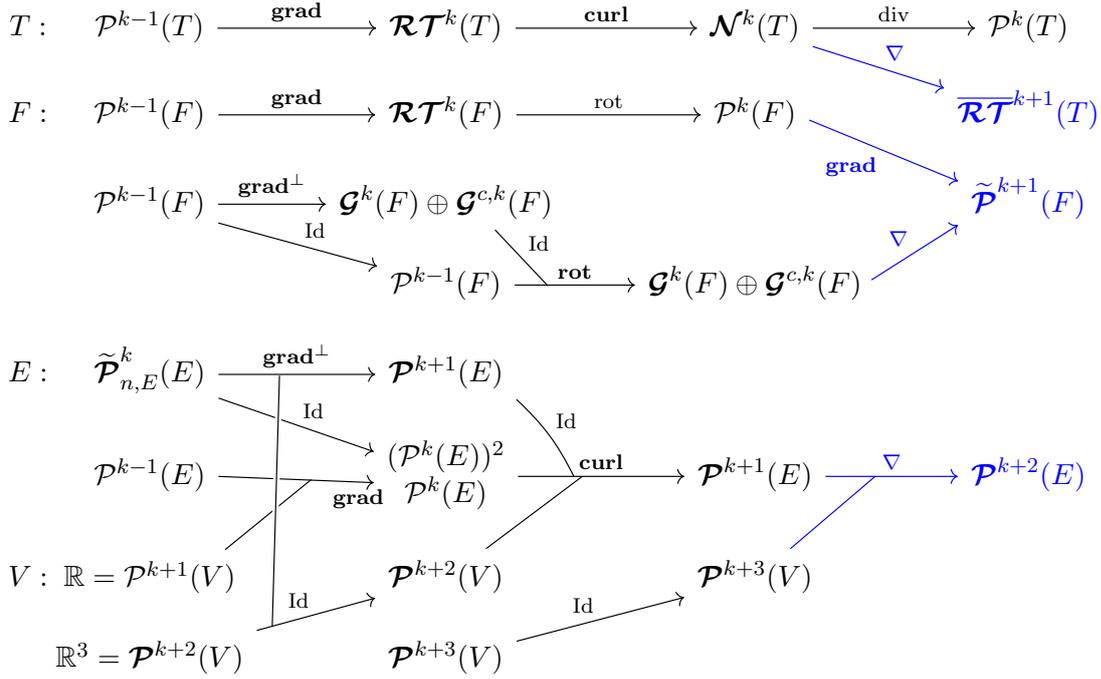

For a given cell $T$ we define the local discrete spaces $\uHgrad$, $\uHcurl$, $\uHv$, $\uLs$ and $\uLt$ as 
the restriction of the global one to $T$, i.e.\ containing only the components attached to $T$ and those attached to the faces, edges and vertices lying on its boundary. 
We define in the same way the local discrete spaces attached to a face $F$ or an edge $E$.

\subsection{Interpolators.} \label{Interpolators}
In this section we define the interpolator linking discrete spaces to their continuous counterpart.
Since we project on objects of lower dimension (edges and vertices) we will need a somewhat high smoothness for the continuous functions.
For a vertex $V \in \Vh$ we define $\bvec{x}_V \in \Real^3$ to be its coordinate.
The interpolator on the space $\uHgradh$ is defined for any $q \in C^1(\overline{\Omega})$ by
\begin{equation} \label{eq:defIgrad}
\begin{aligned}
\uIgradh q =&\; ((\GRAD{q}(\bvec{x}_V))_{V \in \Vh},(q_E, \vlproj{k}{E} (\nE \times (\GRAD{q} \times \nE)))_{E \in \Eh}, \\
&\;(\lproj{k-1}{F}(q),\lproj{k-1}{F}(\GRAD(q) \cdot \nF)_{F \in \Fh}, (\lproj{k-1}{T}(q))_{T \in \Th}), 
\end{aligned}
\end{equation}
where for any edge $E \in \Eh$, $q_E$ is such that $\lproj{k-1}{E}(q_E) = \lproj{k-1}{E}(q)$
and for any vertex $V \in \VE$, $q_{E} (\bvec{x}_V) = q(\bvec{x}_V)$.

The interpolator on the space $\uHcurlh$ is defined for any $\bvec{v} \in \bvec{C}^1(\overline{\Omega})$ by
\begin{equation} \label{eq:defIcurl}
\begin{aligned}
\uIcurlh{\bvec{v}} =&\; ((\CURL \bvec{v}(\bvec{x}_V))_{V \in \Vh},(\bvec{v}_E, \vlproj{k+1}{E}((\CURL \bvec{v} \cdot \nE ) \nE + \GRAD(\bvec{v} \cdot \nE) \times \nE))_{E \in \Eh},\\
&\; (\Rproj[F]{k-1}(\ttr{v}), \Rcproj[F]{k}(\ttr{v}), \lproj{k-1}{F}(\bvec{v} \cdot \nF),\\ 
&\;\Gproj[F]{k}(\nF \times (\TGRAD \bvec{v} \cdot \nF)),
\Gcproj[F]{k}(\nF \times (\TGRAD \bvec{v} \cdot \nF))
)_{F \in \Fh},\\
&\; (\Rproj{k-1}(\bvec{v}), \Rcproj{k}(\bvec{v}))_{T \in \Th}),
\end{aligned}
\end{equation}
where $\ttr{v}$ is the tangential trace of $\bvec{v}$ on $F$, 
and where for any edge $E \in \Eh$, $\bvec{v}_E$ is such that $\vlproj{k}{E}(\bvec{v}_E) = \vlproj{k}{E}(\bvec{v})$
and for any vertex $V \in \VE$, $\bvec{v}_{E} (\bvec{x}_V) = \bvec{v}(\bvec{x}_V)$.

The interpolator on the space $\uHvh$ is defined for any $\bvec{w} \in \bvec{C}^0(\overline{\Omega})$ by
\begin{equation} \label{eq:defIH}
\begin{aligned}
\uIHh{\bvec{w}} =&\; ((\bvec{w}_E)_{E \in \Eh}, (\lproj{k}{F}(\bvec{w} \cdot \nF), \Gproj[F]{k}(\ttr{w}),\Gbcproj[F]{k}(\ttr{w}))_{F \in \Fh},\\
&\; (\Gproj{k-1}(\bvec{w}),\Gbcproj{k}(\bvec{w}))_{T \in \Th}),
\end{aligned}
\end{equation}
where for any edge $E \in \Eh$, $\bvec{w}_E \in \bPoly{k+3}(E)$ is such that $\vlproj{k+1}{E}(\bvec{w}_E) = \vlproj{k+1}{E}(\bvec{w})$
and for any vertex $V \in \VE$, $\bvec{w}_{E}(\bvec{x}_V) = \bvec{w}(\bvec{x}_V)$.

The interpolator on the space $\uLth$ is defined for any $\bvec{W} \in \bvec{C}^0(\overline{\Omega})^3$ by
\begin{equation} \label{eq:defIL}
\uILh{\bvec{W}} = ((\vlproj{k+2}{E}(\bvec{W}\; \nE))_{E \in \Eh},(\FRTbproj{k+1}(\bvec{W}))_{F \in \Fh},(\RTbproj[T]{k+1}(\bvec{W}))_{T \in \Th}) .
\end{equation}

The interpolator on the space $\uLsh$ is defined for any $q \in L^2(\Omega)$ by 
\begin{equation} \label{eq:defILs}
\uILsh q = ( (\lproj{k}{T}(q))_{T \in \Th}) 
\end{equation}

\subsection{Discrete operators} \label{Discreteoperators}
\subsubsection{Gradient.}
In the following sections we define the discrete operators starting with the discrete gradient operator $\uGh$.
The operator $\uGh$ is the collection of the local discrete operators \eqref{eq:defuGT} acting on the edges, faces and cells.
For any edge $E \in \Eh$ we define the operator 
$\uGE : \uHgradE \rightarrow \uHcurlE$ such that
$\forall \ul{q}_{E} \in \uHgradE$
\begin{equation} \label{eq:defuGE}
\uGE \ul{q}_{E} = ((\bvec{0})_{V \in \VE},\bvec{v}_{E} , \dotp{\bvec{v}_E} \times \nE),
\end{equation}
where $\bvec{v}_E$ is such that 
$\vlproj{k}{E}(\bvec{v}_E) = \dotp{q_{E}} \nE + \gdofE{q}$ and $\forall V \in \VE$, 
$\bvec{v}_{E}(\bvec{x}_V) = \gdofV{q}$. 
We write $\dotp{q_{E}}$ the derivative of $q_E$ along the edge $E$ (oriented by $\nE$).

For any face $F \in \Fh$ we define the operator 
$\GF : \uHgradF \rightarrow \bPoly{k}(F)$ such that
$\forall \ul{q}_F \in \uHgradF$,
$\forall \bvec{w}_F \in \bPoly{k}(F)$
\begin{equation} \label{eq:defGF}
\int_F \GF \ul{q}_F \cdot \bvec{w}_F = - \int_F q_F \DIV \bvec{w}_F + \sum_{E \in \EF} \wFE \int_E q_{E} \bvec{w}_F \cdot \nFE .
\end{equation}
We also define 
$\GFp : \uHgradF \rightarrow \bPoly{k}(F)$ such that 
$\forall \ul{q}_F \in \uHgradF$,
$\forall \bvec{w}_F \in \bPoly{k}(F)$
\begin{equation} \label{eq:defGFp}
\int_F \GFp \ul{q}_F \cdot \bvec{w}_F = - \int_F \gsdofF{q} \ROT \bvec{w}_F - \sum_{E \in \EF} \wFE \int_E (\gdofE{q} \cdot \nF ) (\bvec{w}_F \cdot \nE) .
\end{equation}

The full operator $\uGF : \uHgradF \rightarrow \uHcurlF$
is defined to be the collection and projection of the local operators.
Explicitly for all $\ul{q}_h \in \uHgradF$
\begin{equation} \label{eq:defuGF}
\uGF \ul{q}_F = ((\uGE \ul{q}_{E})_{E \in \EF}, (\Rproj[F]{k-1}(\GF \ul{q}_F), \Rcproj[F]{k}(\GF \ul{q}_F),\gsdofF{q},\Gproj[F]{k}(\GFp \ul{q}_F),\Gcproj[F]{k}(\GFp \ul{q}_F))).
\end{equation}

The scalar trace $\trgrad : \uHgradF \rightarrow \Poly{k+1}(F)$ is defined such that $\forall \ul{q}_F \in \uHgradF$, 
$\forall \bvec{v}_F \in \Roly{c,k+2}(F)$,
\begin{equation} \label{eq:deftrgrad}
\int_F \trgrad \ul{q}_F \DIV \bvec{v}_F = - \int_F \GF \ul{q}_F \cdot \bvec{v}_F + \sum_{E \in \EF} \wFE \int_E q_E (\bvec{v}_F \cdot \nFE).
\end{equation}
\begin{remark} 
The relation \eqref{eq:deftrgrad} holds for all $\bvec{v}_F \in \Roly{k}(F) \oplus \Roly{c,k+2}(F)$. 
This is the same definition as for the discrete De Rham complex \cite{ddr}.
\end{remark}

\begin{remark}
We can also define a tangential trace in the same manner as \cite{ddr}.
It is not required thanks to the choice of norm \eqref{eq:defopnG} 
but should be considered to show consistency results for $\uHgradh$.
\end{remark}

For any cell $T \in \Th$ we define the operator 
$\GT : \uHgrad \rightarrow \bPoly{k}(T)$ such that
$\forall \ul{q}_T \in \uHgrad$,
$\forall \bvec{w}_T \in \bPoly{k}(T)$
\begin{equation} \label{eq:defGT}
\int_T \GT \ul{q}_T \cdot \bvec{w}_T = - \int q_T \DIV \bvec{w}_T + \sum_{F \in FT} \wTF \int_F \trgrad \ul{q}_{F} \bvec{w}_T \cdot \nF .
\end{equation}
Likewise we define the full operator $\uGT : \uHgrad \rightarrow \uHcurl$ for all $\ul{q}_T \in \uHgrad$ by
\begin{equation} \label{eq:defuGT}
\uGT \ul{q}_T = ((\uGE \ul{q}_E)_{E \in \ET}, (\uGF \ul{q}_F)_{F \in \FT}, (\Rproj{k-1}(\GT \ul{q}_T), \Rcproj{k}(\GT \ul{q}_T)).
\end{equation}

The global operator $\uGh$ is obtained by gathering the local operators $\uGT, T \in \Th$.
Since the interpolators require taking the full gradient even on edges we must consider functions to be defined in a small neighborhood 
$E \subset \mathcal{O}_E$ open in $\Real^3$. We also define $\bvec{X}$ to be the polynomial $\bvec{X}(\bvec{x}) = \bvec{x}$.
\begin{lemma}[Consistency properties] 
The discrete gradients and trace satisfy the following consistency properties for all $E \in \Eh$, $F \in \Fh$ and $T \in \Th$:
\begin{alignat}{3}
\uGE(\uIgrad[E] q) =&\; \uIcurl[E] (\GRAD q) \quad & \forall q &\in C^2(\mathcal{O}_E) \label{eq:consuGE} \\
\GF(\uIgrad[F] q) =&\; \GRAD q & \forall q &\in \Poly{k+1}(F) \label{eq:consGF} \\
\GFp(\uIgrad[F] (q \; \bvec{X}\cdot \nF )) =&\; -\VROT q & \forall q &\in \Poly{k+1}(F) \label{eq:consGFp} \\
\trgrad (\uIgrad[F] q) =&\; q & \forall q &\in \Poly{k+1}(F) \label{eq:constrgrad}\\
\lproj{k-1}{F} (\trgrad \ul{q}_F) =&\; q_F & \forall \ul{q}_F &\in \uHgradF \label{eq:conspitrgrad}\\
\GT(\uIgrad q) =&\; \GRAD q & \forall q &\in \Poly{k+1}(T) \label{eq:consGT}
\end{alignat}
\end{lemma}
\begin{proof}
Properties \eqref{eq:consGF}, \eqref{eq:constrgrad}, \eqref{eq:conspitrgrad} and \eqref{eq:consGT} 
are proven in \cite[Lemma~3.3]{ddr}.

\underline{Proof of \eqref{eq:consuGE}.}
The idea is to use integration by parts together with the continuity on vertices to remove the projection.
Let $E \in \Eh$, $q \in C^2(\mathcal{O}_E)$, $\uGE (\uIgrad[E] q) := ((\bvec{0})_{V \in \VE},\bvec{v}_E, \dotp{\bvec{v}_E} \times \nE$,
$q_E$ as in \eqref{eq:defIgrad}
and $\bvec{w}_E \in \bPoly{k+2}(E)$ such that $\vlproj{k}{E}(\bvec{w}_E) = \vlproj{k}{E}(\GRAD q)$ and 
$\bvec{w}_E(\bvec{x}_V) = \GRAD q (\bvec{x}_V)$ for $V \in \VE$.
We have $\bvec{0} = \CURL \GRAD q (\bvec{x}_V)$, $\forall V \in \VE$.
We must show that $\bvec{v}_E = \bvec{w}_E$ and that 
$\dotp{\bvec{v}_E} \times \nE = \vlproj{k+1}{E}( ((\CURL \GRAD q ) \cdot \nE) \nE + (\GRAD (\GRAD q \cdot \nE) \times \nE)) 
= \vlproj{k+1}{E}( (\GRAD (\GRAD q \cdot \nE) \times \nE))$.
Take a standard basis $(\bvec{e}_0,\bvec{e}_1,\bvec{e}_2)$ such that $\nE = (1,0,0)$ and $\bvec{v}_E := (v_0,v_1,v_2)$.
For all $\bvec{r} \in \bPoly{k}(E)$, 
\begin{equation*}
\begin{aligned}
\int_E \bvec{v}_E \cdot \bvec{r} =&\; \int_E \dotp{q_E} \, r_0 + \int_E \lproj{k}{E} (\partial_1 q) \, r_1 + \lproj{k}{E} (\partial_2 q) \, r_2 \\
=&\; - \int_E q_E \dotp{r_0} + [q_E \, r_0] + \int_E \lproj{k}{E} (\partial_1 q) \, r_1 + \lproj{k}{E} (\partial_2 q) \, r_2 \\
=&\; - \int_E q \dotp{r_0} + [q \, r_0] + \int_E \partial_1 q \, r_1 + \partial_2 q \, r_2 \\
=&\; \int_E \GRAD q \cdot \bvec{r}.
\end{aligned}
\end{equation*}
We use that $q_E(\bvec{x}_V) = q(\bvec{x}_V)$, $V \in \VE$ to get the third line.
Hence, $\vlproj{k}{E}(\bvec{v}_E) = \vlproj{k}{E}(\GRAD q)$ and 
since $\bvec{v}_E(\bvec{x}_V) = \GRAD q (\bvec{x}_V)$ we have $\bvec{v}_E = \bvec{w}_E$.
We conclude with the same argument for $\bvec{r} \in \bPoly{k+1}(E)$ applied to $\dotp{\bvec{v}_E} \times \nE$.

\underline{Proof of \eqref{eq:consGFp}.}
Let $q \in \Poly{k+1}(F)$. 
Since $q$ does not depend on the coordinate in the $\nF$ direction, we have $\GRAD (q \; \bvec{X}\cdot \nF) \cdot \nF = q$.
The relevant parts of $\uIgrad[F](q \; \bvec{X}\cdot \nF )$ are 
$\gsdofF{q \; \bvec{X}\cdot \nF} = \lproj{k-1}{F}( \GRAD (q \; \bvec{X}\cdot \nF) \cdot \nF) = \lproj{k-1}{F}(q)$ and 
$\gdofE{q \; \bvec{X}\cdot \nF} = \lproj{k}{F}( \GRAD (q \; \bvec{X}\cdot \nF) - (\GRAD (q \; \bvec{X}\cdot \nF) \cdot \nE ) \nE)$. 
Hence we have $\gdofE{q \; \bvec{X}\cdot \nF} \cdot \nF = \lproj{k}{E}(q)$ and 
for all $\bvec{w} \in \bPoly{k}(F)$,
\begin{equation*}
\begin{aligned}
\int_F \GFp (\uIgrad[F](q \; \bvec{X}\cdot \nF ) ) \cdot \bvec{w} =&\; 
- \int_F \lproj{k-1}{F}(q) \ROT \bvec{w} - \sum_{E \in \EF} \wFE \int_E \lproj{k}{E}(q)\, \bvec{w} \cdot \nE \\
=&\; - \int_F q \ROT \bvec{w} - \sum_{E \in \EF} \wFE \int_E q \, \bvec{w} \cdot \nE \\
=&\; - \int_F \VROT q \cdot \bvec{w} .
\end{aligned}
\end{equation*}

\end{proof}

\subsubsection{Curl.}
The operator $\uCh$ is the collection of the local discrete operators \eqref{eq:defuCT} acting on the edges and faces.
For any edge $E \in \Eh$ we define the operator 
$\CE : \uHcurlE \rightarrow \bPoly{k+3}(E)$ such that
$\forall \uvec{v}_{E} = ((\rdofV{v})_{V \in \VE},\bvec{v}_{E}, \rdofE{v}) \in \uHcurlE$
\begin{equation} \label{eq:defCE}
\vlproj{k+1}{E} (\CE \uvec{v}_{E}) = \rdofE{v} - \dotp{\bvec{v}_E} \times \nE , \quad \CE \uvec{v}_E (\bvec{x}_V) = \rdofV{v} .
\end{equation}

For any face $F \in \Fh$ we define the operators
$\CF : \uHcurlF \rightarrow \Poly{k}(F)$ and
$\bCF : \uHcurlF \rightarrow \bPoly{k}(F)$ for all
$\uvec{v}_F = ((\bvec{v}_E ,\rdofE{v})_{E \in \Eh}, (\rkvec[F]{v},\rkcvec[F]{v},v_F, \rdofFg{v}, \rdofFgc{v})) \in \uHcurlF$
such that 
$\forall r_F \in \Poly{k}(F)$
\begin{equation} \label{eq:defCF}
\int_F \CF \uvec{v}_F r_F = \int_F \rkvec[F]{v} \cdot \VROT r_F - \sum_{E \in \EF} \wFE \int_E (\bvec{v}_E \cdot \nE)\, r_F,
\end{equation}
and 
$\forall \bvec{r}_F \in \bPoly{k}(F)$
\begin{equation} \label{eq:defbCF}
\int_F \bCF \uvec{v}_F \cdot \bvec{r}_F = \int_F v_F \ROT \bvec{r}_F + \sum_{E \in EF} \wFE \int_E (\bvec{v}_{E} \cdot \nF) (\bvec{r}_F \cdot \nE) .
\end{equation}

The full operator $\uCF : \uHcurlF \rightarrow \uHvF$
is defined as the collection and projection of the local operators.
Explicitly for all $\uvec{v}_F \in \uHcurlF$
\begin{equation} \label{eq:defuCF}
\uCF \uvec{v}_F = ((\CE \uvec{v}_{E})_{E \in \EF}, (\CF \uvec{v}_F, \Gproj[F]{k}(\bCF \uvec{v}_F) + \rdofFg{v}, \Gbcproj[F]{k}(\bCF \uvec{v}_F) + \rdofFgc{v})).
\end{equation}

\begin{lemma}[Local complex property] \label{lemma:complexCFGF}
For all $F \in \Fh$ it holds:
\begin{equation}
\Image \uGF \subset \Ker \uCF .
\end{equation}
\end{lemma}
\begin{proof}
Let $\ul{q}_F \in \uHgradF$, we have to show that $\uCF (\uGF \ul{q}_F) = 0$.
We define $\uvec{v}_F = \uGF \ul{q}_F$.
It is immediate to check for the edges since $\vlproj{k+1}{E}(\CE \uvec{v}_E) = \rdofE{v} - \dotp{\bvec{v}_E} \times \nE$ with
$\rdofE{v} = \dotp{\bvec{v}_E} \times \nE$ and $\CE \uvec{v}_E (\bvec{x}_V) = \bvec{0}$.
Next in order to prove that $\CF \uvec{v}_F = 0$ take any $r_F \in \Poly{k}(F)$, then 
\begin{equation*}
\begin{aligned}
\int_F \CF \uvec{v}_F r_F =&\; \int_F \Rproj[F]{k-1}(\GF \ul{q}_F) \cdot \VROT r_F - \sum_{E \in \EF} \wFE \int_E \bvec{v}_E \cdot \nE \, r_F \\
=&\; \int_F \GF \ul{q}_F \cdot \VROT r_F - \sum_{E \in \EF} \wFE \int_E \dotp{q_E} \, r_F\\
=&\; -\int_F q_F \cdot \DIV(\VROT r_F) + \sum_{E \in \EF} \wFE \int_E q_E \VROT r_F \cdot \nFE - \dotp{q_E} \, r_F\\
=&\; \sum_{E \in \EF} \wFE \int_E q_E \, \dotp{r_F} - \dotp{q_E} \, r_F\\
=&\; \sum_{E \in \EF} \wFE [q_E \, r_F].
\end{aligned}
\end{equation*}
The last term is null since we integrate over a closed loop.
It remains to prove $\bCF \uvec{v}_F = 0$.
For any $\bvec{r}_F$ successively in $\Goly{k}(F)$ and $\Goly{c,k}(F)$ it is immediate to check that 
$\int_F \bCF (\uGF \ul{q}_F) \cdot \bvec{r}_F = 0$ since \eqref{eq:defbCF} and \eqref{eq:defGFp} are opposite.
\end{proof}

We define the tangential trace $\trcurl : \uHcurlF \rightarrow \Poly{k}(F)$ such that 
$\forall (r_F, \bvec{w}_F) \in \Poly{0,k+1}(F) \times \Roly{c,k}(F)$,
\begin{equation}
\int_F \trcurl \uvec{v}_F \cdot (\VROT r_F + \bvec{w}_F) = \int_F \CF \uvec{v}_F r_F + \sum_{E \in \EF} \wFE \int_E (\bvec{v}_E \cdot \nE) r_F + \int_F \rkcvec[F]{v} \cdot \bvec{w}_F .
\end{equation}
This is almost the same definition as in \cite{ddr}. As such we have almost the same properties.
\begin{lemma}[Properties of the tangential trace] \label{lemma:proptrcurl}
It holds
\begin{alignat}{3}
\Rproj[F]{k-1}(\trcurl \uvec{v}_F) = \rkvec[F]{v} \; &\text{and} \; 
\Rcproj[F]{k}(\trcurl \uvec{v}_F) = \rkcvec[F]{v} \quad & \forall \uvec{v}_F \in \uHcurlF,& \label{eq:trcurlXcurl}\\
\trcurl (\uIcurl[F] \bvec{v}) &= \vlproj{k}{F} \bvec{v} \quad & \forall \bvec{v} \in \NE{k+1}(F),& \label{eq:trcurlIcurl}\\
\RTproj[F]{k} (\trcurl (\uGF \ul{q}_F) )&= \RTproj[F]{k} (\GF \ul{q}_F) \quad & \forall \ul{q}_F \in \uHgradF.& \label{eq:trcurlGF}
\end{alignat}
\end{lemma}
\begin{proof}
The proof of \cite[Proposition~3.3]{ddr} almost works here,
the sole difference being the continuity of $\bvec{v}_E \cdot \nE$ in the boundary term:
we will only have $\lproj{k}{E} (\bvec{v}_E \cdot \nE) = \dotp{q_E}$ instead of $\bvec{v}_E \cdot \nE = \dotp{q_E}$.
This only affects the proof of \eqref{eq:trcurlGF}, for which we must restrict ourselves to test functions in 
$\Poly{0,k}(F) \times \Roly{c,k}(F)$ instead of 
$\Poly{0,k+1}(F) \times \Roly{c,k}(F)$.
This explains the addition of $\RTproj[F]{k}$ since $\VROT \Poly{0,k}(F) \oplus \Roly{c,k}(F) \approx \RT{k}(F)$.
\end{proof}

For any $T \in \Th$ we define the operator $\CT : \uHcurl \rightarrow \bPoly{k}(T)$ for all\\
$\uvec{v}_T = ((\uvec{v}_F)_{F \in \FT}, \rkvec{v},\rkcvec{v}) \in \uHcurl$ such that 
$\forall \bvec{r}_T \in \bPoly{k}(T)$, 
\begin{equation} \label{eq:defCT}
\int_T \CT \uvec{v}_T \cdot \bvec{r}_T = \int_T \rkvec{v} \cdot \CURL \bvec{r}_T + \sum_{F \in \FT} \wTF \int_F \trcurl \uvec{v}_F \cdot (\bvec{r}_T \times \nF) .
\end{equation}
The full operator $\uCT : \uHcurl \rightarrow \uHv$ is such that, for all $\uvec{v}_T \in \uHcurl$,
\begin{equation} \label{eq:defuCT}
\uCT \uvec{v}_T := ((\CE \uvec{v}_E)_{E \in \ET}, (\uCF \uvec{v}_F)_{F \in \FT}, \Gproj{k-1} (\CT \uvec{v}_T), \Gcproj{k} (\CT \uvec{v}_T)).
\end{equation}

\subsubsection{Jacobian.}
The operator $\uNah$ is the collection of the local discrete operators \eqref{eq:defuNaT}.
For all edge $E \in \Fh$ we define the operator 
$\NaE : \uHvE \rightarrow \bPoly{k+2}(E)$ for all $\bvec{w}_E \in \uHvE$ by
\begin{equation} \label{eq:defNaE}
\NaE \bvec{w}_E = \dotp{\bvec{w}_E} .
\end{equation}

For all face $F \in \Fh$ we define the operator  
$\NaF : \uHvF \rightarrow \FRTb{k+1}$ for all\\ 
$\uvec{w}_F = ((\bvec{w}_E)_{E \in \Eh}, w_F, \bgvec[F]{w}, \bgcvec[F]{w}) \in \uHvF$ 
such that
$\forall \bvec{V}_F = \FRTnvec{V} + \brcvec[F]{V} + \brvec[F]{V} + \rkvec[F]{V} \in \FRTb{k+1}$,
\begin{equation} \label{eq:defNaF}
\begin{aligned}
\int_F \NaF (\uvec{w}_F) \tdot \bvec{V}_F =& - \int_F \bgcvec[F]{w} \cdot \TDIV (\brcvec[F]{V}) 
- \int_F \bgvec[F]{w} \cdot \TDIV (\brvec[F]{V})
- \int_F w_F \DIV (\FRTnvec{V}) \\
&\quad + \sum_{E \in \EF} \wFE \int_E \bvec{w}_{E} \bvec{V}_F \nFE .
\end{aligned}
\end{equation}
We define the full operator $\uNaF : \uHvF \rightarrow \uLsF$ by 
\begin{equation} \label{eq:defuNaF}
\uNaF \uvec{w}_F = ((\NaE \bvec{w}_E)_{E \in \EF}, \NaF \uvec{w}_F) .
\end{equation}

We prove a first commutative property:
\begin{lemma} \label{lemma:NaFuIH}
For all $F \in \Fh$ it holds:
\begin{alignat}{3}
\NaF (\uIH[F] \bvec{w}) &= \FRTbproj{k+1}( \TGRAD \bvec{w}) , \quad \quad
&&\forall \bvec{w} \in \bvec{C}^1(\overline{F}). \label{eq:consNaF}
\end{alignat}
\end{lemma}
\begin{proof}
For all $\bvec{w} \in \bvec{C}^1(\overline{F})$
and all $\bvec{V}_F \in \FRTb{k+1}$, 
\begin{equation*}
\begin{aligned}
\int_F \NaF ( \uIH[F] \bvec{w}) \tdot \bvec{V}_F =& - \int_F \Gcproj[F]{k}(\ttr{w}) \cdot \TDIV (\brcvec[F]{V}) 
- \int_F \Gproj[F]{k}({\ttr{w}}) \cdot \TDIV (\brvec[F]{V})\\
&\quad - \int_F \lproj{k}{F} (\bvec{w} \cdot \nF) \DIV (\FRTnvec{V}) 
+ \sum_{E \in \EF} \wFE \int_E \vlproj{k+1}{E} (\bvec{w}) \bvec{V}_F \nFE\\
=& - \int_F \ttr{w} \cdot \TDIV (\brcvec[F]{V} + \brvec[F]{V}) 
- \int_F (\bvec{w} \cdot \nF) \cdot \DIV (\FRTnvec{V}))\\
&\quad + \sum_{E \in \EF} \wFE \int_E \bvec{w} \bvec{V}_F \nFE\\
=& - \int_F \bvec{w} \cdot \TDIV (\bvec{V}_F) 
+ \sum_{E \in \EF} \wFE \int_E \bvec{w} \bvec{V}_F \nFE\\
=& \int_F \TGRAD \bvec{w} \tdot \bvec{V}_F .
\end{aligned}
\end{equation*}
We used Lemma \ref{lemma:defGbc} and the definition \eqref{eq:defRb} to remove the first two projections ($\Gcproj[F]{k}$ and $\Gproj[F]{k}$)
and integration by parts to conclude.
\end{proof}

We define the trace operator
$\trna : \uHvF \rightarrow (\Poly{k+2}(F))^3$ 
by the relation:
$\forall \bvec{V}_F \in (\Roly{c,k+3}(F)^\intercal)^3$,
$\forall \uvec{w}_F \in \uHvF$,
\begin{equation} \label{eq:deftrna}
\int_F \trna(\uvec{w}_F) \cdot \TDIV \bvec{V}_F = - \int_F \NaF \uvec{w}_F \tdot \bvec{V}_F + \sum_{E \in \EF} \wFE \int_E \bvec{w}_{E} \bvec{V}_F \nFE .
\end{equation}
The isomorphism \eqref{eq:isodiv} ensures the well-posedness.
\begin{remark} \label{rem:validitytrna}
The relation \eqref{eq:deftrna} also holds for all $\bvec{V}_F \in (\bPoly{k+1}(F)^\intercal)^3$.
Indeed if $\bvec{V}_F$ belongs to $(\Roly{k+1}(F)^\intercal)^3$ then $\TDIV \bvec{V}_F = 0$ and the left-hand side of \eqref{eq:deftrna} is null.
And since $(\Roly{k+1}(F)^\intercal)^3 \subset (\bPoly{k+1}(F)^\intercal)^3 \approx \FRTb{k+1}$ we can apply \eqref{eq:defNaF} to show that the right-hand side is also zero.
Hence, the relation holds for all $(\Roly{k+1}(F)^\intercal)^3 \oplus (\Roly{c,k+3}(F)^\intercal)^3 \supset (\bPoly{k+1}(F)^\intercal)^3$.
\end{remark}

\begin{lemma}[Consistency properties]
For all $F \in \Fh$ the following relations hold:
\begin{alignat}{3}
\trna(\uIH[F] \bvec{w}) = \bvec{w},&
&&\forall \bvec{w} \in \bPoly{k+2}(F),
\label{eq:trnaIH} \\
\begin{matrix}
  \lproj{k}{F}(\trna \uvec{w}_F \cdot \nF) = w_F,\\
  \Gcproj[F]{k}(\trna \uvec{w}_F) = \bgcvec[F]{w},\\
  \Gproj[F]{k}(\trna \uvec{w}_F) = \bgvec[F]{w},
\end{matrix}& \quad \quad
&&\forall \uvec{w}_F \in \uHvF. 
\label{eq:pitrna}
\end{alignat}
\end{lemma}
\begin{proof}
\underline{Proof of \eqref{eq:pitrna}.}
Let $\uvec{w}_F \in \uHvF$ and $\bvec{v}_F \in \Roly{c,k+1}(F)$.
Remark \ref{rem:validitytrna} allows to write:
\begin{equation*}
\begin{aligned}
\int_F \trna (\uvec{w}_F) \cdot \nF \, \DIV (\bvec{v}_F)
=&\, \int_F \trna (\uvec{w}_F) \cdot \nabla \cdot (\nF \otimes \bvec{v}_F)\\
=&\, - \int_F \NaF \uvec{w}_F \tdot (\nF \otimes \bvec{v}_F)
+ \sum_{E \in \EF} \wFE \int_E (\bvec{w}_E \cdot \nF) (\bvec{v}_F \cdot \nFE) \\
=&\, \int_F w_F \DIV (\bvec{v}_F).
\end{aligned}
\end{equation*}
Since this holds for all $\bvec{v}_F \in \Roly{c,k+1}(F)$, inferring isomorphism \eqref{eq:isodiv}
we have:
\begin{equation*}
\lproj{k}{F}(\trna \uvec{w}_F \cdot \nF) = 
\lproj{k}{F}(w_F) = w_F.
\end{equation*}
The other two equations are proven in the same fashion.\\
\underline{Proof of \eqref{eq:trnaIH}.}
Let $\bvec{w} \in \bPoly{k+2}(F)$ and $\bvec{V}_F \in (\Roly{c,k+3}(F)^\intercal)^3$, it holds
\begin{equation*}
\begin{aligned}
\int_F \trna(\uIH[F] \bvec{w}) \cdot \TDIV \bvec{V}_F 
=&\; - \int_F \NaF \uIH[F] \bvec{w} \tdot \bvec{V}_F + \sum_{E \in \EF} \wFE \int_E (\bvec{w}_E) \bvec{V}_F \nFE \\
=&\; - \int_F \TGRAD \bvec{w} \tdot \bvec{V}_F + \sum_{E \in \EF} \wFE \int_E \bvec{w} \bvec{V}_F \nFE \\
=&\; \int_F \bvec{w} \cdot \TDIV \bvec{V}_F .
\end{aligned}
\end{equation*}
We used Lemma \ref{lemma:NaFuIH} to write $\NaF \uIH[F] \bvec{w} = \FRTbproj{k+1}( \TGRAD \bvec{w}) = \TGRAD \bvec{w}$ 
since $\TGRAD \bvec{w} \in (\bPoly{k+1}(F)^\intercal)^3$,
and $\bvec{w}_E = \bvec{w}$ since $\bvec{w} \in \bPoly{k+2} \subset \bPoly{k+3}$ is continuous 
to remove the projections.
\end{proof}

For all $T \in \Th$ we define the operator
$\NaT : \uHv \rightarrow \RTb{k+1}(T)$ 
such that 
$\forall \uvec{w}_T = ((\uvec{w}_F)_{F \in \FT}, \bgvec{w}, \bgcvec{w}) \in \uHv$, 
$\forall \bvec{V}_T = \brcvec{V} + \brvec{V} + \rkvec{V} \in \RTb{k+1}(T)$,
\begin{equation} \label{eq:defNaT}
\begin{aligned}
\int_T \NaT (\uvec{w}_T) \tdot \bvec{V}_T =& - \int_T \bgcvec{w} \cdot \TDIV (\brcvec{V}) 
- \int_T \bgvec{w} \cdot \TDIV (\brvec{V})
 + \sum_{F \in \FT} \wTF \int_F \trna \uvec{w}_F \bvec{V}_T \nF.
\end{aligned}
\end{equation}
We also define the potential reconstruction operator
$\pna : \uHv \rightarrow \bPoly{k+1}(T)$ 
by the relation:
$\forall \bvec{V}_T \in (\Roly{c,k+2}(T)^\intercal)^3$,
$\forall \uvec{w}_T \in \uHv$,
\begin{equation} \label{eq:defpna}
\int_T \pna(\uvec{w}_T) \cdot \TDIV \bvec{V}_T = - \int_T \NaT \uvec{w}_T \tdot \bvec{V}_T + \sum_{F \in \FT} \wTF \int_F \trna \uvec{w}_F \bvec{V}_T \nF .
\end{equation}
The global operator $\uNaT : \uHv \rightarrow \uLt$ is defined for all $\uvec{w}_T \in \uHv$ by 
\begin{equation} \label{eq:defuNaT}
\uNaT \uvec{w}_T = ((\uNaF \uvec{w}_F)_{F \in \FT}, \NaT \uvec{w}_T) .
\end{equation}
\begin{remark} \label{rem:NaTqI}
Since $\TDIV (\Roly{k}(T)^\intercal)^3 = 0$ by Remark \ref{rem:qIinRT} and \eqref{eq:pitrna} we see that $\forall q \in \Poly{k}(T)$, 
\begin{equation*}
\int_T \NaT (\uvec{w}_T) \tdot (q \ID) =  \int_T \bgvec{w} \cdot \GRAD q + \sum_{F \in \FT} \wTF \int_F q \; w_F .
\end{equation*}
\end{remark}

\begin{remark} \label{rem:validitypna}
With the same argument as in Remark \ref{rem:validitytrna}, \eqref{eq:defpna} is valid for all $\bvec{V}_T \in (\bPoly{k+1}(T)^\intercal)^3$.
\end{remark}

\begin{lemma}
For all $T \in \Th$ it holds:
\begin{alignat}{3}
\NaT (\uIH \bvec{w}) = \TGRAD \bvec{w},& \quad \quad
&&\forall \bvec{w} \in \bPoly{k+1}(T), \label{eq:consNaT}\\
\pna(\uIH \bvec{w}) = \bvec{w},&
&&\forall \bvec{w} \in \bPoly{k+1}(T),
\label{eq:pnaIH} \\
\begin{matrix}
  \Gcproj{k}(\pna \uvec{w}_T) = \bgcvec{w},\\
  \Gproj{k-1}(\pna \uvec{w}_T) = \bgvec{w},
\end{matrix}& \quad \quad
&&\forall \uvec{w}_T \in \uHv. 
\label{eq:pipna}
\end{alignat}
\end{lemma}
\begin{proof}
The proof of \eqref{eq:consNaT} is similar to the proof of Lemma \ref{lemma:NaFuIH}, using \eqref{eq:trnaIH} to remove the projection on the boundary term.
The proofs of \eqref{eq:pnaIH} and \eqref{eq:pipna} are the same as the proofs of \eqref{eq:trnaIH} and \eqref{eq:pitrna}.
\end{proof}

\subsubsection{Divergence.}
Finally, we define the discrete divergence operator, for all $T \in \Th$ by:
\[\DT := \Tr \NaT \in \Poly{k}(T) .
\]
As in the continuous case the divergence is the trace of the gradient,
but we can also define it by a formula mimicking the integration by parts.
By Remark \ref{rem:NaTqI}, $\forall \uvec{w}_T \in \uHv$, $\DT$ is such that $\forall q_T \in \Poly{k}(T)$,
\begin{equation}
\begin{aligned}
\int_T \DT \uvec{w}_T q_T =& \int_T \Tr (\NaT \uvec{w}_T) q_T \\
=& \int_T \NaT \uvec{w}_T \tdot (q_T \ID)\\
=& - \int_T \bgvec{w} \cdot \GRAD q_T + \sum_{F \in \FT} \wTF \int_F (\trna \uvec{w}_{F} \cdot \nF) q_T\\
=& - \int_T \bgvec{w} \cdot \GRAD q_T + \sum_{F \in \FT} \wTF \int_F w_F q_T.
\end{aligned}
\end{equation}
We get the same definition as the one of the de Rham complex of \cite{ddr}.

\subsection{Discrete $L^2$-product.}
We build scalar products on the discrete spaces.
They are made of the sum of the $L^2$ scalar product on each cell and of a stabilization term taking the lower dimensional objects (edges, vertices and faces) into account.
Since we will not need their definitions on $\uHgrad$ and $\uHcurl$ to study the Jacobian operator,
we will not write them down explicitly.
They are quite similar to the $L^2$ product of $\uHv$ but require the introduction of potential reconstruction operators 
akin to those of \cite{ddr}.
First we define them locally for all $T \in \Th$:
For all $\uvec{v}_T, \uvec{w}_T \in \uHv$ we set
\begin{equation} \label{eq:defspNa}
\spNa{\uvec{v}_T}{\uvec{w}_T} = \int_T \pna \uvec{v}_T \cdot \pna \uvec{w}_T + \stNa{\uvec{v}_T}{\uvec{w}_T} ,
\end{equation}
\begin{equation} \label{eq:defstNa}
\begin{aligned}
\stNa{\uvec{v}_T}{\uvec{w}_T} =&\; \sum_{F \in \FT} h_F \int_F (\pna \uvec{v}_T - \trna \uvec{v}_F) \cdot (\pna \uvec{w}_T - \trna \uvec{v}_F) \\
&\; + \sum_{E \in \EF} h_E^2 \int_E (\pna \uvec{v}_T - \bvec{v}_E) \cdot (\pna \uvec{w}_T - \bvec{w}_E).
\end{aligned}
\end{equation}
For all $\uvec{V}_T, \uvec{W}_T \in \uLt$ we set
\begin{equation} \label{eq:defspLt}
\spLt{\uvec{V}_T}{\uvec{W}_T} = \int_T \bvec{V}_T \tdot \bvec{W}_T + \stLt{\uvec{V}_T}{\uvec{W}_T},
\end{equation}
\begin{equation} \label{eq:defstLt}
\begin{aligned}
\stLt{\uvec{V}_T}{\uvec{W}_T} =&\; \sum_{F \in \FT} h_F \int_F (\Tttr{V} - \bvec{V}_F) \tdot (\Tttr{W} - \bvec{W}_F) \\
&\; + \sum_{E \in \EF} h_E^2 \int_E (\bvec{V}_T \; \nE - \bvec{V}_E) \cdot (\bvec{W}_T \; \nE - \bvec{W}_E),
\end{aligned}
\end{equation}
where $\Tttr{V} := \bvec{V}_T - (\bvec{V}_T \; \nF) \otimes \nF$.
Global scalar products are then merely the sum of local scalar product over every face $T \in \Th$.
For all $\uvec{v}_T \in \uHv$ and $\uvec{W}_T \in \uLt$ the norm induced by this scalar product is denoted by:
\begin{equation*}
\normNa{\uvec{v}_T} = \spNa{\uvec{v}_T}{\uvec{v}_T}^{1/2}, \quad \normLt{\uvec{W}_T} = \spLt{\uvec{W}_T}{\uvec{W}_F}^{1/2}.
\end{equation*}

We also define norms built from the sum over the objects of every dimension.
For all $\ul{q}_T \in \uHgrad$ we define
\begin{equation} \label{eq:defopnG}
\begin{aligned}
\opnHgrad{ \ul{q}_T}^2 =&\; \norm[T]{q_F}^2 + \sum_{F \in \FT} h_F \left ( \norm[F]{q_f}^2 + h_F \norm[F]{\gsdofF{q}}^2 \right . \\
&\;\left . + \sum_{E \in EF} h_E \left [ \norm[E]{q_E}^2 + h_E \norm[E]{\gdofE{q}}^2 + \sum_{V \in \VE} h_E^2 \vert \gdofV{q} \vert^2 \right ] \right ).
\end{aligned}
\end{equation}
For all $\uvec{v}_T \in \uHcurl$ we define
\begin{equation}
\begin{aligned}
&\opnHcurl{ \uvec{v}_T}^2 =\; \norm[T]{\rkvec{v}}^2 + \norm[T]{\rkcvec{v}}^2 + \sum_{F \in \FT} h_F \left ( \norm[F]{\rkvec[F]{v}}^2 + \norm[F]{\rkcvec[F]{v}}^2 + \norm[F]{v_F}^2 \right .\\
&\quad \left . + h_F \norm[F]{\rdofFg{v}}^2 + h_F \norm[F]{\rdofFgc{v}}^2 
+ \sum_{E \in \EF} h_E \left [ \norm[E]{\bvec{v}_E}^2 + h_E \norm[E]{\rdofE{v}}^2 + \sum_{V \in \VE} h_E^2 \vert \rdofV{v} \vert^2 \right ] \right ) .
\end{aligned}
\end{equation}
For all $\uvec{w}_T \in \uHv$ we define
\begin{equation} \label{eq:defopnNa}
\begin{aligned}
\opnNa{ \uvec{w}_T}^2 =&\; \norm[T]{\bgvec{w}}^2 + \norm[T]{\bgcvec{w}}^2 + \sum_{F \in \FT} h_F \left ( \norm[F]{\bgvec[F]{w}}^2 +  \norm[F]{\bgcvec[F]{w}}^2 + \norm[F]{w_F}^2 \right . \\
&\; \left . + \sum_{E \in \EF} h_E \norm[E]{\bvec{w}_E}^2 \right ) .
\end{aligned}
\end{equation}
For all $\uvec{W}_T \in \uLt$ we define
\begin{equation} \label{eq:defopnLt}
\opnLt{ \uvec{W}_T}^2 = \norm[T]{\bvec{W}_T}^2 + \sum_{F \in \FT} h_F \left ( \norm[F]{\bvec{W}_F}^2 + \sum_{E \in \EF} h_E \norm[E]{\bvec{W}_E}^2 \right ) .
\end{equation}
And for all $q_T \in \uLs$ we define
\begin{equation} 
\normLs{ {q}_T}^2 = \norm[T]{q_T}^2.
\end{equation}

We show the equivalence between the norm induced by \eqref{eq:defspNa} and \eqref{eq:defopnNa} in Lemma \ref{lemma:normequiv}
and the equivalence between those induced by \eqref{eq:defspLt} and \eqref{eq:defopnLt} in Lemma \ref{lemma:normLtequiv}.

We define the global norms over $\Omega$ as the sum of the local norms over every cell $T \in \Th$,
i.e.\ $\opnNa[h]{\uvec{v}_h}^2 = \sum_{T \in \Th} \opnNa{\uvec{v}_T}^2$.

\subsection{Results on discrete $L^2$-products.}
We show some results to justify the choice of discrete norms.
\begin{lemma} \label{lemma:boundNaF}
For all $F \in \Fh$ and all $\uvec{w}_F \in \uHvF$ it holds:
\begin{equation*}
\norm[F]{\NaF \uvec{w}_F} \lesssim h_F^{-1} \opnNa[F]{\uvec{w}_F} .
\end{equation*}
\end{lemma}
\begin{proof}
This is a direct consequence of Lemma \ref{lemma:discretepoincare} and \ref{lemma:discretetrace} to the definition \eqref{eq:defNaF} with $\bvec{V}_F = \NaF \uvec{w}_F$.
\end{proof}

\begin{lemma}[Boundedness of the local trace] \label{lemma:boundtrna}
For all $F \in \Fh$ and all $\uvec{w}_F \in \uHvF$ it holds:
\begin{equation*}
\norm[F]{\trna \uvec{w}_F} \lesssim \norm[F]{\bgvec[F]{w}} +  \norm[F]{\bgcvec[F]{w}} + \norm[F]{w_F} + \sum_{E \in \EF} h_E^\frac{1}{2} \norm[E]{\bvec{w}_E} .
\end{equation*}
\end{lemma}
\begin{proof}
For any $\uvec{w}_F \in \uHvF$, let $\bvec{V}_F \in (\Roly{c,k+3}(F)^\intercal)^3$ be such that $\TDIV \bvec{V}_F = \trna \uvec{w}_F$.
From the estimate on the isomorphism \eqref{eq:isodiv} it holds $\norm[F]{\bvec{V}_F} \lesssim h_F \norm[F]{\trna \uvec{w}_F}$.
Then from \eqref{eq:deftrna} and \eqref{eq:defNaF} we write:
\begin{equation*}
\begin{aligned}
\int_F \trna \uvec{w}_F \cdot \TDIV \bvec{V}_F 
=&\; - \int_F \NaF \uvec{w}_F \tdot \bvec{V}_F + \sum_{E \in \EF} \wFE \int_E \bvec{w}_E \bvec{V}_F \nFE \\
\lesssim&\; \norm[F]{\NaF \uvec{w}_F} \norm[F]{\bvec{V}_F} + \sum_{E \in \EF} h_E^\frac{1}{2} \norm[E]{\bvec{w}_E} h_E^{-1} \norm[F]{\bvec{V}_F}.
\end{aligned}
\end{equation*}
We used the discrete trace inequality Lemma \ref{lemma:discretetrace} on the boundary term, 
and we conclude with Lemma \ref{lemma:boundNaF}.
\end{proof}

\begin{lemma}[Inverse Poincaré inequality] \label{lemma:InvPoincare}
For all $T \in \Th$ and all $\uvec{w}_T \in \uHv$ it holds:
\[ \norm{\NaT \uvec{w}_T} \lesssim h^{-1} \opnNa{\uvec{w}_T}. 
\]
\end{lemma}
\begin{proof}
Let $T \in \Th$ and $\uvec{w}_T \in \uHv$. 
The proof hinge on the bound of isomorphism \eqref{eq:isodiv} $\norm{\TDIV \NaT \uvec{w}_T} \lesssim h^{-1} \norm{\NaT \uvec{w}_T}$ to show that:
\begin{equation*}
\begin{aligned}
\int_T \NaF \uvec{w}_T \tdot \NaT \uvec{w}_T
\lesssim&\; \norm{\bgcvec{w}} h^{-1} \norm{\NaT \uvec{w}_T} + \norm{\bgvec{w}} h^{-1} \norm{\NaT \uvec{w}_T} \\
&\quad + \norm[F]{\trna \uvec{w}_F} \; h^{- \frac{1}{2}} \norm[T]{\NaT \uvec{w}_T} \\
\lesssim&\; h^{-1} \norm{\NaT \uvec{w}_T} ( \norm{\bgcvec{v}} + \norm{\bgvec{v}} + \sum_{F \in \FT} h^{\frac{1}{2}} \norm{\trna \uvec{v}_F}).
\end{aligned}
\end{equation*}
We conclude with Lemma \ref{lemma:boundtrna}.
\end{proof}

\begin{lemma}[Boundedness of the local potential] \label{lemma:boundpna}
For all $T \in \Th$ and all $\uvec{w}_T \in \uHv$ it holds: 
\begin{equation} \label{eq:boundpna}
\norm{\pna \uvec{w}_T} \lesssim \opnNa{\uvec{w}_T}
\end{equation}
\end{lemma}
\begin{proof}
Let $T \in \Th$, $\uvec{w}_T \in \uHv$
and let $\bvec{V}_T \in (\Roly{c,k+2}(T)^\intercal)^3$ such that $\TDIV \bvec{V}_T = \pna \uvec{w}_T$.
From the definition \eqref{eq:defpna} we have
\begin{equation*}
\begin{aligned}
\norm{\pna \uvec{w}_T}^2 =&\; \int_T \pna \uvec{w}_T \cdot \TDIV \bvec{V}_T \\ 
\leq&\; \norm{\NaT \uvec{w}_T} \norm{\bvec{V}_T} + \sum_{F \in \FT} \norm[F]{\trna \uvec{w}_F} \norm[F]{\bvec{V}_T} \\
\lesssim&\; \left ( h_T \norm{\NaT \uvec{w}_T} + \sum_{F \in \FT} h_F^\frac{1}{2} \norm[F]{\trna \uvec{w}_F} \right ) \norm[T]{\pna \uvec{w}_T} .
\end{aligned}
\end{equation*}
We conclude with Lemma \ref{lemma:boundtrna} and \ref{lemma:InvPoincare}.
\end{proof}

\begin{lemma} \label{lemma:normequiv}
It holds, for all $T \in \Th$
\[ \normNa{ \uvec{w}_T} \approx \opnNa{\uvec{w}_T}, \forall \uvec{w}_T \in \uHv .
\]
\end{lemma}
\begin{proof}
We apply the trace inequality to the definitions \eqref{eq:defspNa} and \eqref{eq:defstNa} to write:
\begin{equation*}
\begin{aligned}
\normNa{\uvec{w}_T}^2 
\leq&\; \norm[T]{\pna \uvec{w}_T}^2 
+ \sum_{F \in \FT} h_F \left ( \norm[F]{\trna \uvec{w}_F}^2 + \norm[F]{\pna \uvec{w}_T}^2 \right ) 
+ \sum_{E \in \ET} h_E^2 \left ( \norm[E]{\bvec{w}_E}^2 + \norm[E]{\pna \uvec{w}_T}^2 \right ) \\
\lesssim&\; \norm[T]{\pna \uvec{w}_T}^2 
+ \sum_{F \in \FT} h_F \norm[F]{\trna \uvec{w}_F}^2 
+ \sum_{E \in \ET} h_E^2  \norm[E]{\bvec{w}_E}^2 \\
\lesssim&\; \opnNa{\uvec{w}_T}^2 .
\end{aligned}
\end{equation*}
We conclude with Lemma \ref{lemma:boundpna} and \ref{lemma:boundtrna}.

Conversely, to prove $\opnNa{\uvec{w}_T} \lesssim \normNa{\uvec{w}_T}$ we begin to write
\begin{equation*}
\begin{aligned}
\norm[E]{\bvec{w}_E}^2 \lesssim&\; \norm[E]{\bvec{w}_E - \pna \uvec{w}_T}^2 + h_E^{-2} \norm[T]{\pna \uvec{w}_T}^2,\\
\norm[F]{\trna \uvec{w}_F}^2 \lesssim&\; \norm[F]{\trna \uvec{w}_F - \pna \uvec{w}_T}^2 + h_F^{-1} \norm[T]{\pna \uvec{w}_T}^2.
\end{aligned}
\end{equation*}
Then we conclude with consistency properties \eqref{eq:pipna} and \eqref{eq:pitrna} which allow us to bound 
$\norm[T]{\bgvec{w}}$, $\norm[T]{\bgcvec{w}}$, $\norm[F]{\bgvec[F]{w}}$, $\norm[F]{\bgcvec[F]{w}}$ and $\norm[F]{w_F}$.
For example $\norm[F]{\bgcvec[F]{w}} = \norm[F]{\Gcproj[F]{k} \trna \uvec{w}_F} \leq \norm[F]{\trna \uvec{w}_F}$.
\end{proof}

\begin{lemma} \label{lemma:normLtequiv}
It holds, for all $T \in \Th$
\begin{equation*}
\normLt{\uvec{W}_T} \approx \opnLt{\uvec{W}_T}, \forall \uvec{W}_T \in \uLt .
\end{equation*}
\end{lemma}
\begin{proof}
The same proof as Lemma \ref{lemma:normequiv} works.
\end{proof}

\section{Complex property.} \label{Complexproperty}
In this section we study the following sequence:
\begin{equation} \label{eq:complexseq}
\begin{tikzcd}
\uHgrad \arrow[r,"\uGh"] &
\uHcurlh \arrow[r,"\uCh"] &
\uHvh \arrow[r,"\Dh"] &
\uLsh .
\end{tikzcd}
\end{equation}
We will show in Theorem \ref{th:complexe} that \eqref{eq:complexseq} is indeed a complex, 
but first we show that the interpolators form a cochain morphism from a continuous de Rham complex into the sequence \eqref{eq:complexseq}.
\begin{lemma}[Local commutation properties] \label{lemma:localcommprop}
It holds for all $T \in \Th$, 
\begin{subequations} \label{eq:comm}
\begin{align}
\uGT (\uIgrad q) =&\; \uIcurl (\GRAD q), \quad \, \forall q \in C^2(\overline{T}), \label{eq:comm0}\\
\uCT (\uIcurl \bvec{v}) =&\; \uIH (\CURL \bvec{v}), \quad \, \forall \bvec{v} \in \bvec{C}^1(\overline{T}), \label{eq:comm1}\\
\uNaT (\uIH \bvec{w}) =&\; \uIL (\TGRAD \bvec{w}), \quad \forall \bvec{w} \in \bvec{C}^1(\overline{T}), \label{eq:comm2}\\
\DT (\uIH \bvec{w}) =&\; \lproj{k}{T} (\DIV \bvec{w}), \quad \, \forall \bvec{w} \in \bvec{C}^0(\overline{T}) \cap \bvec{H}^1(T) \label{eq:comm3}.
\end{align}
\end{subequations}
\end{lemma}
\begin{proof}
\underline{Proof of \eqref{eq:comm0}.}
We already have proved the relation on edges \eqref{eq:consuGE}. 
Let $q \in C^2(\overline{T})$, we will show that the relation holds for $\GFp$.
The proofs for $\GF$ and $\GT$ are almost the same.
For all $\bvec{w}_F \in \bPoly{k}(F)$,
\begin{equation} \label{eq:complexseq.proof1}
\begin{aligned}
\int_F \GFp (\uIgrad[F] q ) \cdot \bvec{w}_F 
=&\; - \int_F \lproj{k-1}{F} (\partial_{\nF} q) \ROT \bvec{w}_F - \sum_{E \in \EF} \wFE \int_E \lproj{k}{E}(\partial_{\nF} q) \; \bvec{w}_F \cdot \nE \\
=&\; - \int_F \partial_{\nF} q\; \ROT \bvec{w}_F - \sum_{E \in \EF} \wFE \int_E \partial_{\nF} q\; \bvec{w}_F \cdot \nE \\
=&\; - \int_F \VROT (\partial_{\nF} q) \cdot \bvec{w}_F .
\end{aligned}
\end{equation}
We conclude applying \eqref{eq:complexseq.proof1} for $\bvec{w}_F \in \Goly{k}(F)$ and $\bvec{w}_F \in \Goly{c,k}(F)$ 
since $\nF \times (\TGRAD (\GRAD q) \cdot \nF) = - \VROT (\partial_{\nF} q)$.\\
\underline{Proof of \eqref{eq:comm1}.}
Let $\bvec{v} \in C^1(\overline{T})$ and $\uvec{v}_T = \uIcurl \bvec{v}$.
We will show the property for $\CE$ and $\Gcproj[F]{k} \CF + \rdofFgc{v}$, the other components are easier and similar.
The same proof as \eqref{eq:consuGE} shows that $\dotp{\bvec{v}_E} = \vlproj{k+1}{E}(\partial_{\nE} \bvec{v})$.
Let us choose an arbitrary basis $(\bvec{x}_0, \bvec{x}_1, \bvec{x}_2)$ such that $\nE = (1,0,0)$ and write $\bvec{v} = (v_0, v_1, v_2)$.
In this basis we have
\begin{equation*}
\begin{aligned}
\vlproj{k+1}{E}(\CE \uvec{v}_E) =&\; \rdofE{v} - \dotp{\bvec{v}_E} \times \nE \\
=&\; \vlproj{k+1}{E} \left [ (\CURL \bvec{v} \cdot \nE) \nE + \GRAD (\bvec{v} \cdot \nE) \times \nE \right ] + \vlproj{k+1}{E}(\partial_0 \bvec{v}) \times \nE \\
=&\; \vlproj{k+1}{E} \left [ \begin{pmatrix} \partial_1 v_2 - \partial_2 v_1 \\ 0 \\ 0 \end{pmatrix} 
+ \begin{pmatrix} 0 \\ \partial_2 v_0 - \partial_0 v_2 \\ \partial_0 v_1 - \partial_1 v_0 \end{pmatrix} \right ] \\
=&\; \vlproj{k+1}{E} (\CURL \bvec{v}) .
\end{aligned}
\end{equation*}
Now let us take another basis such that $\nF = (1,0,0)$.
By the same argument we have $\CF (\uIcurl[F] \bvec{v}) = \vlproj{k}{F} (\VROT v_0)$ so
\begin{equation*}
\begin{aligned}
\Gcproj[F]{k} (\CF \uvec{v}_F) + \rdofFgc{v} 
=&\; \Gcproj[F]{k} \left [ \begin{pmatrix} 0 \\ \partial_2 v_0 \\ -\partial_1 v_0 \end{pmatrix} \right ]
+ \Gcproj[F]{k}\left [ \begin{pmatrix} 0 \\ - \partial_0 v_2 \\ \partial_0 v_1 \end{pmatrix} \right ] \\
=&\; \Gcproj[F]{k} (\ttr{v}).
\end{aligned}
\end{equation*}
\underline{Proof of \eqref{eq:comm2}.}
This is an immediate consequence of \eqref{eq:defNaE}, \eqref{eq:consNaF} and \eqref{eq:consNaT}.\\
\underline{Proof of \eqref{eq:comm3}.}
Let $\bvec{w} \in \bvec{C}^0(\overline{T}) \cap \bvec{H}^1(T)$.
For all $q_T \in \Poly{k}(T)$, since $\GRAD q_T \in \Goly{k-1}(T)$, we have:
\begin{equation*}
\begin{aligned}
\int_T \DT (\uIH \bvec{w})
=&\; - \int_T \Gproj{k-1} (\bvec{w}) \cdot \GRAD q_T + \sum_{F \in \FT} \wTF \int_F \lproj{k}{F} (\bvec{w} \cdot \nF)\; q_T \\
=&\; - \int_T \bvec{w} \cdot \GRAD q_T + \sum_{F \in \FT} \wTF \int_F (\bvec{w} \cdot \nF)\; q_T \\
=& \int_T \DIV \bvec{w}\; q_T.
\end{aligned}
\end{equation*}
\end{proof}

\begin{theorem}[Complex property] \label{th:complexe}
It holds: 
\begin{subequations} \label{eq:compl}
\begin{align}
\uIgradh \Real = \Ker \uGh, \label{eq:complexe0}\\
\Image \uGh \subset \Ker \uCh, \label{eq:complexe1}\\
\Image \uCh \subset \Ker \Dh, \label{eq:complexe2}\\
\Image \Dh = \Poly{k}(\Th). \label{eq:complexe3}
\end{align}
\end{subequations}
\end{theorem}
\begin{proof}
\underline{Proof of \eqref{eq:complexe0}.}
The inclusion $\uIgradh \Real \subset \Ker \uGh$ follows directly from \eqref{eq:comm0}.
Conversely if $\ul{q}_h \in \uHgradh$ is such that $\uGh \ul{q}_h = 0$ then
since $\uGE \ul{q}_E = 0$, \eqref{eq:defuGE} implies $\gdofV{q} = 0$, $\forall V \in \Vh$, $\gdofE{q} = 0$ and $\dotp{q_E} = 0$, $\forall E \in \Eh$.
So $q_E$ is constant on every edge, however $q_E$ is also continuous and $\Omega$ has a single connected component.
Thus, there is $C \in \Real$ such that $\forall E \in \Eh$, $q_E \equiv C$.
From \eqref{eq:defGF} and \eqref{eq:defuGF} we have $\forall \bvec{w}_F \in \Roly{c,k}(F)$, 
\begin{equation*}
0 = - \int_F q_F \DIV \bvec{w}_F + \sum_{E \in \EF} \wFE \int_E C \bvec{w}_F \cdot \nFE = \int_F (C - q_F) \DIV \bvec{w}_F.
\end{equation*}
Since $\DIV : \Roly{c,k}(F) \rightarrow \Poly{k-1}(F)$ is onto we must have $q_F \equiv C$, $\forall F \in \Fh$.
Likewise, since $\gdofE{q} = 0$, \eqref{eq:defGFp} and \eqref{eq:defuGF} give 
\begin{equation*} 
-\int_F \gsdofF{q} \; \ROT \bvec{w}_F = 0,\quad \forall \bvec{w} \in \bPoly{k}(F) .
\end{equation*}
Once again we must have $\gsdofF{q} = 0$, $\forall F \in \Fh$.
The same argument gives $\trna \ul{q}_F \equiv C$ and $q_T \equiv C$, $\forall T \in \Th$.
Thus $\Ker \uGh \subset \uIgradh \Real$. \\
\underline{Proof of \eqref{eq:complexe1}.}
We already have $\Image \uGF \subset \Ker \uCF$ by Lemma \ref{lemma:complexCFGF}. 
Let $\ul{q} \in \uHgradh$.
For any $T \in \Th$, since we project on $\Goly{k-1}(T) \oplus \Goly{c,k}(T)$ in \eqref{eq:defuCT} 
it is enough to show $\int_T \CT (\uGT \ul{q}_T) \cdot \bvec{r}_T = 0$, $\forall \bvec{r}_T \in \NE{k}(T)$.
Starting from \eqref{eq:defCT} we write
\begin{equation*}
\begin{aligned}
\int_T \CT (\uGT \ul{q}_T) \cdot \bvec{r}_T 
=&\; \int_T \Rproj{k-1}( \GT \ul{q}_T) \cdot \CURL \bvec{r}_T + \sum_{F \in \FT} \wTF \int_F \trcurl (\uGF \ul{q}_F) \cdot (\bvec{r}_T \times \nF) \\
=&\; \int_T (\GT \ul{q}_T) \cdot \CURL \bvec{r}_T + \sum_{F \in \FT} \wTF \int_F (\GF \ul{q}_F) \cdot (\bvec{r}_T \times \nF) \\
=&\; - \int_T {q}_T \; \DIV (\CURL \bvec{r}_T) + \sum_{F \in \FT} \wTF \int_F \trgrad (\ul{q}_F) (\CURL(\bvec{r}_T) \cdot \nF) \\
&\; + \sum_{F \in \FT} \wTF \left [ -\int_F q_F \; \DIV(\bvec{r}_T \times \nF)  + \sum_{E \in \EF} \wFE \int_E q_E (\bvec{r}_T \times \nF) \cdot \nFE \right ] \\
=&\; \sum_{F \in \FT} \wTF \int_F {q}_F \left [\CURL(\bvec{r}_T) \cdot \nF - \DIV (\bvec{r}_T \times \nF) \right ] \\
&\; - \sum_{F \in \FT} \wTF \sum_{E \in \EF} \wFE \int_E q_E\; \bvec{r}_T \cdot \nE.
\end{aligned}
\end{equation*}
We used $\CURL \bvec{r}_T \in \Roly{k-1}$ \eqref{eq:defkoszul} and $\bvec{r}_T \times \nF \subset \RT{k}(F)$ \eqref{eq:trNRT} along with \eqref{eq:trcurlGF} on the first line. 
Then we used \eqref{eq:defGT} and \eqref{eq:defGF} on the second line and \eqref{eq:conspitrgrad} on the last.
We conclude inferring $\CURL(\bvec{r}_T) \cdot \nF = \DIV (\bvec{r}_T \times \nF)$ 
and $\sum_{F \in \FT} \wTF \sum_{E \in \EF} \wFE \int_E q_E\; \bvec{r}_T \cdot \nE = 0$.
The last sum is zero since each edge shares exactly two faces on $\partial \overline{T}$ with opposite orientation. 
Hence we are counting each term twice, with a different sign each time.\\
\underline{Proof of \eqref{eq:complexe2}.} 
The same proof as \cite[Theorem~20]{ddr} works.\\
\underline{Proof of \eqref{eq:complexe3}.} See Lemma \ref{lemma:rightinvdiv}.
\end{proof}

The complex is exact if and only if the inclusions \eqref{eq:complexe1} and \eqref{eq:complexe2} are in fact equalities.
We can show that this is the same as asking for $\Omega$ to be contractible.
The proof is a slight adaptation of \cite[Section~4.3]{2020FullyDiscrete} 
and will not be duplicated here.
\begin{theorem}[Exactness]
If $\Omega$ is contractible then
\begin{subequations}
\begin{align}
\Image \uGh = \Ker \uCh, \label{eq:exactcomplexe1}\\
\Image \uCh = \Ker \Dh, \label{eq:exactcomplexe2}.
\end{align}
\end{subequations}
\end{theorem}
\begin{proof}
\underline{Proof of \eqref{eq:exactcomplexe1}.} Let $\uvec{v}_h \in \uHcurlh$ be such that $\uCh \uvec{v}_h = 0$.
We want to find $\ul{q}_h \in \uHgradh$ such that $\uGh \ul{q} = \uvec{v}$.
Starting from the proof of \cite[Theorem~3.1]{ddr}, if $\Omega$ is contractible and $\uCh \uvec{v}_h = 0$ then we can find $q \in \Poly[c]{k+1}(\Ech)$ such that 
$\forall E \in \Eh$, $\dotp{q} = \lproj{k}{E}(\bvec{v}_E)$.
Let $\gdofE{q} := \vlproj{k}{E}(\bvec{v}_E)$ and $\gdofV{q} = \bvec{v}_E(\bvec{x}_V)$. 
Since $\CE \uvec{v}_E = 0 \implies \rdofE{v} = \dotp{\bvec{v}_E} \times \nE$ and $\rdofV{v} = 0$ we have $\uvec{v}_E = \uGE (\ul{q}_E)$.
We must also have $\rdofFg{v} = - \Gproj[F]{k}(\bCF \uvec{v}_E)$ hence $\forall \bvec{r}_F \in \Goly{k}(F)$
\begin{equation*}
\begin{aligned}
\int_F \uCF \uvec{v} \cdot \bvec{r}_F =&\; \int_F v_F \ROT \bvec{r}_F + \sum_{E \in \EF} \wFE \int_F (\bvec{v}_E \cdot \nF) (\bvec{r}_F \cdot \nE),\\
-\int_F \GFp (\ul{q}) \cdot \bvec{r}_F =&\; \int_F \gsdofF{q} \ROT \bvec{r}_F + \sum_{E \in \EF} \wFE \int_F (\gdofE{q} \cdot \nF) (\bvec{r}_F \cdot \nE).\\
\end{aligned}
\end{equation*}
Thus we can take $\gsdofF{q} := v_F$.
We construct $q_F$ and $q_T$ exactly as in \cite[Theorem~3.2]{ddr}.\\
\underline{Proof of \eqref{eq:exactcomplexe2}.}
Let $\uvec{w}_h \in \uHvh$ be such that $\Dh \uvec{w}_h = 0$.
We want to find $\uvec{v}_h \in \uHcurlh$ such that $\uCh \uvec{v}_h = \uvec{w}_h$.
If $\Omega$ is contractible then \cite[Theorem~3.2]{ddr} 
provides $\rkvec{v} \in \Roly{k-1}(T)$, $\rkcvec{v} \in \Roly{c,k}(T)$, $\rkvec[F]{v} \in \Roly{k-1}(F)$, $\rkcvec[F]{v} \in \Roly{c,k}(F)$ 
and a normal component (along $\nE$) of $\bvec{v}_E$ such that $\CF \uvec{v}_F = w_F$ and 
$\Gproj{k-1}(\CT \uvec{v}_T) = \bgvec{w}$, $\Gcproj{k}(\CT \uvec{v}_T) = \bgcvec{w}$. 
It remains to show that $\CE$ and $\bCF$ are onto without using the above-mentioned degrees of freedom.
Let $\rdofE{v} := \dotp{\bvec{v}_E} \times \nE + \bvec{w}_E$, 
$v_F := 0$, 
$\rdofFg{v} := \bgvec[F]{w} - \Gproj[F]{k}(\bCF \uvec{v}_F)$ and 
$\rdofFgc{v} := \bgcvec[F]{w} - \Gcproj[F]{k}(\bCF \uvec{v}_F)$ (this makes sense since $\bCF$ does not depend on $\rdofFg{v}$ nor on $\rdofFgc{v}$).
It is easily checked that this choice gives $\uCh \uvec{v}_h = \uvec{w}_h$.
\end{proof}

\section{Consistency results.} \label{Consistencyresults}
The last things we need to show in order to efficiently use this complex are consistency results.
First we show primal consistency results, controlling the error made when we use the interpolators.
Then we show some Poincare type results useful to show stability, including a discrete counterpart to the right inverse for the divergence Lemma \ref{lemma:rightinvdiv}.
Finally we show adjoint consistency results, which control the error made when we perform a discrete integration by parts.


\begin{lemma}[Primal consistency] \label{lemma:primalconsistency}
For all $T \in \Th$ it holds:
\begin{equation}
\norm{\pna ( \uIH \bvec{w}) - \bvec{w}} \lesssim h^{k+2} \seminorm[\bvec{H}^{k+2}]{\bvec{w}},\quad \forall \bvec{w}\in \bvec{H}^{k+2}(T) .
\end{equation}
\end{lemma}
\begin{proof}
For all $T \in \Th$, \eqref{eq:pnaIH} shows that $\pna \uIH$ is a projection on $\bPoly{k+1}(T)$. 
Thus we just have to show that 
$\norm{\pna (\uIH \bvec{w})} \lesssim \norm{\bvec{w}} + h \norm[\bvec{H}^1]{\bvec{w}} + h^2 \norm[\bvec{H}^2]{\bvec{w}}$
to conclude with the lemma on approximation properties of bounded projector \cite[Lemma~1.43]{hho}.
Starting from \eqref{eq:boundpna} we have
\begin{equation*}
\begin{aligned}
\norm{\pna (\uIH \bvec{w})} \lesssim&\; \opnNa{\uIH \bvec{w}}\\
\lesssim&\; \norm[T]{\Gproj{k-1} \bvec{w}} + \norm[T]{\Gcproj{k} \bvec{w}}
+ \sum_{E \in \ET} h_E \norm[E]{\vlproj{k+1}{E} \bvec{w}} + \sum_{V \in \VE} h_E^\frac{3}{2} \vert \bvec{w}(\bvec{x}_V) \vert\\
&\; + \sum_{F \in \FT} h_F^\frac{1}{2} \left ( \norm[F]{\lproj{k}{F}(\bvec{w} \cdot \nF)} + \norm[F]{\Gproj[F]{k} \bvec{\ttr{w}}} + \norm[F]{\Gcproj[F]{k} \bvec{\ttr{w}}} \right ) \\
\lesssim&\; \norm[T]{\bvec{w}} + h_T \seminorm[\bvec{H}^1(T)]{\bvec{w}} + h_T^2 \seminorm[\bvec{H}^2(T)]{\bvec{w}},
\end{aligned}
\end{equation*}
where we used the continuous trace inequality \cite[Lemma~1.31]{hho} and the boundedness of $L^2$ projectors.
\end{proof}

\begin{lemma}[Stabilization forms consistency]
For all $T \in \Th$ it holds:
\begin{equation} \label{eq:consistencystNa}
\stNa{\uIH{ \bvec{w}}}{\uIH{ \bvec{w}}}^{1/2} \lesssim h^{k+2} \seminorm[\bvec{H}^{k+2}(T)]{\bvec{w}}, \forall \bvec{w} \in \bvec{H}^{k+2}(T) .
\end{equation}
\begin{equation} \label{eq:consistencystLt}
\stLt{\uIL{ \bvec{W}}}{\uIL{ \bvec{W}}}^{1/2} \lesssim h^{k+1} \seminorm[\bvec{H}^{k+1}(T)]{\bvec{W}}, \forall \bvec{W} \in \bvec{H}^{k+1}(T)  \cap \bvec{C}^0(\overline{T}).
\end{equation}
\end{lemma}
\begin{proof}
\underline{Proof of \eqref{eq:consistencystNa}.}
For all $\bvec{z}_T \in \bPoly{k+1}(T)$ we have $\pna (\uIH \bvec{z}_T) = \bvec{z}_T$
by \eqref{eq:pnaIH} and $\trna (\uIH \bvec{z}_T) = \bvec{z}_T$ by \eqref{eq:trnaIH} so for all $\uvec{v}_T \in \uHv$, 
\begin{equation*}
\begin{aligned}
\stNa{\uIH \bvec{z}_T}{\uvec{v}_T} 
=&\; \sum_{F \in \FT} h_F \int_F (\pna \uIH \bvec{z}_T - \trna \uIH \bvec{z}_T) \cdot (\pna \uvec{v}_T - \trna \uvec{v}_F) \\
=&\; \sum_{E \in \EF} h_E^2 \int_E (\pna \uIH \bvec{z}_T - \bvec{z}_T) \cdot (\pna \uvec{v}_T - \bvec{v}_E) = 0.
\end{aligned}
\end{equation*}
Hence 
\begin{equation*}
\stNa{\uIH \bvec{w}_T}{\uIH \bvec{w}_T} = \stNa{\uIH (\bvec{w}_T - \vlproj{k+1}{T})}{ \uIH (\bvec{w}_T - \vlproj{k+1}{T})} \lesssim \normNa{\uIH (\bvec{w}_T - \vlproj{k+1}{T})}^2.
\end{equation*}
We conclude by the norm equivalence Lemma \ref{lemma:normequiv} and \cite[Theorem~1.45]{hho}.\\
\underline{Proof of \eqref{eq:consistencystLt}.}
Let $\bvec{W} \in \bvec{H}^{k+1}(T) \cap \bvec{C}^0(\overline{T})$, we have:
\begin{equation*}
\begin{aligned}
\sum_{F \in \FT} h_F \norm[F]{\Tttrnb{\RTbproj{k+1}(\bvec{W})} - \FRTbproj{k+1}(\bvec{W})}^2 
\leq&\; \sum_{F \in \FT} h_F \norm[F]{(I - \RTbproj{k+1}) \bvec{W}}^2 \\
\leq&\; \sum_{F \in \FT} h_F (\norm[F]{\bvec{W} - \vlproj{k}{T} \bvec{W}}^2 + \norm[F]{\vlproj{k}{T} \bvec{W} - \RTbproj{k+1} \bvec{W}}^2) \\
\lesssim&\; h^{2(k+1)} \seminorm[\bvec{H}^{k+1}]{\bvec{W}} + \norm[F]{\bvec{W} - \vlproj{k}{T} \bvec{W}}^2 .
\end{aligned}
\end{equation*}
Here the second equality comes from $\Tttrnb{\RTb{k+1}(T)} \subset \FRTb{k+1}$
and we used the approximation properties on traces \cite[Theorem~1.45 and Equation~1.75]{hho} on the first term 
and the discrete trace inequality \cite[Lemma~1.32]{hho} on the second term to get the last equality.
Likewise, we have:
\begin{equation*}
\sum_{F \in \FT} h_F \sum_{E \in \EF} h_E \norm[E]{\RTbproj{k+1}(\bvec{W}) \; \nE - \vlproj{k+2}{E}(\bvec{W} \; \nE)}^2 
\lesssim \sum_{F \in \FT} \sum_{E \in \EF} h_F \norm[F]{(I - \RTbproj{k+1}) \bvec{W}}^2 .
\end{equation*}
We conclude with \cite[Theorem~1.45 and Equation~1.74]{hho}.
\end{proof}

\subsection{Poincaré inequality.}
We begin by stating two lemmas which will be useful to prove the Poincare inequality.
\begin{lemma} \label{lemma:normgrad}
For all $T \in \Th$, $F \in \FT$ and all $\uvec{w}_T \in \uHv$ it holds that
\begin{equation}
\begin{aligned}
\norm[F]{\TGRAD \trna \;\uvec{w}_F}^2 + \sum_{E \in \EF} h_E^{-1} \norm[E]{\trna \;\uvec{w}_F - \bvec{w}_{E}}^2 \lesssim&\; \opnLt[F]{\uNaF \uvec{v}_F}^2,\\
\norm[T]{\TGRAD \pna \;\uvec{w}_T}^2 + \sum_{F \in \FT} h_F^{-1} \norm[F]{\pna \;\uvec{w}_T - \trna \;\uvec{w}_F}^2 \lesssim&\; \opnLt[T]{\uNaT \uvec{v}_T}^2.
\end{aligned}
\end{equation}
\end{lemma}
\begin{proof}
The proof is a simple adaptation of \cite[Lemma~5.7]{ddr}.
\end{proof}

\begin{lemma}[Poincaré inequality] \label{lemma:poincareH1}
For all $\uvec{w}_h \in \uHvh$ such that $\sum_{T \in \Th} \int_{T} \pna \uvec{w}_T = 0$ it holds that
\begin{equation}
\opnNa[h]{\uvec{w}_h} \lesssim \opnLt[h]{\uNah \uvec{w}_h}.
\end{equation}
\end{lemma}
\begin{proof}
The proof is a simple adaptation of \cite[Theorem~5.3]{ddr}.
\end{proof}
\begin{remark}
When $k \geq 1$ the assumption $\sum_{T \in \Th} \int_{T} \pna \uvec{w}_T = 0$ translates to 
$\sum_{T \in \Th} \int_{T} \bvec{w}_T = 0$ by \eqref{eq:pipna}.
However this does not hold when $k = 0$.
\end{remark}

Lastly we show that the fully discrete divergence is onto $\Dh : \uHvh \rightarrow \uLsh$. 
The main difficulty is to show the boundedness of the inverse with the discrete norms.
\begin{lemma}[Right-inverse for the divergence] \label{lemma:rightinvdiv}
For all $\ul{p}_h \in \uLsh$ there is $\uvec{w}_h \in \uHvh$ such that $\Dh \uvec{w}_h = \ul{p}_h$ and 
$\opnNa[h]{\uvec{w}_h} + \opnLt[h]{\uNah \uvec{w}_h} \lesssim \normLs[h]{\ul{p}_h}$.
\end{lemma}
\begin{proof}
\underline{Existence.}
Let $\ul{p}_h = (p_T)_{T \in \Th} \in \uLsh$.
Lemma \ref{lemma:normpczH2} provides $\tilde{p} \in C^1(\overline{\Omega})$ 
such that $\forall T \in \Th$, $\tilde{p}_{\vert T} \in \Poly{k + \max_{h, T \in \Th}(2 \vert \FT \vert)}(T)$,
$\lproj{k}{T} \tilde{p} = p_T$
and $\norm[L^2(\Omega)]{\tilde{p}} \approx \normLs[h]{\ul{p}_h}$. 
Under the assumption on the regularity of the mesh sequence we have $\max_{h, T \in \Th}(\vert \FT \vert) \lesssim 1$ (\cite[Lemma~1.12]{hho}) 
so that the maximum degree is bounded independently of $h$.
Since $\tilde{p}$ is a piecewise polynomial, continuous, of continuous derivative and of trace zero on the boundary, $\tilde{p} \in H^2_0(\Omega)$.
We apply Theorem \ref{th:divlift} to find $\bvec{u} \in \bvec{H}^3(\Omega)$ such that $\DIV \bvec{u} = \tilde{p}$, 
$\norm[\bvec{H}^3]{\bvec{u}} \lesssim \norm[H^2]{\tilde{p}}$, 
$\norm[\bvec{H}^2]{\bvec{u}} \lesssim \norm[H^1]{\tilde{p}}$ and 
$\norm[\bvec{H}^1]{\bvec{u}} \lesssim \norm[L^2]{\tilde{p}}$.
We build $\uvec{w}_h \in \uHvh$ by $\uvec{w}_F = \uIH[F](\bvec{u})$, $\forall F \in \Fh$,
$\bgcvec{w} = \Gcproj{k}(\bvec{u})$, $\forall T \in \Th$, 
and $\forall q_T \in \Poly{0,k}(T)$:
\begin{equation*}
\int_T \bgvec{w} \GRAD q_T := - \int_T p_T q_T + \sum_{F \in \FT} \wTF \int_F \trna ( \uIH[F] \bvec{u}) \cdot \nF q_T.
\end{equation*}
Hence by construction, $\forall T \in \Th$, $\forall q_T \in \Poly{0,k}(T)$,
$\int_T \DT \uvec{w}_T  q_T = \int_T p_T q_T$.
It remains to show the equality for $q_T \in \Poly{0}(T)$,
and using \eqref{eq:pitrna} we have
\begin{equation*}
\begin{aligned}
\int_T \DT \uvec{w}_T =&\; \sum_{F \in \FT} \wTF \int_F \trna ( \uIH[F] \bvec{u}) \cdot \nF \\
=&\; \sum_{F \in \FT} \wTF \int_F \lproj{k}{F}(\bvec{u} \cdot \nF) \\
=&\; \int_T \DIV \bvec{u} = \int_T p_T .
\end{aligned}
\end{equation*}
Thus we have $\Dh \uvec{w}_h = \ul{p}_h$.\\
\underline{Boundedness.}
Many bounds follow directly from the $L^2$-boundedness of the interpolator.
Indeed, we make use of the continuous trace inequality \cite[Lemma~1.31]{hho} to get
\begin{equation*}
\begin{aligned}
\norm[\bvec{H}^k(F)]{\bvec{u}} \lesssim&\; h^{-\frac{1}{2}} \norm[\bvec{H}^k(T)]{\bvec{u}} + h^\frac{1}{2} \norm[\bvec{H}^{k+1}(T)]{\bvec{u}},\\
\norm[\bvec{H}^k(E)]{\bvec{u}} \lesssim&\; h^{-1} \norm[\bvec{H}^k(T)]{\bvec{u}} + \norm[\bvec{H}^{k+1}(T)]{\bvec{u}} + h\norm[\bvec{H}^{k+2}(T)]{\bvec{u}}.
\end{aligned}
\end{equation*}
Inferring the estimate on $\norm[\bvec{H}^k(T)]{\bvec{u}}$ and Lemma \ref{lemma:discretepoincare} we get
\begin{equation} \label{eq:estimatetrcontinuous}
\begin{aligned}
\norm[\bvec{H}^k(F)]{\bvec{u}} \lesssim&\; h^{\min \lbrace 0,\frac{1}{2} - k \rbrace} \normLs[T]{\ul{p}_h},\\
\norm[\bvec{H}^k(E)]{\bvec{u}} \lesssim&\; h^{-k} \normLs[T]{\ul{p}_h}.
\end{aligned}
\end{equation}
This allows to bound all terms of $\opnNa[h]{\uvec{w}_h}$ but $\bgvec{w}$. 
This one is also easily bounded, for all $T \in \Th$, let $q_T \in \Poly{0,k}(T)$ be such that $\GRAD q_T = \bgvec{w}$.
We have $\norm{q_T} \lesssim h_T \norm{\bgvec{w}}$ and 
\begin{equation*}
\begin{aligned}
\int_T \bgvec{w} \GRAD q_T =&\; - \int_T p_T q_T + \sum_{F \in \FT} \wTF \int_F \trna ( \uIH[F] \bvec{u}) \cdot \nF q_T \\
\lesssim&\; \norm{p_T} \norm{q_T} + \sum_{F \in \FT} h^{\frac{1}{2}} \norm[F]{\trna ( \uIH[F] \bvec{u})} h^{-1}\norm[T]{q_T}\\
\lesssim&\; \norm{p_T} \norm{\bgvec{w}}.
\end{aligned}
\end{equation*}

It remains to estimate $\opnLt{\uNah \uvec{w}_h} \approx \sum_{T \in \Th} \norm[T]{\NaT \uvec{w}_T} + \sum_{F \in \Fh} h_F \norm[F]{\NaF \uvec{w}_F} + \sum_{E \in \Eh} h_E^2 \norm[E]{\NaE \uvec{w}_E}$.
Lemma \ref{lemma:NaFuIH} and an easily proved variation of \eqref{eq:consuGE} and \eqref{eq:estimatetrcontinuous} give
\begin{equation*}
\begin{aligned}
h^2 \norm[E]{\NaE \uIH[E] \bvec{u}} \lesssim&\; h^2 \norm[\bvec{H}^1(E)]{\bvec{u}}^2 \lesssim h^2 (h^{-1})^2 \normLs[T]{\ul{p}_h},\\
h \norm[F]{\NaF \uIH[F] \bvec{u}} \lesssim&\; h \norm[\bvec{H}^1(F)]{\bvec{u}}^2 \lesssim h (h^{-\frac{1}{2}})^2 \normLs[T]{\ul{p}_h}.
\end{aligned}
\end{equation*}
To estimate $\norm[T]{\NaT \uvec{w}_T}$ let $\bvec{V}_T := \NaT \uvec{w}_T$ and take $q_T \in \Poly{0,k}(T)$ such that 
$\GRAD q_T = \TDIV \brvec{V}$ with $\norm{q_T} \approx \norm{\brvec{V}}$.
Starting from its definition \eqref{eq:defNaT} we write
\begin{equation*}
\begin{aligned}
\int_T \NaT (\uvec{w}_T) \tdot \bvec{V}_T 
=&\; - \int_T \bgcvec{w} \cdot \TDIV (\brcvec{V}) - \int_T \bgvec{w} \cdot \TDIV (\brvec{V}) + \sum_{F \in \FT} \wTF \int_F \trna \uvec{w}_F \bvec{V}_T \nF \\
=&\; - \int_T \Gcproj{k}(\bvec{u}) \cdot \TDIV (\brcvec{V}) + \int_T p_T q_T + \sum_{F \in \FT} \wTF \int_F \trna \uvec{w}_F (\bvec{V}_T - q_T \ID) \nF \\
=&\; - \int_T \bvec{u} \cdot \TDIV (\brcvec{V}) + \int_T p_T q_T + \sum_{F \in \FT} \wTF \int_F \bvec{u} (\bvec{V}_T - q_T \ID) \nF \\
=&\; \cancel{- \int_T \bvec{u} \cdot \TDIV (\brcvec{V})} + \int_T p_T q_T + \int_T \TGRAD \bvec{u} \tdot (\bvec{V}_T - q_T \ID) + \cancel{\int_T \bvec{u} \cdot \TDIV (\bvec{V}_T - q_T \ID)} \\
\lesssim&\; \norm{p_T} \norm{q_T} + \norm{\TGRAD \bvec{u}} \norm{\bvec{V}_T - q_T \ID}\\
\lesssim&\; \norm{p_T} \norm{\NaT \uvec{w}_T}.
\end{aligned}
\end{equation*}
We used that $\RTb{k+1} \nF \subset \bPoly{k}(F)$ (from \eqref{eq:trNRT}) with \eqref{eq:pitrna} and Lemma \ref{lemma:defGbc} on the second line,
an integration by parts on the third and that $\TDIV (\bvec{V}_T - q_T \ID) = \TDIV \bvec{V}_T - \TDIV \brvec{V} = \TDIV (\brcvec{V})$ to cancel the terms in the fourth line.
The result follows from the Cauchy-Swartchz inequality and the estimates on $\bvec{u}$.
\end{proof}
\begin{remark} \label{rem:rightinvdiv}
We can easily adapt Lemma \ref{lemma:rightinvdiv} to require $\sum_{T \in \Th} \int_T \pna \uvec{w}_T = 0$.
Defining $\uvec{w'}_h = \uvec{w}_h - \uIHh \left (\frac{1}{\int_\Omega 1} \sum_{T \in \Th} \int_T \pna \uvec{w}_T \right )$,
it is clear from \eqref{eq:pnaIH} that $\sum_{T \in \Th} \int_T \pna \uvec{w'}_T = 0$,
from Lemma \ref{lemma:localcommprop} that $\Dh \uvec{w'}_h = \ul{p}_h$ and from \eqref{eq:boundpna} that the estimate of Lemma \ref{lemma:rightinvdiv} on the norm of $\uvec{w'}_h$ still holds.
\end{remark}

\subsection{Adjoint consistency.}
We define the adjoint consistency error for all $\bvec{V} \in \bvec{C}^0(\overline{\Omega}) \cap \bvec{H}^1_0(\Omega)$
and all $\uvec{w}_h \in \uHvh$ by:
\begin{equation} \label{eq:defAdjG}
\AdjG(\bvec{V},\uvec{w}_h) = \sum_{T \in \Th} \left (\spLt{\uIL \bvec{V}}{\uNaT \uvec{w}_T} + \int_T \TDIV \bvec{V} \cdot \pna \uvec{w}_T \right ) .
\end{equation}

\begin{theorem}[Adjoint consistency for the gradient] \label{th:AdjConsG}
For all $\bvec{V} \in \bvec{C}^0(\overline{\Omega}) \cap \bvec{H}^1_0(\Omega)$ such that $\bvec{V} \in \bvec{H}^{k+2}(\Th)$
and all $\uvec{w}_h \in \uHvh$, it holds:
\begin{equation}
\seminorm{\AdjG(\bvec{V},\uvec{w}_h)} \lesssim h^{k+1} \left ( \seminorm[\bvec{H}^{k+1}]{\bvec{V}} + \seminorm[\bvec{H}^{k+2}]{\bvec{V}} \right )
\left ( \opnNa[h]{\uvec{w}_h} + \opnLt[h]{\uNah \uvec{w}_h} \right ).
\end{equation}
\end{theorem}
\begin{proof}
Remarks \ref{rem:validitypna} and \ref{rem:sequential} show that $\forall \bvec{V}_h \in (\RT{k+1}(\Th)^\intercal)^3$, $\forall T \in \Th$,  
\begin{equation*}
\int_T \pna \uvec{w}_T \cdot \TDIV \bvec{V}_T + \int_T \NaT \uvec{w}_T \tdot \bvec{V}_T - \sum_{F \in \FT} \wTF \int_F \trna \uvec{w}_F \bvec{V}_T \nF = 0 .
\end{equation*}
Let $\trnac \uvec{w}_F$ be the continuous polynomial given by Lemma \ref{lemma:normpczH2} such that $\vlproj{k+2}{F} (\trnac \uvec{w}_F) = \trna \uvec{w}_F$,
$(\trnac \uvec{w}_F)_{\vert E} = \bvec{0}$, $(\TGRAD \trnac \uvec{w}_F)_{\vert E} = \bvec{0}$ and $\norm[F]{\trnac \uvec{w}_F} \approx \norm[F]{\trna \uvec{w}_F}$.
It holds that $\sum_{F \in \FT} \wTF \int_F \trna \uvec{w}_F \bvec{V}_T \nF = \sum_{F \in \FT} \wTF \int_F \trnac \uvec{w}_F \bvec{V}_T \nF$.
Moreover, since $\bvec{V} \cdot \nOmega = 0$ on $\partial \Omega$ and since the $\trnac \bvec{w}_F$ are single valued we have
\begin{equation} \label{eq:adjgradzerosum}
\sum_{T \in \Th} \sum_{F \in \FT} \wTF \int_F \trnac \uvec{w}_F \bvec{V} \nF = 0 .
\end{equation}
Hence we can write:
\begin{equation*}
\begin{aligned}
\AdjG(\bvec{V},\uvec{w}_h) =& \sum_{T \in \Th} \left (\int_T (\bvec{V} - \bvec{V}_T) \tdot \NaT \uvec{w}_T + \int_T \TDIV (\bvec{V} - \bvec{V}_T) \cdot \pna \uvec{w}_T \right. \\
& \quad \left. + \sum_{F \in \FT} \wTF \int_E \trnac \uvec{w}_F (\bvec{V}_T - \bvec{V}) \nF 
+ \stLt{\uIL \bvec{V}}{\NaT \uvec{w}_T} \right ) \\
\lesssim&\; \sum_{T \in \Th} \left ( \norm{\bvec{V} - \bvec{V}_T} + \norm{\TDIV ( \bvec{V} - \bvec{V}_T)} \right )
\left ( \norm{\NaT \uvec{w}_T} + \norm{\pna \uvec{w}_T} \right )\\
& \quad + \stLt{\uIL \bvec{V}}{\uIL \bvec{V}}^{1/2} \stLt{\NaT \uvec{w}_T}{\NaT \uvec{w}_T}^{1/2} \\
& \quad + \sum_{F \in \FT} \wTF \int_F \trnac \uvec{w}_F (\bvec{V}_T - \bvec{V}) \nF .
\end{aligned}
\end{equation*}
Applying \eqref{eq:consistencystLt} and Lemma \ref{lemma:normLtequiv} gives:
\begin{equation*}
\stLt{\uIL \bvec{V}}{\uIL \bvec{V}}^{1/2} \stLt{\NaT \uvec{w}_T}{\NaT \uvec{w}_T}^{1/2} \lesssim
h^{k+1} \seminorm[\bvec{H}^{k+1}(T)]{\bvec{V}} \opnLt{\uNaT \uvec{w}_T} .
\end{equation*}
Using the approximation properties of the spaces $\RT{k+1}(T)$ given by a slight adaptation of \cite[Lemma~6.8]{ddr}
we can find $\bvec{V}_T \in (\RT{k+1}(T)^\intercal)^3$ such that
\begin{equation*}
\norm{\bvec{V} - \bvec{V}_T} + \norm{\TDIV ( \bvec{V} - \bvec{V}_T)} \lesssim h^{k+1} \left ( \seminorm[\bvec{H}^{k+1}(T)]{\bvec{V}} + \seminorm[\bvec{H}^{k+2}(T)]{\bvec{V}} \right )
\end{equation*}
By \eqref{eq:boundpna} we see that
\begin{equation*}
\norm{\NaT \uvec{w}_T} + \norm{\pna \uvec{w}_T} \lesssim \opnLt{\uNaT \uvec{w}_T} + \opnNa{\uvec{w}_T} .
\end{equation*}
Lastly we use Theorem \ref{th:tracelift} to find $\LiftNaF{\uvec{w}_T} \in \Hv{T}$ such that
\begin{equation*}
\begin{aligned}
\sum_{F \in \FT} \wTF \int_F \trnac \uvec{w}_F (\bvec{V}_T - \bvec{V}) \nF =& 
\sum_{F \in \FT} \wTF \int_F \LiftNaF{\uvec{w}_T} (\bvec{V}_T - \bvec{V}) \nF \\
=& \int_T \TGRAD \LiftNaF{\uvec{w}_T} \tdot (\bvec{V}_T - \bvec{V}) 
+ \int_T \LiftNaF{\uvec{w}_T} \cdot \TDIV (\bvec{V}_T - \bvec{V}) .
\end{aligned}
\end{equation*}
Hence
\begin{equation*}
\seminorm{\sum_{F \in \FT} \wTF \int_F \trna \uvec{w}_F (\bvec{V}_T - \bvec{V}) \nF} \lesssim
\left ( \norm{\bvec{V} - \bvec{V}_T} + \norm{\TDIV ( \bvec{V}_T - \bvec{V})} \right )
\left ( \norm{\TGRAD \LiftNaF{\uvec{w}_T}} + \norm{\LiftNaF{\uvec{w}_T}} \right ),
\end{equation*}
and we conclude with Theorem \ref{th:tracelift} which gives the boundedness of $\LiftNaF{\uvec{w}_T}$.
\end{proof}

We can sharpen the estimate \eqref{eq:defAdjG} when $\bvec{V}$ is the gradient of some field.
Indeed, if were to take $\bvec{V} = \TGRAD v$ in Theorem \ref{th:AdjConsG}, 
we would see that a norm over $\bvec{H}^{k+3}$ appears in the estimate, 
which is not optimal.

We define the adjoint consistency error for all $\bvec{v} \in \bvec{H}^2(\Omega)$ such that $\TGRAD \bvec{v} \cdot \nOmega = 0$ 
and all $\uvec{w}_h \in \uHvh$ by:
\begin{equation} \label{eq:defAdjL}
\AdjL(\bvec{v},\uvec{w}_h) = \sum_{T \in \Th} \left ( \int_T \TLAPLACIAN \bvec{v} \cdot \pna \uvec{w}_T + \spLt{\uNaT \uIH\bvec{v}}{\uNaT \uvec{w}_T} \right ) .
\end{equation}
\begin{theorem}[Adjoint consistency for the Laplacian] \label{th:AdjConsL}
For all $\bvec{v} \in \bvec{H}^2(\Omega) \cap \bvec{C}^1(\overline{\Omega})$ such that $\TGRAD \bvec{v} \cdot \nOmega = 0$ 
and $\bvec{v} \in \bvec{H}^{k+2}(\Th)$ 
and for all $\uvec{w}_h \in \uHvh$, it holds: 
\begin{equation}
\seminorm{\AdjL(\bvec{v},\uvec{w}_h)} \lesssim h^{k+1} \seminorm[\bvec{H}^{k+2}]{\bvec{v}} 
\opnLt[h]{\uNah \uvec{w}_h}.
\end{equation}
\end{theorem}
\begin{proof}
For any $T \in \Th$, \eqref{eq:comm2} gives:
\begin{equation*}
\spLt{\uNaT \uIH \bvec{v}}{\uNaT \uvec{w}_T} = \int_T \RTbproj{k+1} \TGRAD \bvec{v} \tdot \NaT \uvec{w}_T + \stLt{\uIL \TGRAD \bvec{v}}{\uNaT \uvec{w}} .
\end{equation*}
With an integration by parts and since $\int_T \RTbproj{k+1} \TGRAD \bvec{v} \tdot \NaT \uvec{w}_T = \int_T \TGRAD \bvec{v} \tdot \NaT \uvec{w}_T$
we have:
\begin{equation*}
\begin{aligned}
\AdjL(\bvec{v},\uvec{w}_h) =& \sum_{T \in \Th} \left ( \int_T \TGRAD \bvec{v} \tdot (\NaT \uvec{w}_T - \TGRAD \pna \uvec{w}_T ) + 
\stLt{\uIL \TGRAD \bvec{v}}{\uNaT \uvec{w}} \right . \\
&\quad + \left . \sum_{F \in \FT} \wTF \int_F \pna \uvec{w}_T \TGRAD \bvec{v} \; \nF \right ) .
\end{aligned}
\end{equation*}
Since we assume $\TGRAD \bvec{v}  \cdot \nOmega = 0$ we have 
\begin{equation} \label{eq:adjlapzerosum}
\sum_{T \in \Th} \sum_{F \in \FT} \wTF \int_F \trna \uvec{w}_F \TGRAD \bvec{v}\; \nF = 0.
\end{equation}
By Remark \ref{rem:validitypna} it holds 
 $\forall \bvec{v}_T \in \bPoly{k+1}(T)$,
\begin{equation*}
\int_T \TLAPLACIAN \bvec{v}_T \cdot \pna \uvec{w}_T + \int_T \NaT \uvec{w}_T \tdot \TGRAD \bvec{v}_T - \sum_{F \in \FT} \wTF \int_F \trna \uvec{w}_F \TGRAD \bvec{v}_T \; \nF = 0, 
\end{equation*}
so
\begin{equation*}
\int_T \TGRAD \bvec{v}_T \tdot (\NaT \uvec{w}_T - \TGRAD \pna \uvec{w}_T) + \sum_{F \in \FT} \wTF \int_F (\pna \uvec{w}_T - \trna \uvec{w}_F) \TGRAD \bvec{v}_T \; \nF = 0 .
\end{equation*}
This allows us to write for any $\uvec{v}_h = (\bvec{v}_T)_{T \in \Th} \in \bPoly{k+1}(T)$,
\begin{equation*}
\begin{aligned}
\AdjL(\bvec{v},\uvec{w}_h) =& \sum_{T \in \Th} \left ( \int_T \TGRAD (\bvec{v} - \bvec{v}_T) \tdot (\NaT \uvec{w}_T - \TGRAD \pna \uvec{w}_T ) + 
\stLt{\uIL \TGRAD \bvec{v}}{\uNaT \uvec{w}} \right . \\
&\quad + \left . \sum_{F \in \FT} \wFE \int_F (\pna \uvec{w}_T - \trna \uvec{w}_F) \TGRAD (\bvec{v} - \bvec{v}_T) \nF \right ),\\
\seminorm{\AdjL(\bvec{v},\uvec{w}_h)} \lesssim&\; \sum_{T \in \Th} \left ( \norm[T]{\TGRAD (\bvec{v} - \bvec{v}_T)} \norm[T]{\NaT \uvec{w}_T - \TGRAD \pna \uvec{w}_T} \right .\\
& \quad \left .
+ \sum_{F \in FT} \norm[F]{\pna \uvec{w}_T - \trna \uvec{w}_F} \norm[F]{\TGRAD (\bvec{v} - \bvec{v}_T)} 
+ \seminorm{\stLt{\uIL \TGRAD \bvec{v}}{\uNaT \uvec{w}}} \right ).
\end{aligned}
\end{equation*}
Applying Lemma \ref{lemma:normgrad} we get
\begin{equation*}
\norm[F]{\pna \uvec{w}_T - \trna \uvec{w}_F} \norm[F]{\TGRAD (\bvec{v} - \bvec{v}_T)} \lesssim \opnLt{\uNaT \uvec{w}_T} h^{\frac{1}{2}} \norm[F]{\TGRAD (\bvec{v} - \bvec{v}_T)} .
\end{equation*}
Furthermore \eqref{eq:consistencystLt} and Lemma \ref{lemma:normLtequiv} give
\begin{equation*}
\seminorm{\stLt{\uIL \TGRAD \bvec{v}}{\uNaT \uvec{w}}} \lesssim h^{k+1} \seminorm[\bvec{H}^{k+1}]{\TGRAD \bvec{v}} \opnLt{\uNaT \uvec{w}_T} .
\end{equation*}
Hence, applying Lemma \ref{lemma:normgrad} we write:
\begin{equation*}
\begin{aligned}
\seminorm{\AdjL(\bvec{v},\uvec{w}_h)} \lesssim&\; \sum_{T \in \Th} \left ( 
\opnLt{\uNaT \uvec{w}_T} ( \norm[T]{\TGRAD (\bvec{v} - \bvec{v}_T)} 
+ \sum_{F \in \FT} h^{\frac{1}{2}} \norm[F]{\TGRAD (\bvec{v} - \bvec{v}_T)}) \right ) \\
&+ h^{k+1} \seminorm[\bvec{H}^{k+2}]{\bvec{v}} \opnLt{\uNaT \uvec{w}_h} .
\end{aligned}
\end{equation*}
We conclude by taking $\bvec{v}_T = \vlproj{1,k+1}{T} \bvec{v}$ the elliptic projection on $T$ (see \cite[Definition~1.39]{hho}), 
then \cite[Theorem~1.48]{hho} gives: 
\begin{equation*}
\norm[\bvec{H}^1(T)]{\bvec{v} - \vlproj{1,k+1}{T} \bvec{v}} \lesssim h^{k+1} \seminorm[\bvec{H}^{k+2}]{\bvec{v}} ,
\end{equation*} 
\begin{equation*}
h^{\frac{1}{2}} \norm[\bvec{H}^1(F)]{\bvec{v} - \vlproj{1,k+1}{T} \bvec{v}} \lesssim h^{k+1} \seminorm[\bvec{H}^{k+2}]{\bvec{v}} .
\end{equation*}
\end{proof}

\section{Stokes equations.} \label{Stokes}
Finally, we illustrate this complex with the design of a scheme for the Stokes equations.
For the sake of simplicity we use Neumann boundary conditions over the whole boundary, 
that it to say with a free outlet condition. More general conditions are not difficult to enforce and are 
discussed in Section \ref{Alternateboundaryconditions}.
The solution is determined only up to a constant vector field.
This leads to the introduction of a new space:
\begin{equation} \label{eq:defuHvhstar}
\uHvhstar := \lbrace \uvec{v}_h \in \uHvh \st \sum_{T \in \Th} \int_T \pna \uvec{v}_T = 0 \rbrace.
\end{equation}
This is the discrete counterpart of $\bvec{H}^1(\Omega) \cap L^2_0(\Omega)$.

Let $\mu$ be a constant viscosity, 
we define the symmetric bilinear form $\aSk(\uvec{v}_h, \uvec{w}_h) \in \uHvh \times \uHvh \rightarrow \Real$ on all
$\uvec{v}_h, \uvec{w}_h \in \uHvh$ by
\begin{equation} \label{eq:defStokesa}
\aSk(\uvec{v}_h, \uvec{w}_h) := \mu \spLt[h]{\uNah \uvec{v}_h}{\uNah \uvec{w}_h} .
\end{equation} 
We also define the bilinear form $\bSk(\uvec{v}_h, \ul{q}_h) \in \uHvh \times \uLsh \rightarrow \Real$ on all
$\uvec{v}_h \in \uHvh, \ul{q}_h \in \uLsh$ by
\begin{equation} \label{eq:defStokesb}
\bSk(\uvec{v}_h, \ul{q}_h) := \sum_{T \in \Th} \int_T \DT \uvec{v}_T q_T .
\end{equation}
Then we define the bilinear form
$\ASk((\uvec{v}_F,\ul{p}_h),(\uvec{w}_h,\ul{q}_h)) \in (\uHvhstar \times \uLsh) \times (\uHvhstar \times \uLsh) \rightarrow \Real$ by
\begin{equation} \label{eq:defStokes}
\ASk((\uvec{v}_h,\ul{p}_h),(\uvec{w}_h,\ul{q}_h)) = \aSk(\uvec{v}_h,\uvec{w}_h) - \bSk(\uvec{w}_h,\ul{p}_h) + \bSk(\uvec{v}_h,\ul{q}_h) .
\end{equation}

We define a suitable Sobolev-like norm on our discrete spaces such that
$\forall \uvec{v}_h \in \uHvh$, 
\begin{equation} \label{eq:defH1}
\normHNa{\uvec{v}_h} := \left ( \normNa[h]{\uvec{v}_h}^2 + \aSk(\uvec{v}_h, \uvec{v}_h) \right )^{1/2} .
\end{equation}
And for $f \in \bvec{L}^2(\Omega)$ we set
$\LSk \st \uHvhstar \rightarrow \Real$ such that $\forall \uvec{v}_h \in \uHvh$,
\begin{equation} \label{eq:defStokesL}
\LSk(\uvec{v}_h) := \sum_{T \in \Th} \int_T \pna \uvec{v}_T \cdot f .
\end{equation}

We define the discrete problem:\\
Find $(\uvec{v}_h, \ul{p}_h) \in \uHvhstar \times \uLsh$ such that for all $(\uvec{w}_h, \ul{q}_h) \in \uHvhstar \times \uLsh$,
\begin{equation} \label{eq:defStokespbdisc}
\ASk((\uvec{v}_h,\ul{p}_h),(\uvec{w}_h,\ul{q}_h)) = \LSk(\uvec{v}_h) .
\end{equation}
We show well-posedness in Lemma \ref{lemma:wellposedness}.

We consider the following Stokes problem:\\
Find $\bvec{u} \in \bvec{H}^2(\Omega) \cap \bvec{L}^2_0(\Omega)$, $p \in H^1_0(\Omega)$ such that
\begin{equation} \label{eq:defStokespbcon}
\begin{aligned}
- \mu \TLAPLACIAN \bvec{u} + \GRAD p =&\, f, \text{ on } \Omega,\\
\DIV \bvec{u} =&\, 0, \text{ on } \Omega,\\
\frac{\partial \bvec{u}}{\partial \nOmega} =&\, 0, \text{ on } \partial \Omega.
\end{aligned}
\end{equation}
Let $(\bvec{u},p)$ solves \eqref{eq:defStokespbcon} 
and let $(\uvec{v}_h,\ul{p}_h)$ solves \eqref{eq:defStokespbdisc}.
We assume that the continuous solutions $\bvec{u}$ and $p$ have the additional smoothness $\bvec{u} \in \bvec{H}^{k+2}(\Th)$ 
and $p \in H^{k+2}(\Th)$.
We deduce the following error estimate.
\begin{theorem}[Error estimate for Stokes] \label{th:Stokeserr}
Under the smoothness assumption on $\bvec{u}$ and $p$ it holds that 
\begin{equation}
\normHNa{\uvec{v}_h - \uIH[h] \bvec{u}} + \normLs{\ul{p}_h - \uILh p} \lesssim 
h^{k+1} \left ( \seminorm[\bvec{H}^{k+2}(\Th)]{\bvec{u}} + \seminorm[H^{k+1}(\Th)]{p} + \seminorm[H^{k+2}(\Th)]{p} \right ).
\end{equation}
\end{theorem}
\begin{proof}
The proof is a direct application of the third Strang lemma (see \cite{ThirdStrang}) to the estimates given by Lemma \ref{lemma:wellposedness} and 
\ref{lemma:consistencyerr}.
\end{proof}

\begin{lemma}[Well-posedness] \label{lemma:wellposedness}
For any $(\uvec{v}_h, \ul{p}_h) \in \uHvhstar \times \uLsh$ there is $(\uvec{w}_h, \ul{q}_h) \in \uHvhstar \times \uLsh$ 
such that $\normHNa{\uvec{w}_h} + \normLs{\ul{q}_h} \lesssim \normHNa{\uvec{v}_h} + \normLs{\ul{p}_h}$ and 
\begin{equation*}
\ASk((\uvec{v}_h,\ul{p}_h),(\uvec{w}_h,\ul{q}_h)) \gtrsim \normHNa{\uvec{v}_h}^2 + \normLs{\ul{p}_h}^2 .
\end{equation*}
\end{lemma}
\begin{proof}
Let $(\uvec{v}_h, \ul{p}_h) \in \uHvhstar \times \uLsh$, 
we have  
\begin{equation} \label{eq:wp1}
 \ASk((\uvec{v}_h,\ul{p}_h),(\uvec{v}_h,\ul{p}_h)) = \aSk(\uvec{v}_h, \uvec{v}_h) \gtrsim \normHNa{\uvec{v}_h}^2 ,
\end{equation}
where the last inequality comes from Lemma \ref{lemma:poincareH1}.
Moreover by Remark \ref{rem:rightinvdiv} there is $\uvec{w'}_h \in \uHvhstar$
such that $\Dh \uvec{w'}_h = -\ul{p}_h$ and
$\normHNa{\uvec{w'}_h} \lesssim \normLs{\ul{p}_h}$.
Hence
\begin{equation} \label{eq:wp2}
\begin{aligned}
\ASk((\uvec{v}_h,\ul{p}_h),(\uvec{w'}_h,0)) =&\, \aSk(\uvec{v}_h, \uvec{w'}_h) + \normLs{\ul{p}_h}^2 \\
\geq& -\frac{1}{2} \normHNa{\uvec{v}_h}^2 - \frac{1}{2} \normHNa{\uvec{w'}_h}^2 + \normLs{\ul{p}_h}^2 \\
\gtrsim& -\frac{1}{2} \normHNa{\uvec{v}_h}^2 + \frac{1}{2} \normLs{\ul{p}_h}^2.
\end{aligned}
\end{equation}
We conclude summing \eqref{eq:wp1} and \eqref{eq:wp2}.
\end{proof}

We define the consistency error 
$\Aerr \st \uHvhstar \times \uLsh \rightarrow \Real$ by
\begin{equation} \label{eq:defStokesAerr}
\Aerr((\uvec{w}_h,\ul{q}_h)) = \LSk(\uvec{w}_h) - \ASk((\uIH[h] \bvec{u}, \uILh p),(\uvec{w}_h,\ul{q}_h)) .
\end{equation}

\begin{lemma} \label{lemma:consistencyerr}
For all $\uvec{w}_h \in \uHvhstar$, $\ul{q}_h \in \uLsh$, it holds
\begin{equation*}
\begin{aligned}
\Aerr((\uvec{w}_h,\ul{q}_h)) \lesssim &
\, h^{k+1} \left ( \seminorm[\bvec{H}^{k+2}(\Th)]{\bvec{u}} + \seminorm[H^{k+1}(\Th)]{p} + \seminorm[H^{k+2}(\Th)]{p} \right ) \\
&\quad \times \left ( \normHNa{\uvec{v}_h} + \normLs{\ul{p}_h} \right ) .
\end{aligned}
\end{equation*}
\end{lemma}
\begin{proof}
Let $\uvec{w}_h \in \uHvhstar$, $\ul{q}_h \in \uLsh$, then
\begin{equation*}
\begin{aligned}
\Aerr((\uvec{w}_h,\ul{q}_h)) =& \sum_{T \in \Th} \int_T \pna \uvec{w}_T \cdot f - \mu \spNa{\uNaT \uIH \bvec{u}}{\uNaT \uvec{w}_T}\\ 
& \quad + \int_T \DT \uvec{w}_T \, \uILh p - \int_T \DT (\uIH \bvec{u})\, q_T\\
=& \sum_{T \in \Th} \int_T \pna \uvec{w}_T \cdot \GRAD p + \int_T \DT \uvec{w}_T \, \uILh p \\
& \quad - \mu \left ( \int_T \pna \uvec{w}_T \cdot \TLAPLACIAN \bvec{u} + \spNa{\uNaT \uIH \bvec{u}}{\uNaT \uvec{w}_T} \right ) \\
=&\, \AdjG(p \ID, \uvec{w}_h) - \stLt{\uNaT \uvec{w}_T}{\uIL (p \ID)} 
- \mu \AdjL(\bvec{u},\uvec{w}_h) \\
\leq&\; \seminorm{\AdjG(p \ID, \uvec{w}_h)} + \seminorm{\stLt{\uNaT \uvec{w}_T}{\uNaT \uvec{w}_T}}^{1/2} \\
& \quad + \seminorm{\stLt{\uIL (p \ID)}{\uIL (p \ID)}}^{1/2} + \mu \seminorm{\AdjL(\bvec{u}, \uvec{w}_h)} .
\end{aligned}
\end{equation*}
The second equality comes from $f = - \mu \TLAPLACIAN \bvec{u} + \GRAD p$, 
\eqref{eq:comm2} and $\DIV \bvec{u} = 0$,
and the third equality comes from \eqref{eq:defAdjG}, \eqref{eq:defAdjL} as well as:
\begin{equation*}
\int_T \DT \uvec{w}_T\, \lproj{k}{T} p = \int_T \Tr (\NaT \uvec{w}_T) \, p 
= \int_T \NaT \uvec{w}_T \tdot (p \ID) = \int_T \NaT \uvec{w}_T \tdot \RTbproj{k+1} (p \ID).
\end{equation*}
We conclude inferring the estimates of Theorem \ref{th:AdjConsG}, \ref{th:AdjConsL} 
and the consistency \eqref{eq:consistencystLt}.
\end{proof}

\section{Alternative boundary conditions.} \label{Alternateboundaryconditions}
In this section we show how to extend the results of Section \ref{Stokes} to Dirichlet boundary conditions on $\uHvh$.
This is useful for common condition such as the no slip condition or forced inlet condition
and does not require much change.
\subsection{Dirichlet boundary conditions.}
We introduce the space 
\begin{equation}
\uHvhzero := \lbrace \uvec{v}_h \in \uHvh \st \uvec{v}_F \equiv \uvec{0}, \, \forall F \in \Fh, F \subset \partial \Omega \rbrace. 
\end{equation}
Since the pressure is only defined up to a constant value, we introduce the natural space:
\begin{equation}
\uLshstar := \lbrace \ul{q}_h \in \uLsh \st \sum_{T \in \Th} \int_T q_T = 0 \rbrace .
\end{equation}
Then we define the bilinear form:
$\ASk((\uvec{v}_h,\ul{p}_h),(\uvec{w}_h,\ul{q}_h)) \in (\uHvhzero \times \uLshstar) \times (\uHvhzero \times \uLshstar) \rightarrow \Real$ by
\begin{equation}
\ASk((\uvec{v}_h,\ul{p}_h),(\uvec{w}_h,\ul{q}_h)) := \aSk(\uvec{v}_h,\uvec{w}_h) - \bSk(\uvec{w}_h,\ul{p}_h) + \bSk(\uvec{v}_h,\ul{q}_h) .
\end{equation}
With $\aSk$ and $\bSk$ defined by \eqref{eq:defStokesa}, \eqref{eq:defStokesb}, 
we also keep the same definition \eqref{eq:defStokesL} of the source term $\LSk$.
So the discrete problem is:\\
Find $(\uvec{v}_h, \ul{p}_h) \in \uHvhzero \times \uLshstar$ such that for all $(\uvec{w}_h, \ul{q}_h) \in \uHvhzero \times \uLshstar$,
\begin{equation} \label{eq:defStokesDirichlet}
\ASk((\uvec{v}_h,\ul{p}_h),(\uvec{w}_h,\ul{q}_h)) = \LSk(\uvec{v}_h) .
\end{equation}
The continuous Stokes problem becomes:\\
Find $\bvec{u} \in \bvec{H}^1_0(\Omega) \cap \bvec{H}^2(\Omega)$, $p \in H^1(\Omega) \cap L^2_0(\Omega)$ such that
\begin{equation} \label{eq:defStokescontdirichlet}
\begin{aligned}
- \mu \TLAPLACIAN \bvec{u} + \GRAD p =& f, \text{ on } \Omega,\\
\DIV \bvec{u} =& 0, \text{ on } \Omega.
\end{aligned}
\end{equation}

\begin{theorem} \label{th:StokeserrDirichlet}
Under the same assumptions as those of Theorem \ref{th:Stokeserr},
the problem \eqref{eq:defStokesDirichlet} is well-posed and converges toward the continuous solution 
of problem \eqref{eq:defStokescontdirichlet} with the same error estimate as in Theorem \ref{th:Stokeserr}. 
\end{theorem}
\begin{proof}
The proof is similar to the proof of Theorem \ref{th:Stokeserr}.
We need a suitable version of Lemma \ref{lemma:rightinvdiv}, 
and we can expect $\uvec{v}_h \in \uHvhzero$ if $\ul{p}_h \in \uLshstar$ 
by substituting the use of Theorem \ref{th:divlift} by Theorem \ref{th:divlift0} in the proof of Lemma \ref{lemma:rightinvdiv} (see Remark \ref{rem:adaptinvdiv}).
The consistency errors \ref{th:AdjConsG} and \ref{th:AdjConsL} require respectively $\bvec{W} \in \bvec{H}^1_0(\Omega)$ 
and $\TGRAD \bvec{w} \cdot \nOmega = 0$.
However we can check that this is only used to get 
\eqref{eq:adjgradzerosum} and \eqref{eq:adjlapzerosum} 
both of which also hold if $\uvec{v}_h \in \uHvhzero$ instead,
so that $\uvec{v}_F \equiv \uvec{0}$, $\forall F \subset \partial \Omega$.
Finally, we relied on Lemma \ref{lemma:poincareH1} to show that $\ASk$ is weakly coercive.
This too can be readily adapted if we use \cite[Lemma~2.15]{hho} instead of \cite[Theorem~6.5]{hho} in the proof of Lemma \ref{lemma:poincareH1}.
With these three results we can proceed exactly in the same manner as we did for Theorem \ref{th:Stokeserr}.
\end{proof}

\subsection{Mixed boundary conditions.}
We can also use Dirichlet conditions on a subset of the boundary and Neumann conditions elsewhere.
Let $\Gamma_D$ be a relatively closed subset of $\partial \Omega$ with a non-zero measure and 
$\Gamma_N = \partial \Omega \setminus \Gamma_D$.
Furthermore assume that each boundary face 
$\partial \Omega \supset F \in \Fh$
is either contained in $\Gamma_N$ or in $\Gamma_D$ but not in both 
(i.e.\ either $\overline{F} \subset \Gamma_D = \emptyset$ or $F \cap \Gamma_D = \emptyset$).
We also require that $\Gamma_D$ and $\Gamma_N$ contain at least one face 
(else we degenerate to pure Neumann or pure Dirichlet).
The continuous Stokes problem is:\\
Find $\bvec{u} \in \bvec{H}^2(\Omega)$, $p \in H^1(\Omega)$ such that
\begin{equation}
\begin{aligned}
- \mu \TLAPLACIAN \bvec{u} + \GRAD p =&\, f, \text{ on } \Omega,\\
\DIV \bvec{u} =&\, 0, \text{ on } \Omega,\\
\bvec{u} =&\, 0, \text{ on } \Gamma_D,\\
\frac{\partial \bvec{u}}{\partial \nOmega} =&\, 0, \text{ on } \Gamma_N,\\
p =&\, 0, \text{ on } \Gamma_N.
\end{aligned}
\end{equation}
We introduce the discrete space $\uHvhzeroD := \lbrace \bvec{v}_h \in \uHvh \st \uvec{v}_F \equiv \uvec{0}, \, \forall F \in \Gamma_D \rbrace$
and define: 
$\ASk((\uvec{v}_h,\ul{p}_h),(\uvec{w}_h,\ul{q}_h)) \in (\uHvhzeroD \times \uLsh) \times (\uHvhzeroD \times \uLsh) \rightarrow \Real$ by
\begin{equation}
\ASk((\uvec{v}_h,\ul{p}_h),(\uvec{w}_h,\ul{q}_h)) = \aSk(\uvec{v}_h,\uvec{w}_h) - \bSk(\uvec{w}_h,\ul{p}_h) + \bSk(\uvec{v}_h,\ul{q}_h) .
\end{equation}
The discrete problem becomes:\\
Find $(\uvec{v}_h, \ul{p}_h) \in \uHvhzeroD \times \uLsh$ such that for all $(\uvec{w}_h, \ul{q}_h) \in \uHvhzeroD \times \uLsh$,
\begin{equation} \label{eq:defStokesMixed}
\ASk((\uvec{v}_h,\ul{p}_h),(\uvec{w}_h,\ul{q}_h)) = \LSk(\uvec{v}_h).
\end{equation}

\section{Numerical validation.} \label{Numericaltests}

The Stokes complex was implemented within the HArDCore C++ framework (see \url{https://github.com/jdroniou/HArDCore}),
using the linear algebra facilities from the Eigen3 library (see \url{https://eigen.tuxfamily.org})
and the linear solver from the Portable, Extensible Toolkit for Scientific Computation PETSc (see \url{https://petsc.org}).
An implementation of the spaces and operators defined in this paper as well as a Stokes solver can be found at \url{https://github.com/mlhanot/HArDCore3D-Stokes}.
The numerical validation is done with a constant viscosity $\mu = 1$.
We measure the rate of convergence toward an exact solution for various polynomial degrees $k \in \lbrace 0, 1, 2, 3 \rbrace$.
The error is computed as
\begin{equation*}
\frac{\normHNa{\uvec{u}_h - \uIHh \bvec{u}} + \normLs[h]{\ul{p}_h - \uILh p}}
{\normHNa{\uIHh \bvec{u}} + \normLs[h]{\uILh p}} .
\end{equation*}
We expect the error to decrease at a rate $\bigO(h^{k+1})$ thanks to Theorem \ref{th:Stokeserr} and \ref{th:StokeserrDirichlet}.
We also validated the $2$-dimensional variation detailed in appendix \ref{2DComplex}.
These tests are done on various mesh sequences.
One mesh of each sequence is shown in Figure \ref{fig:meshes2D} for the $2$-dimensional cases,
and a cross section of the $3$-dimensional meshes is shown in Figure \ref{fig:meshes3D}.
The results are given in Figure \ref{fig:convrate2D} in $2$ dimensions and in Figure \ref{fig:convrate3D} in $3$ dimensions.
We always obtain results consistent with the theory. 
In $2$ dimensions, we can see that the method is robust and the convergence is barely impacted by the various features of the meshes.
In $3$ dimensions, arbitrary polyhedra can be much wilder than arbitrary polygons. 
We can see that, while the expected convergence rate is asymptotically obtained,
some combinations of degree and mesh exhibit better properties.
On the coarsest meshes with the lowest polynomial degree we notice a drop in the convergence rate.
These meshes are too coarse for the solution sought and the problem disapears when refining or increasing the polynomial degree.

\begin{figure}
\begin{center}
\begin{minipage}[b]{0.3\columnwidth}
\includegraphics[width=0.9\columnwidth,height=0.2\textheight,keepaspectratio]{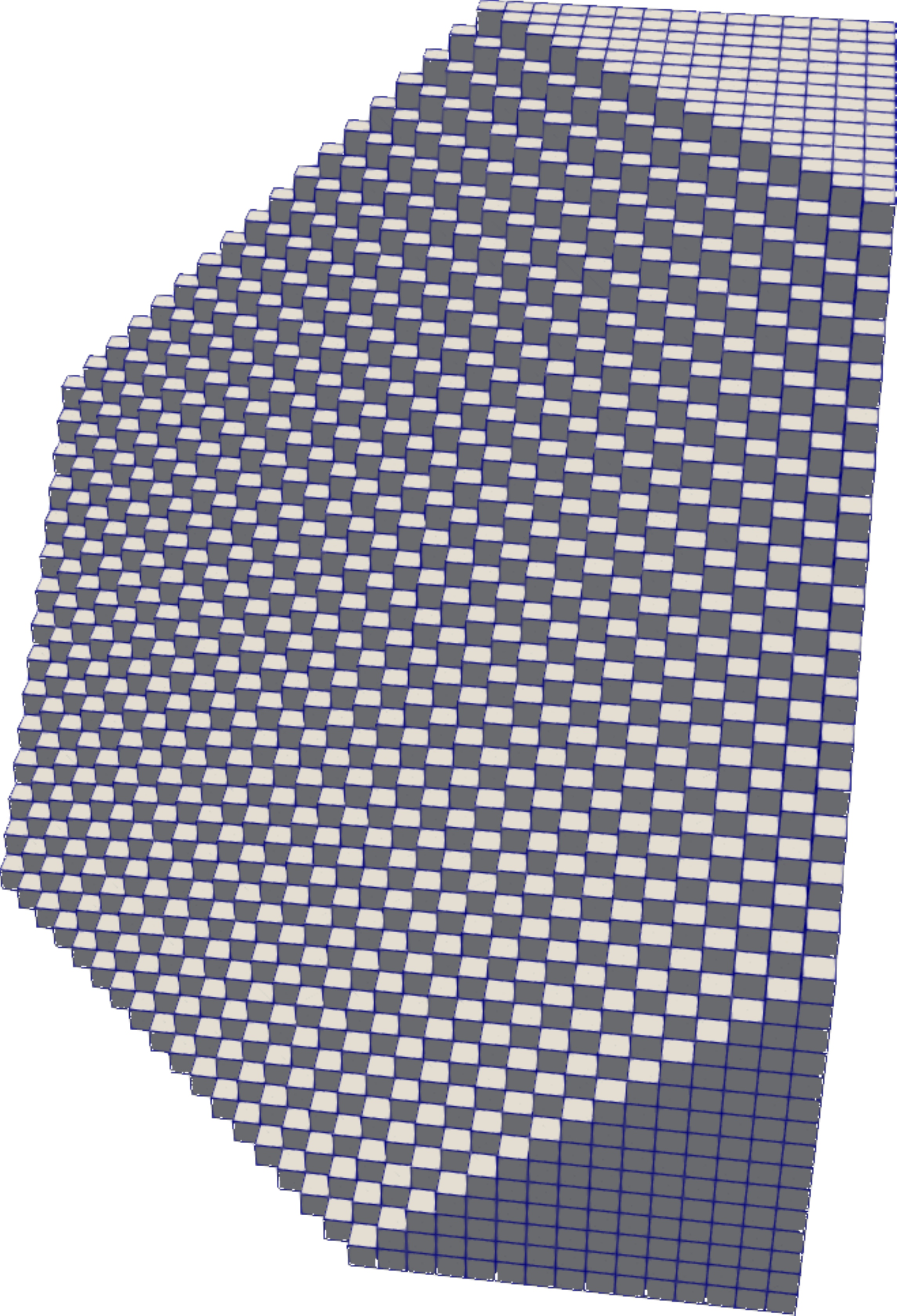}
\subcaption{"Cube" mesh.}
\end{minipage}
\hfill
\begin{minipage}[b]{0.3\columnwidth}
\includegraphics[width=0.9\columnwidth,height=0.2\textheight,keepaspectratio]{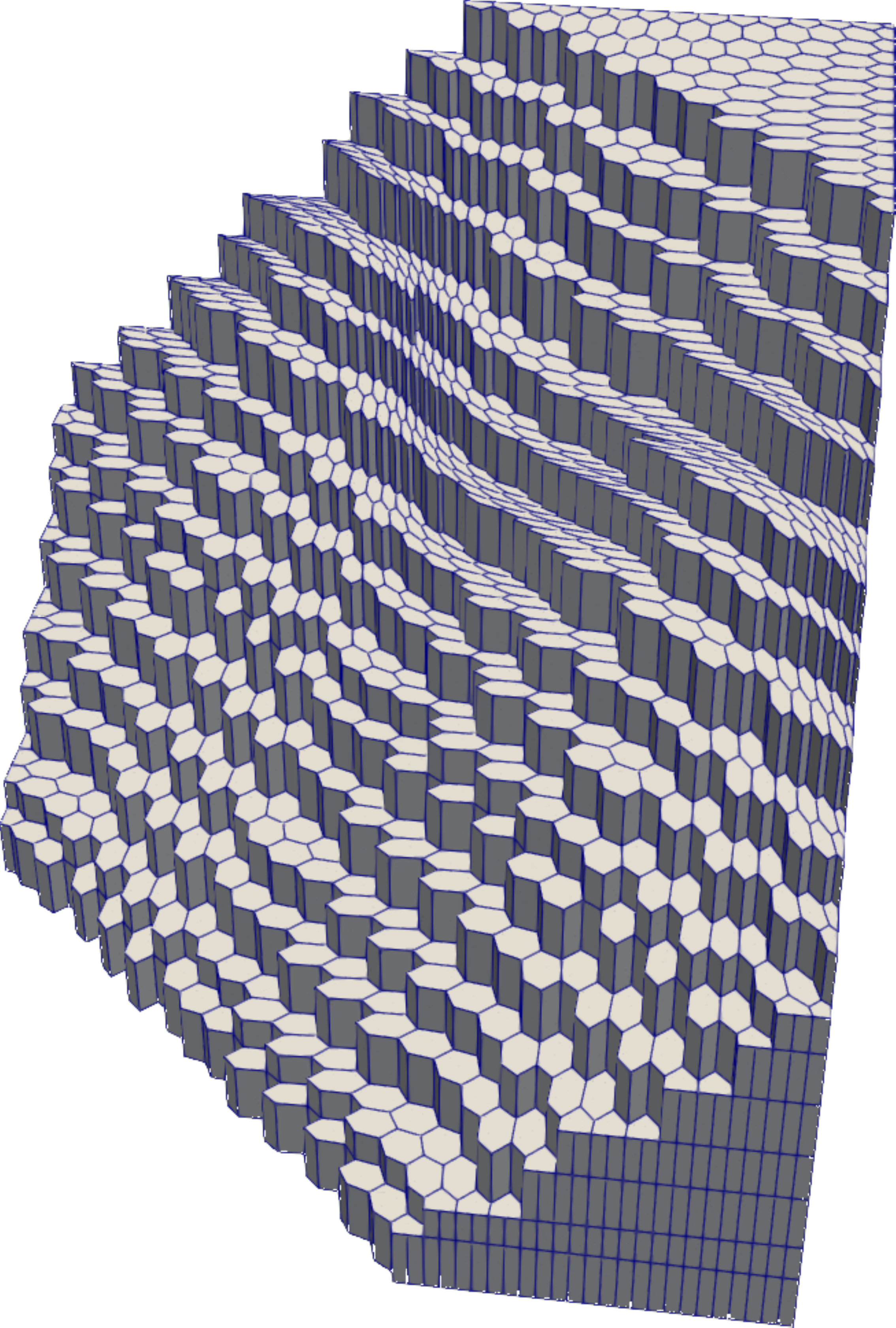}
\subcaption{"Prysmatic" mesh.}
\end{minipage}
\hfill
\begin{minipage}[b]{0.3\columnwidth}
\includegraphics[width=0.9\columnwidth,height=0.2\textheight,keepaspectratio]{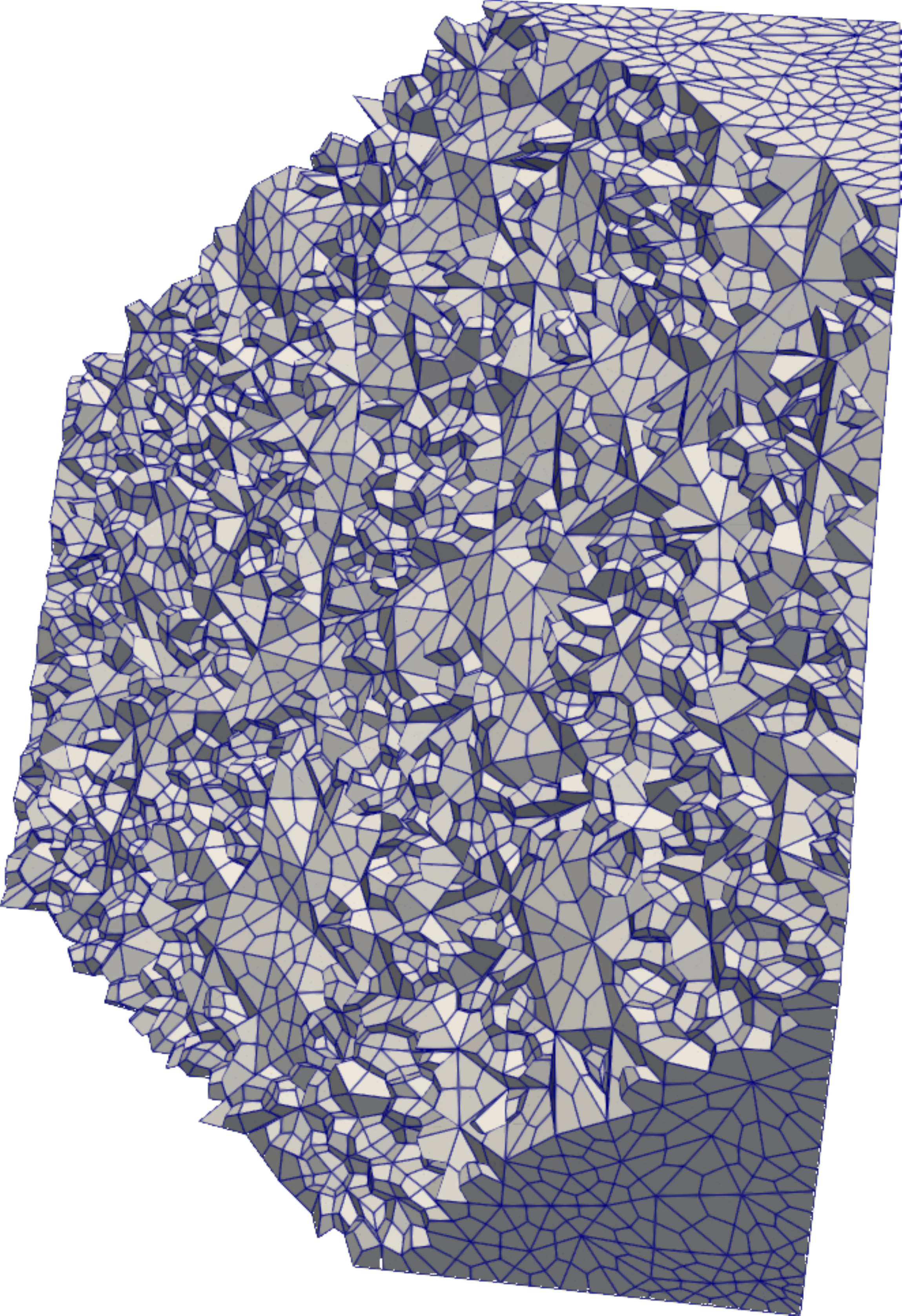}
\subcaption{"Hexahedra" mesh.}
\end{minipage}
\\
\medskip
\begin{minipage}[b]{0.3\columnwidth}
\includegraphics[width=0.9\columnwidth,height=0.2\textheight,keepaspectratio]{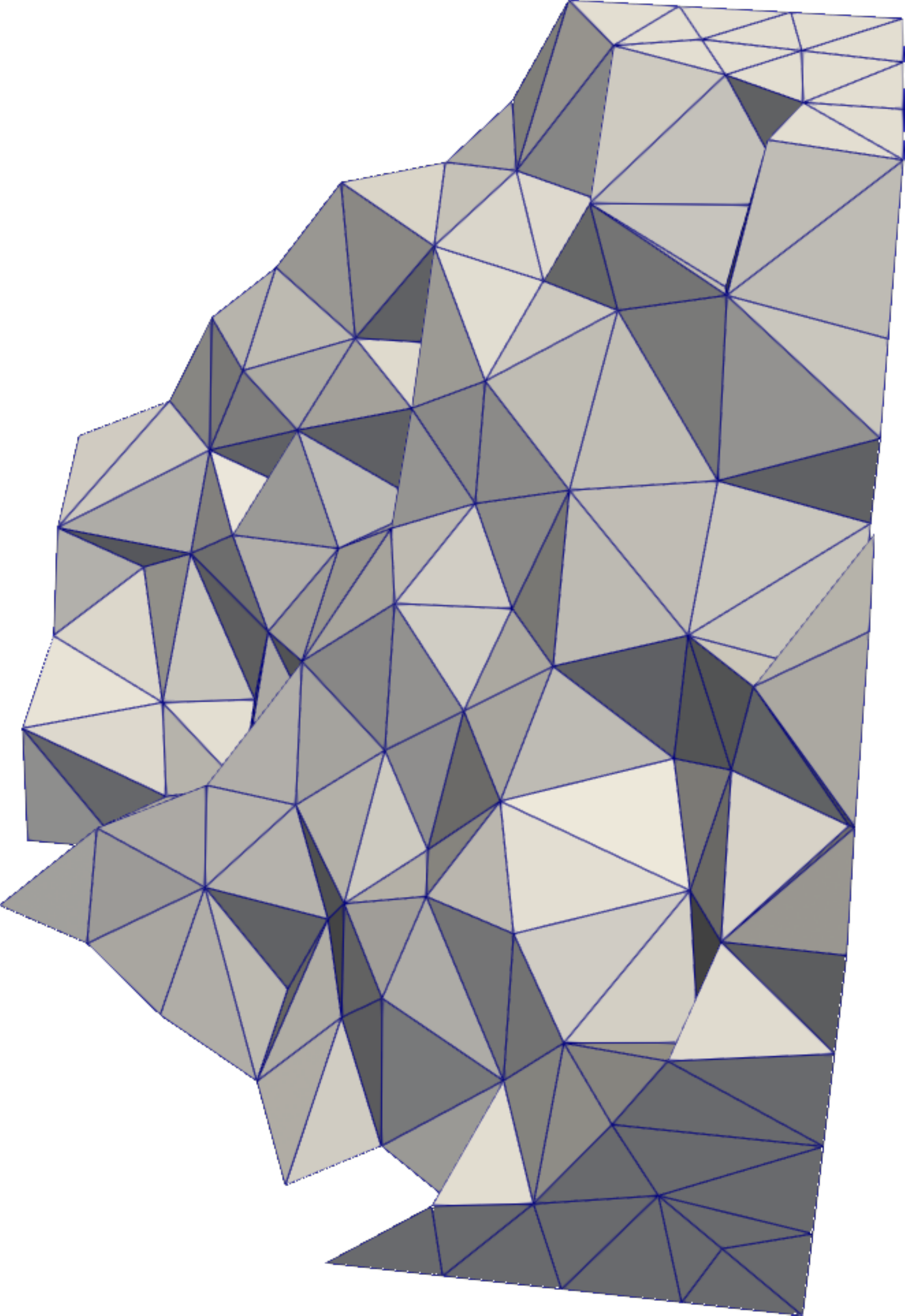}
\subcaption{"Tetgen cube" mesh.}
\end{minipage}
\hfill
\begin{minipage}[b]{0.3\columnwidth}
\includegraphics[width=0.9\columnwidth,height=0.2\textheight,keepaspectratio]{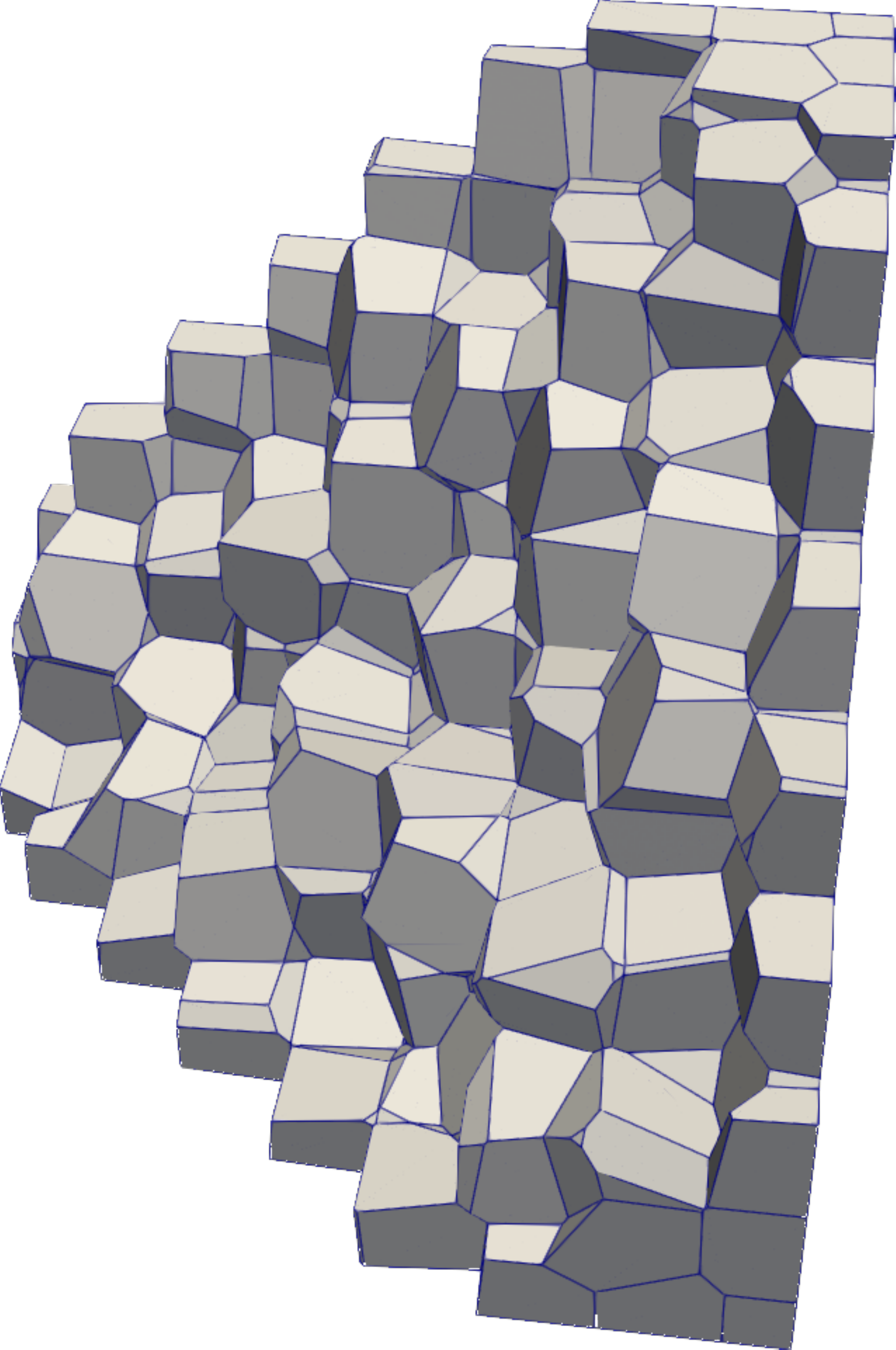}
\subcaption{"Voro" mesh.}
\end{minipage}
\hfill
\begin{minipage}[b]{0.3\columnwidth}
\includegraphics[width=0.9\columnwidth,height=0.2\textheight,keepaspectratio]{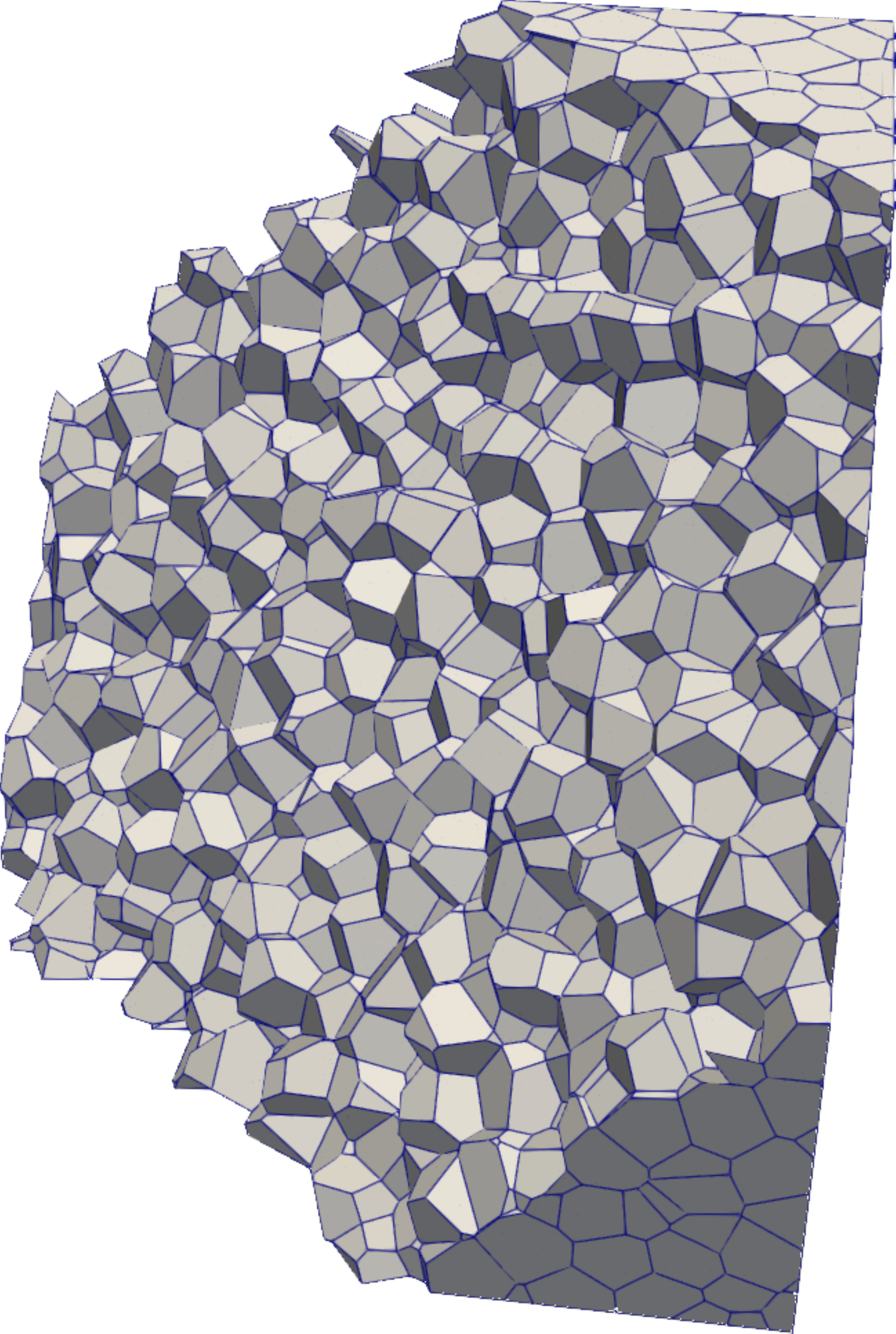}
\subcaption{"Voro tets" mesh.}
\end{minipage}
\end{center}
\caption{Families of mesh used in $3$ dimensions, sliced for visualization.}
\label{fig:meshes3D}
\end{figure}

\begin{figure}
\begin{center}
\begin{minipage}[b]{0.3\columnwidth}
\includegraphics[width=0.9\columnwidth,height=0.15\textheight,keepaspectratio]{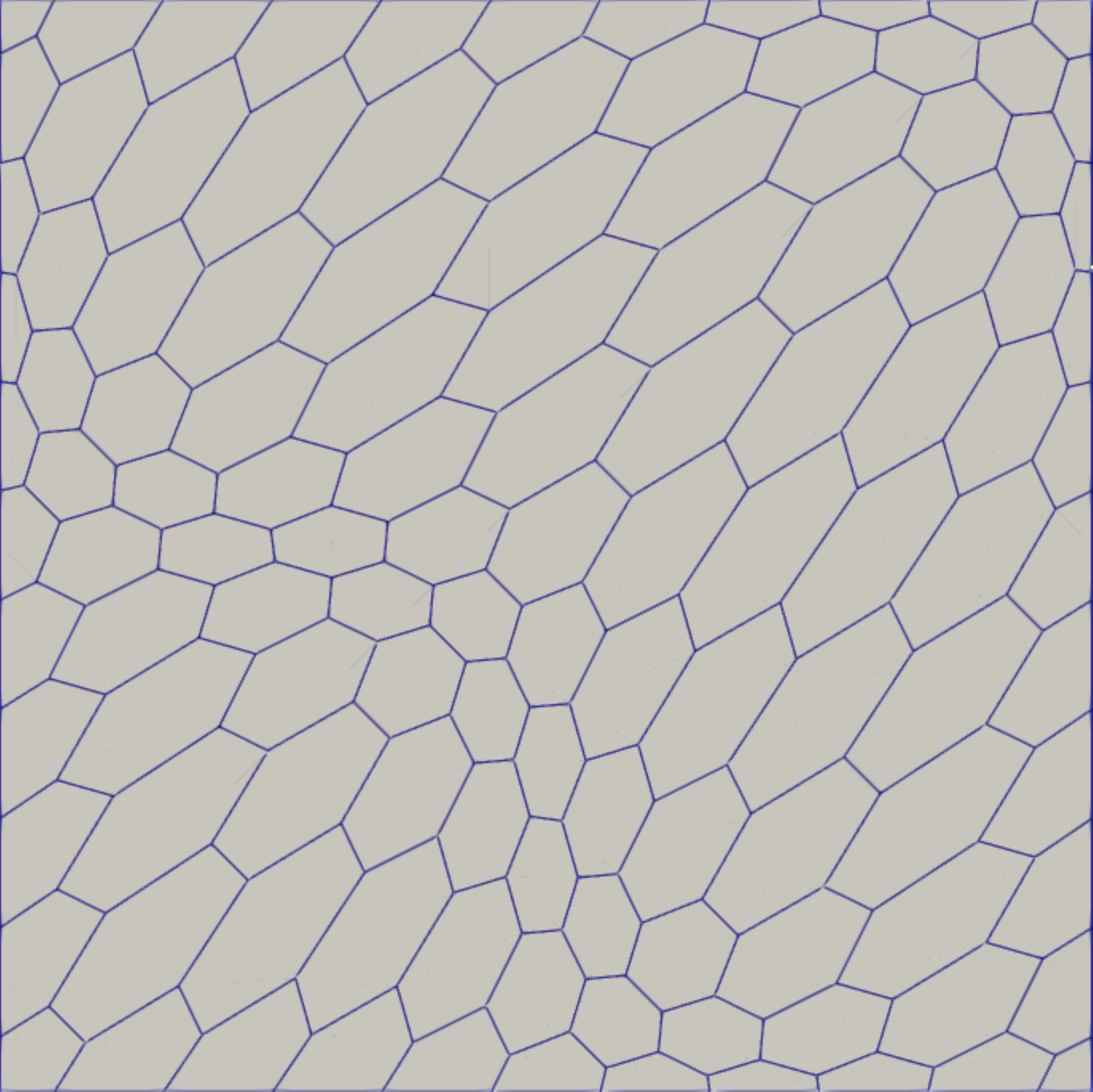}
\subcaption{"Hexa" mesh.}
\end{minipage}
\hfill
\begin{minipage}[b]{0.3\columnwidth}
\includegraphics[width=0.9\columnwidth,height=0.15\textheight,keepaspectratio]{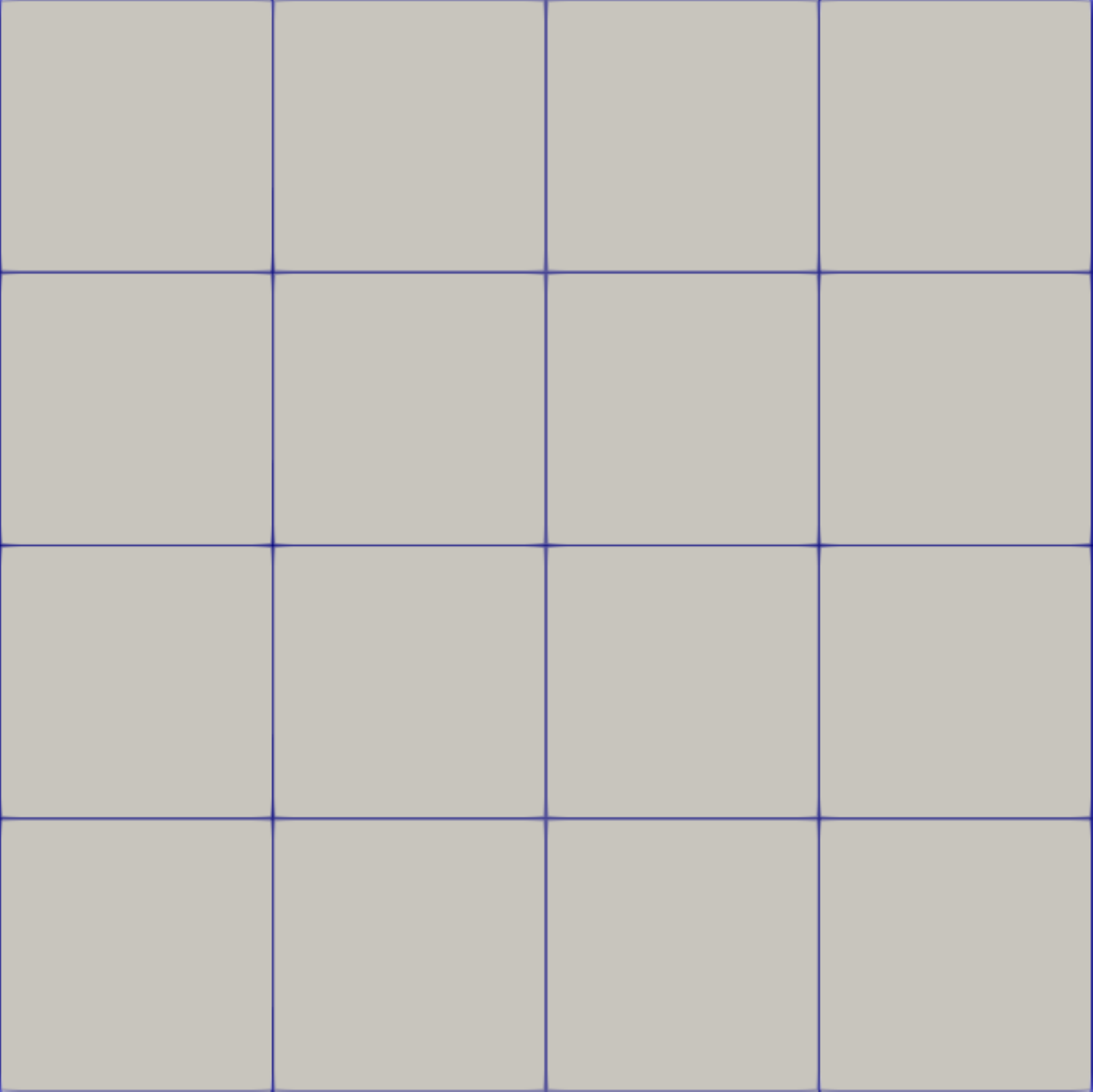}
\subcaption{"Square" mesh.}
\end{minipage}
\hfill
\begin{minipage}[b]{0.3\columnwidth}
\includegraphics[width=0.9\columnwidth,height=0.15\textheight,keepaspectratio]{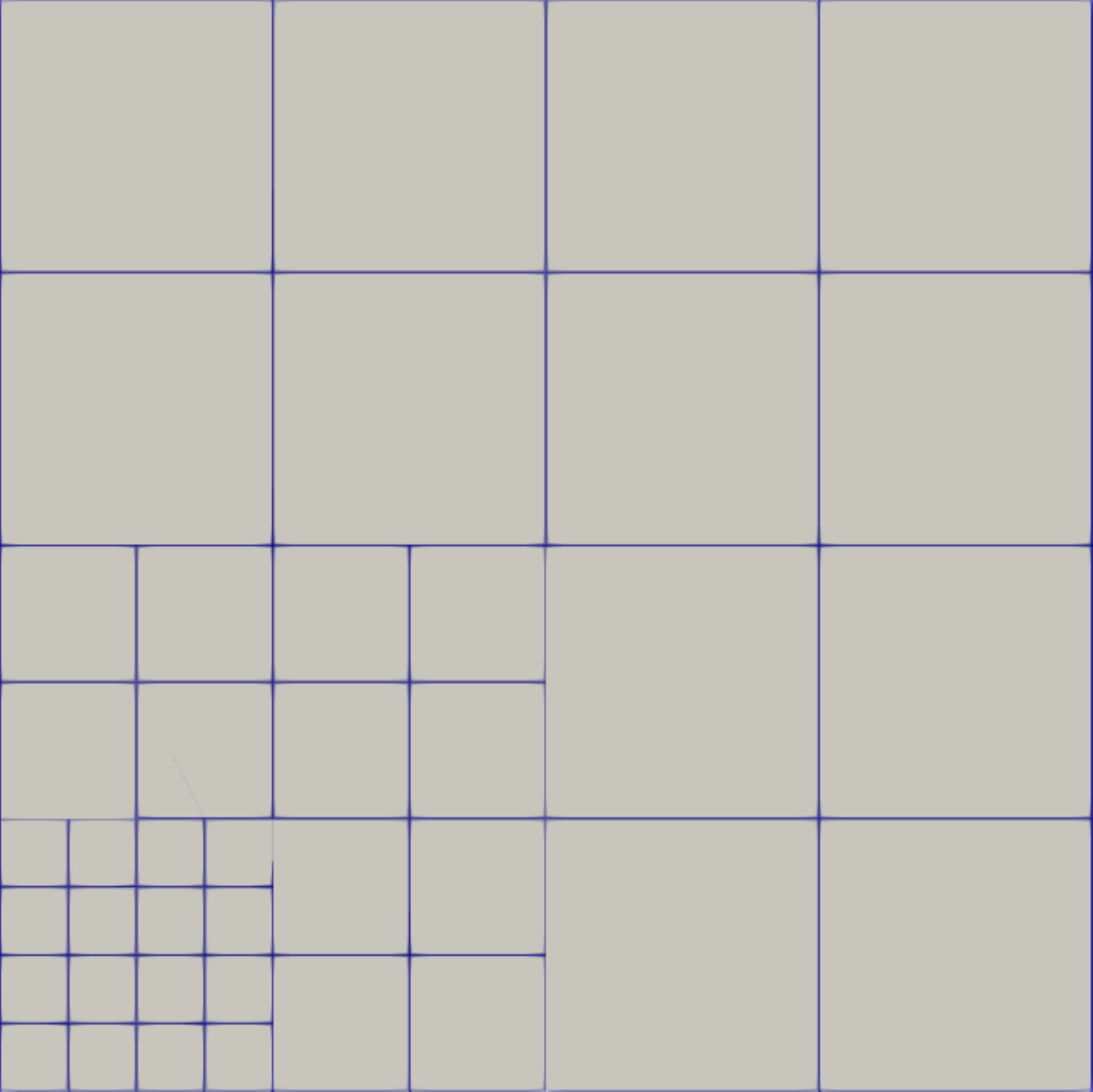}
\subcaption{"Square 2" mesh.}
\end{minipage}
\\
\medskip
\begin{minipage}[b]{0.3\columnwidth}
\includegraphics[width=0.9\columnwidth,height=0.15\textheight,keepaspectratio]{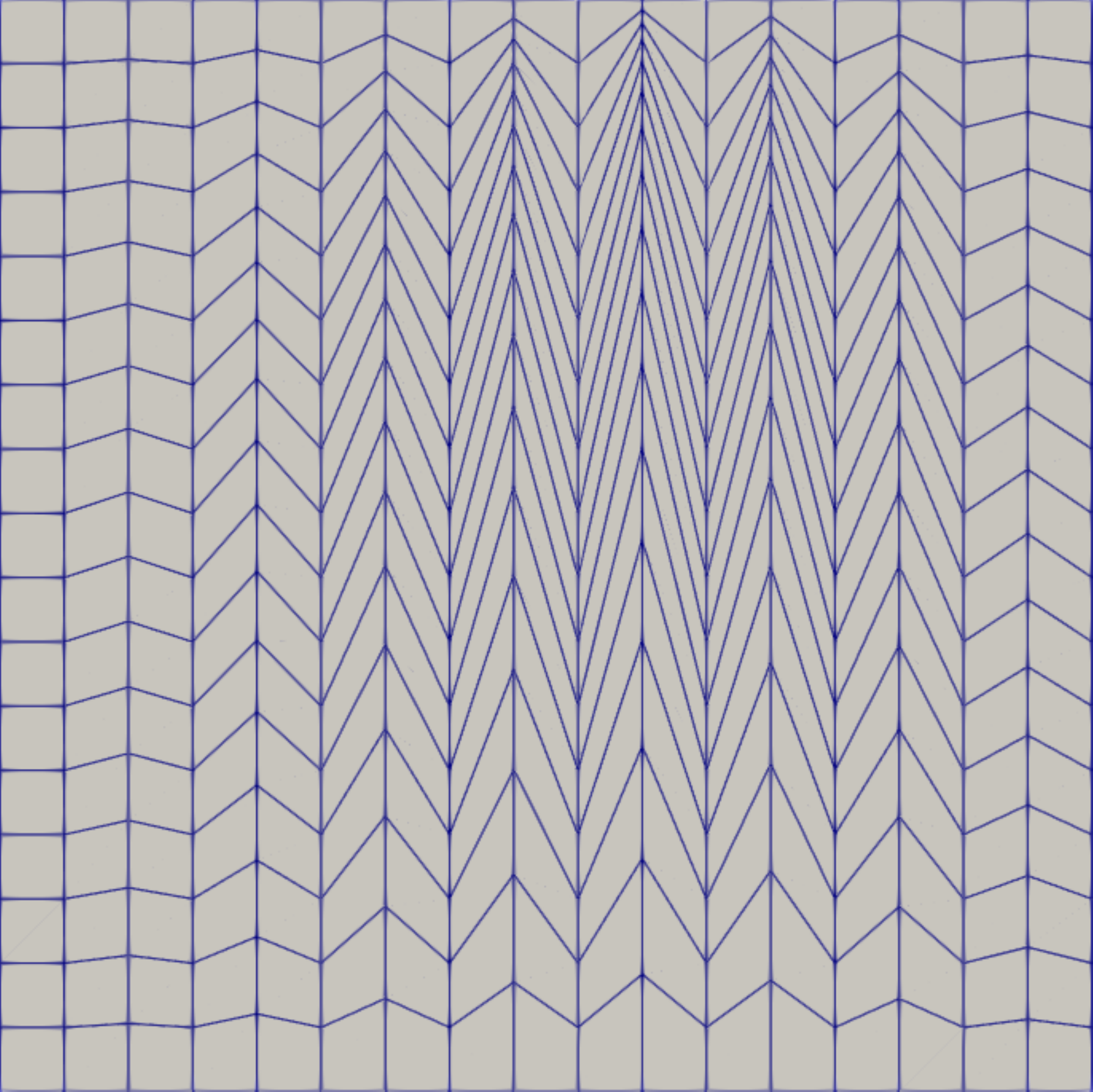}
\subcaption{"Tilted" mesh.}
\end{minipage}
\hfill
\begin{minipage}[b]{0.3\columnwidth}
\includegraphics[width=0.9\columnwidth,height=0.15\textheight,keepaspectratio]{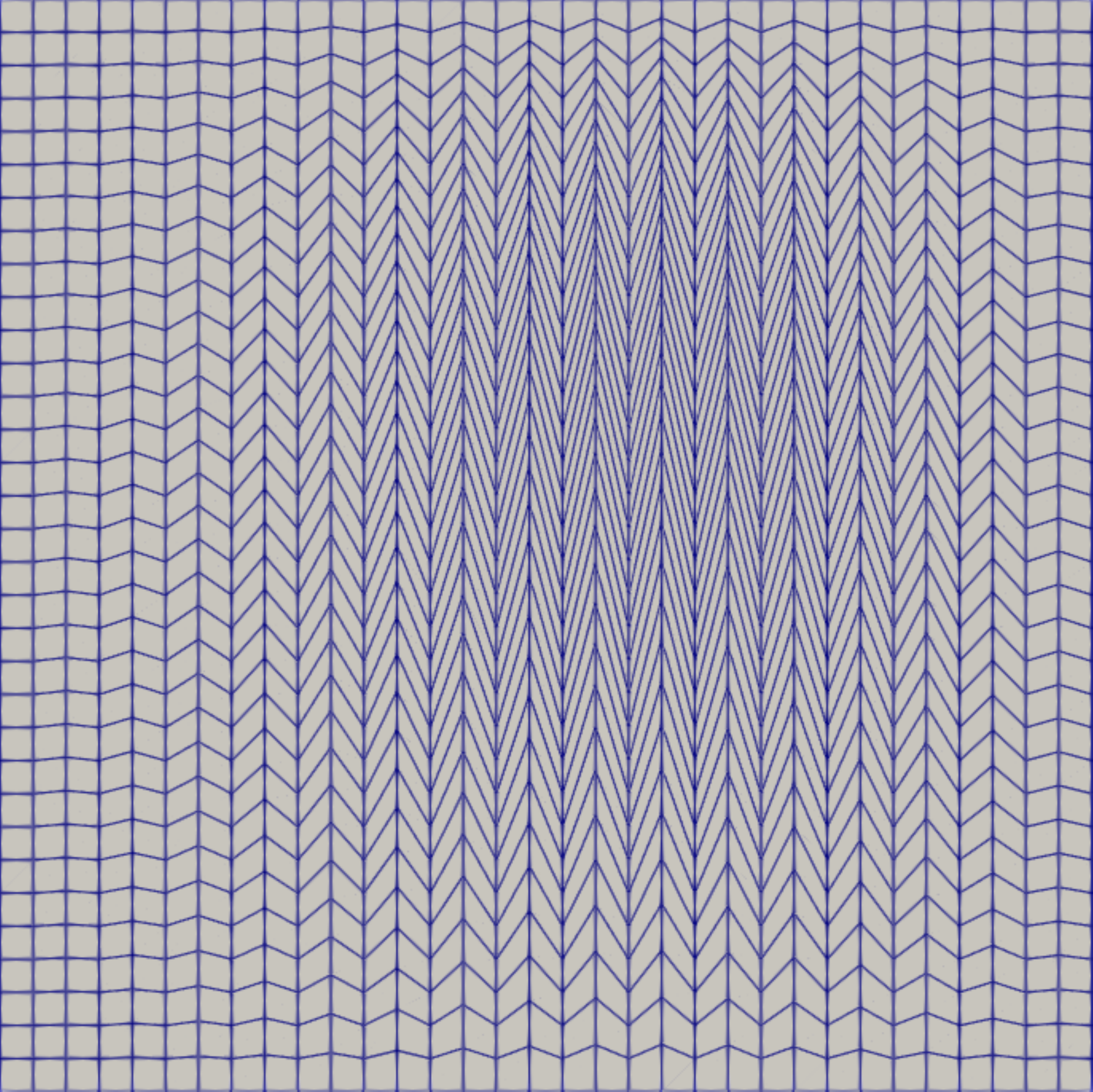}
\subcaption{"Tilted 2" mesh.}
\end{minipage}
\hfill
\begin{minipage}[b]{0.3\columnwidth}
\includegraphics[width=0.9\columnwidth,height=0.15\textheight,keepaspectratio]{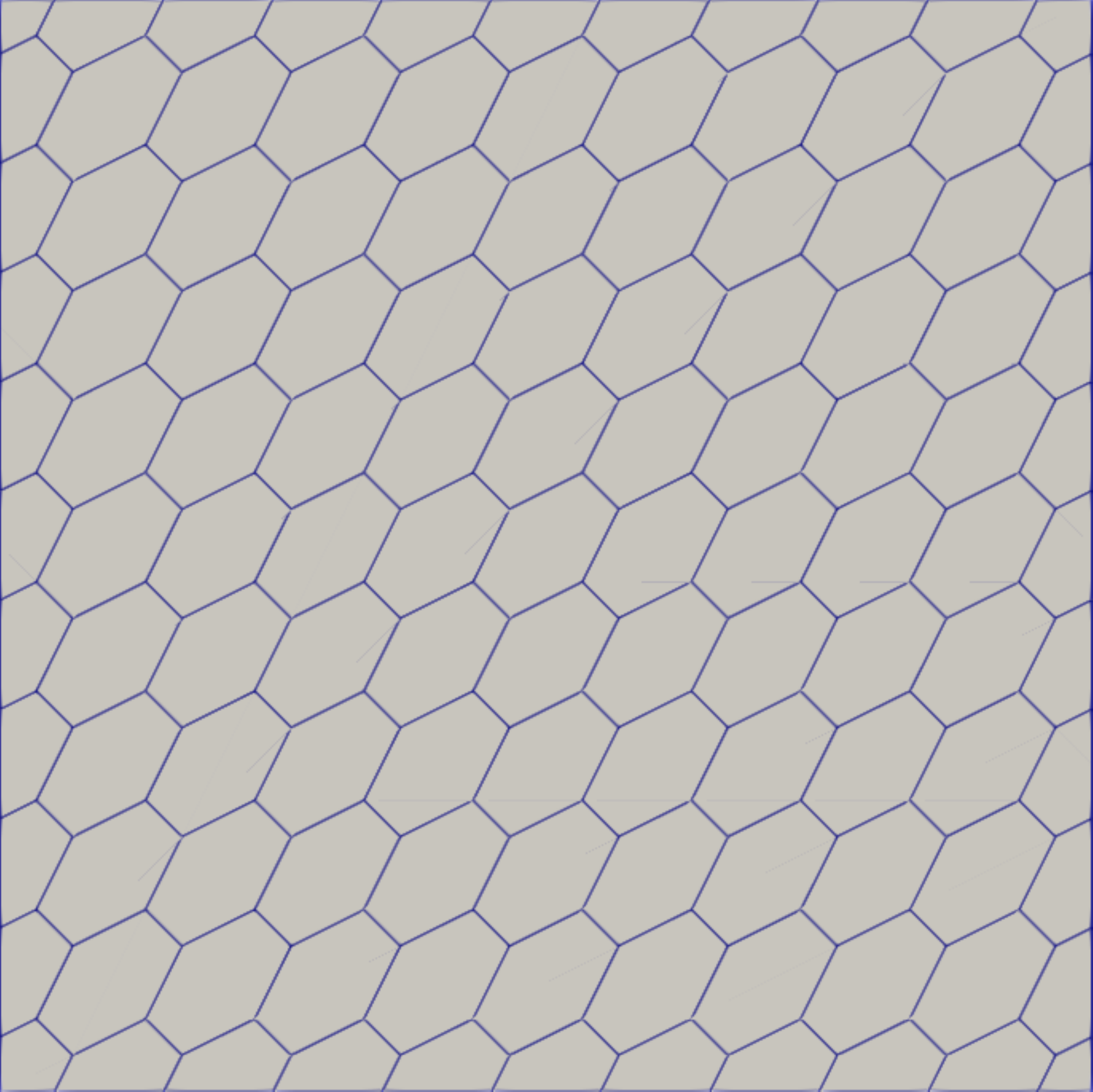}
\subcaption{"Hexa anisotropic" mesh.}
\end{minipage}
\end{center}
\caption{Families of mesh used in $2$ dimensions.}
\label{fig:meshes2D}
\end{figure}

\begin{figure}
\begin{center}
\begin{tikzpicture} 
    \begin{axis}[
    hide axis,
    xmin=0,
    xmax=1,
    ymin=0,
    ymax=1,
    xticklabels={,,},
    yticklabels={,,},
    height=50,
    legend style={legend columns={4}}
    ]
    \addlegendimage{mark=+, style=solid, color=blue}
    \addlegendentry{$k = 0$};
    \addlegendimage{mark=triangle, style=solid, color=red} 
    \addlegendentry{$k = 1$};
    \addlegendimage{mark=square, style=solid, color=brown} 
    \addlegendentry{$k = 2$};
    \addlegendimage{mark=pentagon, style=solid, color=darkgray} 
    \addlegendentry{$k = 3$};
    \end{axis}
\end{tikzpicture}
\\
\medskip
\begin{minipage}[b]{0.45\columnwidth}
\begin{tikzpicture}[scale=0.9]
\begin{loglogaxis}
\addplot +[mark=+, style=solid, color=blue] table[x=meshsize,y=err] {data/2D/0_k0.dat};
\addplot +[mark=triangle, style=solid, color=red] table[x=meshsize,y=err] {data/2D/0_k1.dat};
\addplot +[mark=square, style=solid, color=brown] table[x=meshsize,y=err] {data/2D/0_k2.dat};
\addplot +[mark=pentagon, style=solid, color=darkgray] table[x=meshsize,y=err] {data/2D/0_k3.dat};
\logLogSlopeTriangle{0.90}{0.4}{0.1}{1}{blue};
\logLogSlopeTriangle{0.90}{0.4}{0.1}{2}{red};
\logLogSlopeTriangle{0.90}{0.4}{0.1}{3}{brown};
\logLogSlopeTriangle{0.90}{0.4}{0.1}{4}{darkgray};
\end{loglogaxis}
\end{tikzpicture}
\subcaption{"Hexa" mesh.}
\end{minipage}
\hfill
\begin{minipage}[b]{0.45\columnwidth}
\begin{tikzpicture}[scale=0.9]
\begin{loglogaxis}
\addplot +[mark=+, style=solid, color=blue] table[x=meshsize,y=err] {data/2D/1_k0.dat};
\addplot +[mark=triangle, style=solid, color=red] table[x=meshsize,y=err] {data/2D/1_k1.dat};
\addplot +[mark=square, style=solid, color=brown] table[x=meshsize,y=err] {data/2D/1_k2.dat};
\addplot +[mark=pentagon, style=solid, color=darkgray] table[x=meshsize,y=err] {data/2D/1_k3.dat};
\logLogSlopeTriangle{0.90}{0.4}{0.1}{1}{blue};
\logLogSlopeTriangle{0.90}{0.4}{0.1}{2}{red};
\logLogSlopeTriangle{0.90}{0.4}{0.1}{3}{brown};
\logLogSlopeTriangle{0.90}{0.4}{0.1}{4}{darkgray};
\end{loglogaxis}
\end{tikzpicture}
\subcaption{"Square" mesh.}
\end{minipage}
\hfill
\begin{minipage}[b]{0.1\columnwidth}
\end{minipage}
\\
\medskip
\begin{minipage}[b]{0.45\columnwidth}
\begin{tikzpicture}[scale=0.9]
\begin{loglogaxis}
\addplot +[mark=+, style=solid, color=blue] table[x=meshsize,y=err] {data/2D/2_k0.dat};
\addplot +[mark=triangle, style=solid, color=red] table[x=meshsize,y=err] {data/2D/2_k1.dat};
\addplot +[mark=square, style=solid, color=brown] table[x=meshsize,y=err] {data/2D/2_k2.dat};
\addplot +[mark=pentagon, style=solid, color=darkgray] table[x=meshsize,y=err] {data/2D/2_k3.dat};
\logLogSlopeTriangle{0.90}{0.4}{0.1}{1}{blue};
\logLogSlopeTriangle{0.90}{0.4}{0.1}{2}{red};
\logLogSlopeTriangle{0.90}{0.4}{0.1}{3}{brown};
\logLogSlopeTriangle{0.90}{0.4}{0.1}{4}{darkgray};
\end{loglogaxis}
\end{tikzpicture}
\subcaption{"Square 2" mesh.}
\end{minipage}
\hfill
\begin{minipage}[b]{0.45\columnwidth}
\begin{tikzpicture}[scale=0.9]
\begin{loglogaxis}
\addplot +[mark=+, style=solid, color=blue] table[x=meshsize,y=err] {data/2D/3_k0.dat};
\addplot +[mark=triangle, style=solid, color=red] table[x=meshsize,y=err] {data/2D/3_k1.dat};
\addplot +[mark=square, style=solid, color=brown] table[x=meshsize,y=err] {data/2D/3_k2.dat};
\addplot +[mark=pentagon, style=solid, color=darkgray] table[x=meshsize,y=err] {data/2D/3_k3.dat};
\logLogSlopeTriangle{0.90}{0.4}{0.1}{1}{blue};
\logLogSlopeTriangle{0.90}{0.4}{0.1}{2}{red};
\logLogSlopeTriangle{0.90}{0.4}{0.1}{3}{brown};
\logLogSlopeTriangle{0.90}{0.4}{0.1}{4}{darkgray};
\end{loglogaxis}
\end{tikzpicture}
\subcaption{"Tilted" mesh.}
\end{minipage}
\hfill
\begin{minipage}[b]{0.1\columnwidth}
\end{minipage}
\\
\medskip
\begin{minipage}[b]{0.45\columnwidth}
\begin{tikzpicture}[scale=0.9]
\begin{loglogaxis}
\addplot +[mark=+, style=solid, color=blue] table[x=meshsize,y=err] {data/2D/4_k0.dat};
\addplot +[mark=triangle, style=solid, color=red] table[x=meshsize,y=err] {data/2D/4_k1.dat};
\addplot +[mark=square, style=solid, color=brown] table[x=meshsize,y=err] {data/2D/4_k2.dat};
\addplot +[mark=pentagon, style=solid, color=darkgray] table[x=meshsize,y=err] {data/2D/4_k3.dat};
\logLogSlopeTriangle{0.90}{0.4}{0.1}{1}{blue};
\logLogSlopeTriangle{0.90}{0.4}{0.1}{2}{red};
\logLogSlopeTriangle{0.90}{0.4}{0.1}{3}{brown};
\logLogSlopeTriangle{0.90}{0.4}{0.1}{4}{darkgray};
\end{loglogaxis}
\end{tikzpicture}
\subcaption{"Tilted 2" mesh.}
\end{minipage}
\hfill
\begin{minipage}[b]{0.45\columnwidth}
\begin{tikzpicture}[scale=0.9]
\begin{loglogaxis}
\addplot +[mark=+, style=solid, color=blue] table[x=meshsize,y=err] {data/2D/5_k0.dat};
\addplot +[mark=triangle, style=solid, color=red] table[x=meshsize,y=err] {data/2D/5_k1.dat};
\addplot +[mark=square, style=solid, color=brown] table[x=meshsize,y=err] {data/2D/5_k2.dat};
\addplot +[mark=pentagon, style=solid, color=darkgray] table[x=meshsize,y=err] {data/2D/5_k3.dat};
\logLogSlopeTriangle{0.90}{0.4}{0.1}{1}{blue};
\logLogSlopeTriangle{0.90}{0.4}{0.1}{2}{red};
\logLogSlopeTriangle{0.90}{0.4}{0.1}{3}{brown};
\logLogSlopeTriangle{0.90}{0.4}{0.1}{4}{darkgray};
\end{loglogaxis}
\end{tikzpicture}
\subcaption{"Hexa anisotropic" mesh.}
\end{minipage}
\hfill
\begin{minipage}[b]{0.1\columnwidth}
\end{minipage}
\end{center}
\caption{Absolute error estimate in discrete norm $\normHNa{\cdot} + \normLs[h]{\cdot}$ vs. mesh size $h$ in $2$ dimensions.}
\label{fig:convrate2D}
\end{figure}

\begin{figure}
\begin{center}
\begin{tikzpicture} 
    \begin{axis}[
    hide axis,
    xmin=0,
    xmax=1,
    ymin=0,
    ymax=1,
    xticklabels={,,},
    yticklabels={,,},
    height=50,
    legend style={legend columns={4}}
    ]
    \addlegendimage{mark=+, style=solid, color=blue}
    \addlegendentry{$k = 0$};
    \addlegendimage{mark=triangle, style=solid, color=red} 
    \addlegendentry{$k = 1$};
    \addlegendimage{mark=square, style=solid, color=brown} 
    \addlegendentry{$k = 2$};
    \addlegendimage{mark=pentagon, style=solid, color=darkgray} 
    \addlegendentry{$k = 3$};
    \end{axis}
\end{tikzpicture}
\\
\medskip
\begin{minipage}[b]{0.45\columnwidth}
\begin{tikzpicture}[scale=0.9]
\begin{loglogaxis}
\addplot +[mark=+, style=solid, color=blue] table[x=meshsize,y=err] {data/3D/1/sol_2_k0.dat};
\addplot +[mark=triangle, style=solid, color=red] table[x=meshsize,y=err] {data/3D/1/sol_2_k1.dat};
\addplot +[mark=square, style=solid, color=brown] table[x=meshsize,y=err] {data/3D/1/sol_2_k2.dat};
\addplot +[mark=pentagon, style=solid, color=darkgray] table[x=meshsize,y=err] {data/3D/1/sol_2_k3.dat};
\logLogSlopeTriangle{0.90}{0.4}{0.1}{1}{blue};
\logLogSlopeTriangle{0.90}{0.4}{0.1}{2}{red};
\logLogSlopeTriangle{0.90}{0.4}{0.1}{3}{brown};
\logLogSlopeTriangle{0.90}{0.4}{0.1}{4}{darkgray};
\end{loglogaxis}
\end{tikzpicture}
\subcaption{"Cube" mesh.}
\end{minipage}
\hfill
\begin{minipage}[b]{0.45\columnwidth}
\begin{tikzpicture}[scale=0.9]
\begin{loglogaxis}
\addplot +[mark=+, style=solid, color=blue] table[x=meshsize,y=err] {data/3D/2/sol_2_k0.dat};
\addplot +[mark=triangle, style=solid, color=red] table[x=meshsize,y=err] {data/3D/2/sol_2_k1.dat};
\addplot +[mark=square, style=solid, color=brown] table[x=meshsize,y=err] {data/3D/2/sol_2_k2.dat};
\logLogSlopeTriangle{0.90}{0.4}{0.1}{1}{blue};
\logLogSlopeTriangle{0.90}{0.4}{0.1}{2}{red};
\logLogSlopeTriangle{0.90}{0.4}{0.1}{3}{brown};
\end{loglogaxis}
\end{tikzpicture}
\subcaption{"Prysmatic" mesh.}
\end{minipage}
\hfill
\begin{minipage}[b]{0.1\columnwidth}
\end{minipage}
\\
\medskip
\begin{minipage}[b]{0.45\columnwidth}
\begin{tikzpicture}[scale=0.9]
\begin{loglogaxis}
\addplot +[mark=+, style=solid, color=blue] table[x=meshsize,y=err] {data/3D/3/sol_2_k0.dat};
\addplot +[mark=triangle, style=solid, color=red] table[x=meshsize,y=err] {data/3D/3/sol_2_k1.dat};
\addplot +[mark=square, style=solid, color=brown] table[x=meshsize,y=err] {data/3D/3/sol_2_k2.dat};
\logLogSlopeTriangle{0.90}{0.4}{0.1}{1}{blue};
\logLogSlopeTriangle{0.90}{0.4}{0.1}{2}{red};
\logLogSlopeTriangle{0.90}{0.4}{0.1}{3}{brown};
\end{loglogaxis}
\end{tikzpicture}
\subcaption{"Hexahedra" mesh.}
\end{minipage}
\hfill
\begin{minipage}[b]{0.45\columnwidth}
\begin{tikzpicture}[scale=0.9]
\begin{loglogaxis}
\addplot +[mark=triangle, style=solid, color=red] table[x=meshsize,y=err] {data/3D/4/sol_2_k1.dat};
\addplot +[mark=square, style=solid, color=brown] table[x=meshsize,y=err] {data/3D/4/sol_2_k2.dat};
\logLogSlopeTriangle{0.90}{0.4}{0.1}{1}{blue};
\logLogSlopeTriangle{0.90}{0.4}{0.1}{2}{red};
\logLogSlopeTriangle{0.90}{0.4}{0.1}{3}{brown};
\end{loglogaxis}
\end{tikzpicture}
\subcaption{"Tetgen cube" mesh.}
\end{minipage}
\hfill
\begin{minipage}[b]{0.1\columnwidth}
\end{minipage}
\\
\medskip
\begin{minipage}[b]{0.45\columnwidth}
\begin{tikzpicture}[scale=0.9]
\begin{loglogaxis}
\addplot +[mark=+, style=solid, color=blue] table[x=meshsize,y=err] {data/3D/5/sol_2_k0.dat};
\addplot +[mark=triangle, style=solid, color=red] table[x=meshsize,y=err] {data/3D/5/sol_2_k1.dat};
\addplot +[mark=square, style=solid, color=brown] table[x=meshsize,y=err] {data/3D/5/sol_2_k2.dat};
\logLogSlopeTriangle{0.90}{0.4}{0.1}{1}{blue};
\logLogSlopeTriangle{0.90}{0.4}{0.1}{2}{red};
\logLogSlopeTriangle{0.90}{0.4}{0.1}{3}{brown};
\end{loglogaxis}
\end{tikzpicture}
\subcaption{"Voro" mesh.}
\end{minipage}
\hfill
\begin{minipage}[b]{0.45\columnwidth}
\begin{tikzpicture}[scale=0.9]
\begin{loglogaxis}
\addplot +[mark=+, style=solid, color=blue] table[x=meshsize,y=err] {data/3D/6/sol_2_k0.dat};
\addplot +[mark=triangle, style=solid, color=red] table[x=meshsize,y=err] {data/3D/6/sol_2_k1.dat};
\addplot +[mark=square, style=solid, color=brown] table[x=meshsize,y=err] {data/3D/6/sol_2_k2.dat};
\logLogSlopeTriangle{0.90}{0.4}{0.1}{1}{blue};
\logLogSlopeTriangle{0.90}{0.4}{0.1}{2}{red};
\logLogSlopeTriangle{0.90}{0.4}{0.1}{3}{brown};
\end{loglogaxis}
\end{tikzpicture}
\subcaption{"Voro tets" mesh.}
\end{minipage}
\hfill
\begin{minipage}[b]{0.1\columnwidth}
\end{minipage}
\end{center}
\caption{Absolute error estimate in discrete norm $\normHNa{\cdot} + \normLs[h]{\cdot}$ vs. mesh size $h$ in $3$ dimensions.}
\label{fig:convrate3D}
\end{figure}

\appendix
\section{Results on polynomial spaces.} \label{Resultsonpolynomialspaces}
We begin by showing a few results to complete the introduction of spaces \eqref{eq:defRbc} and \eqref{eq:defRb}.
Let $T \in \Th$ be any cell.
We assume without loss of generality that $\bvec{x}_T = 0$, where $\bvec{x}_T \in T$ is the point given by \eqref{eq:defkoszul}.
We identify the coordinate with three variables $x$, $y$ and $z$, and we define $\Poly{k}[X,Y](T) = \lbrace P \in \Poly{k}(T),\; \text{deg}_Z(P) = 0 \rbrace$.
By \eqref{eq:defkoszul} we can write:
\begin{equation} \label{eq:rc3exp}
(\Roly{c,k}(T)^\intercal)^3 = \begin{pmatrix} 
x P_1 & y P_1 & z P_1\\
x P_2 & y P_2 & z P_2\\
x P_3 & y P_3 & z P_3
\end{pmatrix},\;
P_{i \in \lbrace 1,2,3 \rbrace} \in \Poly{k-1}(T).
\end{equation}
\begin{lemma} \label{lemma:explicitRbCompl}
Keeping the notations of \eqref{eq:rc3exp}, 
the subset $\Rolyb{c,k}(T) = \lbrace W \in (\Roly{c,k}(T)^\intercal)^3 \st \Tr{W} = 0 \rbrace$ 
is characterized by:
\begin{equation*}
\begin{matrix}
P_1 = yz  C^1 + y \gamma - z \beta \\
P_2 = xz  C^2 - x \gamma + z \lambda \\
P_3 = xy  C^3 + x \beta - y \lambda
\end{matrix},\quad
\begin{matrix}
\lambda \in \Poly{k-2}[Y,Z], \beta \in \Poly{k-2}[X,Z], \gamma \in \Poly{k-2}[X,Y], \\
C^i \in \Poly{k-3}[X,Y,Z], C^1 + C^2 + C^3 = 0.
\end{matrix}
\end{equation*}
\end{lemma}
\begin{proof}
The proof relies on repeated use of Euclidean division. 
When we append $x$ (resp. $y$, $z$) as subscript we mean that the polynomial does not depend on $x$ (resp. $y$, $z$).
We write:
\begin{equation*}
\begin{aligned}
P_2 =&\; x Q_2 + R^2_x,\\
P_3 =&\; x Q_3 + R^3_x.
\end{aligned}
\end{equation*}
Then the condition $\Tr{W} = 0$ becomes $x P_1 + y P_2 + z P_3 = x (P1 + y Q_2 + z Q_3) + y R^2_x + z R^3_x = 0$.
Looking at the degree in $X$, we must have 
$P_1 + y Q_2 + z Q_3 = 0$ and $y R^2_x + z R^3_x = 0$.
Hence there exists $\lambda \in \Poly{k-2}[Y,Z]$ such that 
\begin{equation*}
R^2_x = z \lambda,\; R^3_x = - y \lambda .
\end{equation*}
On the other hand, writing $P_1 = y A^1 + z A^1_y + A^1_{y,z}$, we have 
$y A^1 + z A^1_y + A^1_{y,z} + y Q_2 + z Q_3  = 0$.
Looking at the degree in $Y$ and $Z$ we must have $A^1_{y,z} = 0$.
Let us write $Q_3 = y B^3 + B^3_y$, we have 
$y( A^1 + Q_2 + z B^3) + z (A^1_y + B^3_y) = 0$.
The degree in $Y$ shows that $A^1_y + B^3_y = 0$.
Thus there exists $\beta \in \Poly{k-2}[X,Z]$ such that 
\begin{equation*}
A^1_y = -\beta,\; B^3_y = \beta .
\end{equation*}
Let $A^1 = z C^1 + C^1_z$ and $Q_2 = z C^2 + C^2_z$. 
We are left with $A^1 + Q_2 + z B^3 = z( C^1 + C^2 + B^3) + C^1_z + C^2_z = 0$.
Once again the degree in $Z$ shows that $C^1_z + C^2_z = 0$ and $C^1 + C^2 + B^3 = 0$, 
so there exists $\gamma \in \Poly{k-2}[X,Y]$ such that 
\begin{equation*}
C^1_z = \gamma,\; C^2_z = - \gamma .
\end{equation*}
We conclude by writing $B^3 := C^3$ and substituting back each term in the development of $P_i$.
\end{proof}
\begin{remark}
In $2$ dimensions the expression becomes (see Appendix \ref{2DComplex}):
\begin{equation*}
\Rolyb{c,k+1}(F) = \begin{pmatrix} 
- (x - x_F)(y - y_F) Q & - (y - y_F)^2 Q \\
(x - x_F)^2 Q & (x - x_F)(y - y_F) Q
\end{pmatrix},\, Q \in \Poly{k-1}(F).
\end{equation*}
\end{remark}

\begin{lemma} \label{lemma:isoRbG}
For any $T \in \Th$, $\TDIV$ is an isomorphism from $\Rolyb{c,k+1}(T)$ to $\Goly{c,k}(T)$.
\end{lemma}
\begin{proof}
By the definition \eqref{eq:defkoszul} the space $\Goly{c,k}$ is characterized by:
\begin{equation} \label{eq:proofRbG1}
\Goly{c,k}(T) = \begin{pmatrix}
y Q_1 - z Q_2 \\
z Q_3 - x Q_1 \\
x Q_2 - y Q_3 
\end{pmatrix},\quad
Q_1,Q_2,Q_3 \in \Poly{k-1}(T).
\end{equation}
Let $W = \begin{pmatrix} 
x P_1 & y P_1 & z P_1\\
x P_2 & y P_2 & z P_2\\
x P_3 & y P_3 & z P_3
\end{pmatrix}
\in \Rolyb{c,k}(T)$, with $P_i$ given by Lemma \ref{lemma:explicitRbCompl}.
A direct computation gives 
\begin{equation} \label{eq:proofRbG2}
\TDIV W = 
\begin{pmatrix}
yz  {C^1}' + y \gamma' - z \beta' \\
xz  {C^2}' - x \gamma' + z \lambda' \\
xy  {C^3}' + x \beta' - y \lambda'
\end{pmatrix},
\end{equation}
with ${C^i}' = (5 + x\partial_x + y\partial_y + z\partial_z) C^i$, 
$\gamma' = (4 + x\partial_x + y\partial_y) \gamma$,
$\beta' = (4 + x\partial_x + z\partial_z) \beta$,
$\lambda' = (4 + y\partial_y + z\partial_z) \lambda$.
Each of these transformations is a linear automorphism, 
hence we can drop the $'$ in \eqref{eq:proofRbG2} without loss of generality.
We must show that there is a one to one correspondence between the $Q_i$ of \eqref{eq:proofRbG1}
(more precisely their sum, since different $Q_i$ give the same expression)
and the $C^i$, $\gamma$, $\beta$, $\lambda$ of \eqref{eq:proofRbG2}.

Let us write $Q_1 = \gamma + R^1$, $Q_2 = \beta + R^2$ and $Q_3 = \lambda + R^3$ for some $R^1$, $R^2$, $R^3$.
The system becomes:
\begin{equation*}
\begin{matrix}
yz  C^1 + y \gamma - z \beta = y Q_1 - z Q_2\\
xz  C^2 - x \gamma + z \lambda = z Q_3 - x Q_1\\
xy  C^3 + x \beta - y \lambda = x Q_2 - y Q_3
\end{matrix} \iff
\begin{matrix}
yz  C^1 = y R^1 - z R^2\\
xz  C^2 = z R^3 - x R^1\\
xy  C^3 = x R^2 - y R^3
\end{matrix}.
\end{equation*}
When can check that $xyz( C^1 + C^2 + C^3) = xyR^1 - xzR^2 + yzR^3 - yxR^1 + zxR^2 - zyR^3 = 0$.
Hence given $Q_1$, $Q_2$, $Q_3$, 
the Euclidean divisions
$Q_1 := \gamma + z S^1$, 
$Q_2 := \beta + y S^2$ and
$Q_3 := \lambda + x S^3$ 
give suitable $C^1$, $C^2$, $C^3$, $\gamma$, $\beta$, $\lambda$ for 
$C^1 = S^1 - S^2$, $C^2 = S^3 - S^1$ and $C^3 = S^2 - S^3$.
Likewise, given 
$C^1$, $C^2$, $C^3$, $\gamma$, $\beta$, $\lambda$, the system
\begin{equation*}
\begin{aligned}
C^1 =&\; S^1 - S^2\\
C^2 =&\; S^3 - S^1\\
C^3 =&\; S^2 - S^3
\end{aligned}
\end{equation*}
is solvable since $C^1 + C^2 + C^3 = 0$ and
$Q_1 := \gamma + z S^1$, 
$Q_2 := \beta + y S^2$,
$Q_3 := \lambda + x S^3$ give suitable $Q_i$.
\end{proof}

\begin{lemma} \label{lemma:intRbcRb}
For any $T \in \Th$, it holds $\Rolyb{c,k}(T) \cap \Rolyb{k}(T) = \lbrace 0 \rbrace$.
\end{lemma}
\begin{proof}
This is a direct consequence of Lemma \ref{lemma:isoRbG} and Remark \ref{rem:divInvDivGrad}.
\end{proof}

\begin{lemma} \label{lemma:Rcdec}
For any $T \in \Th$, it holds
$(\Roly{c,k}(T)^\intercal)^3 = \Rolyb{c,k}(T) \oplus \Rolyb{k}(T)$.
\end{lemma}
\begin{proof}
Lemma \ref{lemma:intRbcRb} already shows that $\Rolyb{c,k}(T) \cap \Rolyb{k}(T) = \lbrace 0 \rbrace$.
It is enough to compare the dimension of these spaces:
\begin{equation*}
\begin{aligned}
\dim{\Rolyb{c,k}(T)} =&\; 3 \dim{\Poly{k-2}[X,Y]} + 2 \dim{\Poly{k-3}[X,Y,Z]} \\
=&\ 3 \frac{k!}{2! (k- 2)!} + 2 \frac{k!}{3! (k-3)!} = \frac{2 k^3 + 3 k^2 - 5k}{6},
\end{aligned}
\end{equation*}
\begin{equation*}
\dim{\Rolyb{k}(T)} = \dim{\Poly{0,k}(T)} = \frac{(k+3)!}{3! k!} - 1 =  \frac{k^3 + 6 k^2 + 11 k}{6} .
\end{equation*}
The sum of both is 
\begin{equation*}
3 \frac{(k+2)!}{3! (k-1)!} = 3 \dim{\Poly{k-1}(T)} = \dim{(\Roly{c,k}(T)^\intercal)^3}.
\end{equation*}
\end{proof}

Next we show some lemmas on convex polytopes. 
\begin{lemma} \label{lemma:normpcz1}
Let $T \in \Th$, $\bvec{x}_T$ defined as in \eqref{eq:defkoszul}, if $B = \Ball{\bvec{x}_T}{h_B} \subset T$ with $h_B \gtrsim h_T$ and $Q \in \Poly{k}(T)$ then $\norm[L^\infty(B)]{Q} \approx \norm[\infty]{Q}$.
\end{lemma}
\begin{proof}
Let $h_o \in \Real^*_+$ such that $T \subset \Ball{\bvec{x}_T}{h_o}$ and $h_o \lesssim h_T$ ($h_o$ exists by the regularity assumption on the mesh sequence $\Mh$).
Let $\bvec{v}$ by any vector such that $\norm{\bvec{v}} = 1$, then
$\forall \alpha > 0 $ such that $\bvec{x}_T + \alpha \bvec{v} \in T$, 
\begin{equation*}
Q(\bvec{x}_T + \alpha \bvec{v}) = \sum_{i=0}^{k} \frac{\partial_{\bvec{v}}^i Q(\bvec{x}_T)}{i! } \alpha^i ,
\end{equation*}
so
\begin{equation*}
\seminorm{Q(\bvec{x}_T + \alpha \bvec{v})} \leq \sum_{i=0}^{k} \frac{\norm[L^\infty(B)]{\partial_{\bvec{v}}^i Q}}{i! } \alpha^i .
\end{equation*}
By the discrete Poincare inequality Lemma \ref{lemma:discretepoincare} we have $\forall i$, $\norm[L^\infty(B)]{\partial_{\bvec{v}}^i Q} \lesssim h_B^{-i} \norm[L^\infty(B)]{Q}$.
Moreover $\bvec{x}_T + \alpha \bvec{v} \in T$ so $\alpha < h_o$ and
$\seminorm{Q(\bvec{x}_T + \alpha \bvec{v})} \lesssim \norm[L^\infty(B)]{Q} h_o^i h_B^{-i}$. 
Since $h_o \lesssim h_B$ we can conclude that
$\norm[\infty]{Q} \lesssim \norm[L^\infty(B)]{Q}$.
\end{proof}

Any $T \in \Th$ is a convex, open polyhedron. Let $(F_i)_{i \leq \seminorm{\FT}}$ be the set of its faces, each of normal vector $\nF$. 
For all $F_i$ there exists $P_i \in \Poly{1}(\Real^3)$ such that $F_i \subset \Ker(P_i)$, moreover we can normalize $P_i$ such that 
$\bvec{x} \in T \implies P_i(\bvec{x}) > 0$ (since $T$ is convex) and $\seminorm{\partial_{\nF} P_i} = 1$.
\begin{lemma} \label{lemma:normpcz2}
Let $P = \prod_{i \leq \seminorm{\FT}} P_i$, $\bvec{x}_0 \in T$ such that $B_0 = \Ball{\bvec{x}_0}{h_T/2} \subset T$ and
$B = \Ball{\bvec{x}_0}{h_T/4}$
then $\inf_{\bvec{x} \in B} P(\bvec{x}) > \left ( \frac{h_T}{4}\right )^{\seminorm{\FT}}$ and $\norm[L^\infty(T)]{P} \lesssim h_T^{\seminorm{\FT}}$.
\end{lemma}
\begin{proof}
For any $i \leq \seminorm{\FT}$, the value $P_i(\bvec{x})$ at any point $\bvec{x}$ is the distance between $\bvec{x}$ and the plane tangent to $F_i$
and is positive on $T$.
For any $\bvec{x} \in B$, $\bvec{x}$ is a least at a distance $h_T/4$ of any face since $B_0 \subset T$. 
We obtain the lower bound taking the product over all faces. 
Conversely, using the mesh regularity we can find $h_o > 0$, $h_o \lesssim h_T$ such that $T$ is inscribed in a ball of diameter $h_o$.
Then $\forall i \leq \seminorm{\FT}$, $\forall \bvec{x} \in T$, it holds $0 < P_i(\bvec{x}) < h_o \lesssim h_T$.
Again we conclude by taking the product over all faces.
\end{proof}

\begin{remark} \label{remark:normpcz2sq}
Lemma \ref{lemma:normpcz2} also holds for $P = \prod_{i \leq \seminorm{\FT}} P_i^2$, substituting $\seminorm{\FT}$ by $2 \seminorm{\FT}$.
\end{remark}

\begin{lemma} \label{lemma:normpcz}
For any $T \in \Th$ and $q \in \Poly{k}(T)$ there is $P \in \Poly{k + \seminorm{\FT}}(\overline{T})$ such that  
$P_{\vert \partial T} = 0$, $\lproj{k}{T} P = q$ and $\norm{P} \approx \norm{q}$.
\end{lemma}
\begin{proof}
Let $P = \Pi Q$ with $\Pi$ given by Lemma \ref{lemma:normpcz2} and $Q \in \Poly{k}(T)$.
The application 
\begin{equation} \label{eq:normpcz1}
\Poly{k}(T) \ni Q \rightarrow \left ( \lambda \rightarrow \int_T P \lambda \right ) \in \Poly{k}(T)'
\end{equation}
is linear and between two spaces of same dimension, thus it is enough to check that it is injective.
Let $Q \in \Poly{k}(T)$ such that $\forall \lambda \in \Poly{k}(T)$, $\int_T \Pi Q \lambda = 0$;
in particular $\int_T \Pi Q^2 = 0$.
However since $\Pi > 0$ on $T$, we can define the function $\sqrt \Pi \in L^2(T)$ and have $\int_T \left ( \sqrt \Pi Q \right )^2 = \norm{\sqrt \Pi Q}^2 = 0$.
This implies $\sqrt \Pi Q \equiv 0$, so $Q \equiv 0$
and \eqref{eq:normpcz1} is injective.
Thus there is a polynomial $P\in \Poly{k + \seminorm{\FT}}(\overline{T})$ 
such that $P_{\vert \partial T} = 0$ and $\lproj{k}{T} P = q$.
Let us show that $P$ also satisfies the norm equivalence:
In particular we must have 
\begin{equation*}
\begin{gathered}
\int_T (P - q) Q = 0 ,\\
\int_T \Pi Q^2 = \int_T q Q ,\\
\norm{\sqrt \Pi Q}^2 \leq \norm{q}\norm{Q} \lesssim \norm{q} \seminorm{T}^\frac{1}{2} \norm[\infty]{Q} .
\end{gathered}
\end{equation*}
Therefore the discrete Sobolev inequality \cite[Lemma~1.25]{hho} gives
\begin{equation*}
\norm[\infty]{\sqrt \Pi Q}^2 \approx \seminorm{T}^{-1} \norm{\sqrt \Pi Q}^2 \lesssim \seminorm{T}^{-\frac{1}{2}} \norm[\infty]{Q} \norm{q} .
\end{equation*}
On the other consider $B = \Ball{\bvec{x}_0}{h_T/4}$ given in Lemma \ref{lemma:normpcz2}, 
it holds $\inf_{\bvec{x} \in B} \sqrt \Pi (\bvec{x}) \gtrsim h_T^{\seminorm{\FT}/2}$. 
Thus by Lemma \ref{lemma:normpcz1}:
\begin{equation*}
\norm[\infty]{\sqrt \Pi Q} \geq \norm[L^{\infty}(B)]{\sqrt \Pi Q} \gtrsim h_T^{\seminorm{\FT}/2} \norm[L^\infty(B)]{Q} 
\approx h_T^{\seminorm{\FT}/2} \norm[\infty]{Q} .
\end{equation*}
Hence
$h_T^{\seminorm{\FT}} \norm[\infty]{Q}^2 \lesssim \seminorm{T}^{-\frac{1}{2}} \norm[\infty]{Q}\norm{q}$,
$\norm[\infty]{Q} \lesssim \seminorm{T}^{-\frac{1}{2}} h_T^{- \seminorm{\FT}} \norm{q}$ and
\begin{equation*}
\begin{aligned}
\norm{\Pi Q} \approx \seminorm{T}^{\frac{1}{2}} \norm[\infty]{\Pi Q} \leq \seminorm{T}^{\frac{1}{2}} \norm[\infty]{\Pi} \norm[\infty]{Q} 
\lesssim \seminorm{T}^{\frac{1}{2}} h_T^{\seminorm{\FT}} \norm[\infty]{Q} \lesssim \norm{q} .
\end{aligned}
\end{equation*}
\end{proof}

We can use the same proof with Remark \ref{remark:normpcz2sq} to enforce the continuity of derivatives.
\begin{lemma} \label{lemma:normpczH2}
For any $T \in \Th$ and $q \in \Poly{k}(T)$ there is $P \in \Poly{k + 2 \seminorm{\FT}}(\overline{T})$ such that  
$P_{\vert \partial T} = 0$, $(\GRAD P)_{\vert \partial T} = \bvec{0}$, $\lproj{k}{T} P = q$ and $\norm{P} \approx \norm{q}$.
\end{lemma}

\section{Trace lifting.} \label{Tracelifting}
In order to prove consistency results we often need functions of Sobolev spaces with suitable properties. 
We construct them in this section.
\begin{theorem} \label{th:tracelift}
For all $\uvec{v}_T \in \uHv$ there is $\LiftNaF{\uvec{v}_T} \in \Hv{T}$ such that
\begin{equation}
\begin{aligned} \label{eq:traceliftbound}
\LiftNaF{\uvec{v}_T} =&\; \trnac \uvec{v}_F \text{ on } \partial T,\\
\norm[T]{\LiftNaF{\uvec{v}_T}} + \norm[T]{\TGRAD \LiftNaF{\uvec{v}_T}} \lesssim&\; \opnNa{\uvec{v}_T} + \opnLt{\uNaT \uvec{v}_T}.
\end{aligned}
\end{equation}
\end{theorem}
This lift is built upon \cite[Theorem~18.40]{firstcourseSobolev}:
Let $\Omega \subset \Real^N$,
$N \geq 2$ be an open set whose boundary $\partial \Omega$ is uniformly Lipschitz continuous of parameters $\epsilon$, $L$ and $M$ (see \cite[Definition~13.11]{firstcourseSobolev}).
Then for all $g \in \Besov{\partial \Omega}$, there is $c \in \Real$ depending only on $N$ and a function $u \in H^1(\Omega)$ such that $\Tr(u) = g$,
\begin{equation} \label{eq:btr1}
\norm[\Ls{\Omega}]{u} \leq M^{1/2} \epsilon^{1/2} \norm[\Ls{ \partial \Omega}]{g}
\end{equation}
and
\begin{equation} \label{eq:btr2}
\norm[\Lt{\Omega}]{\GRAD u} \leq c M (1 + L)^{3 + N/2} \epsilon^{-1/2} \norm[\Ls{\partial \Omega}]{g}
+ c M (1 + L)^{2 + (N + 1)/2} \semiBtr{g} .
\end{equation}
With the Besov seminorm defined by (see \cite[Definition~18.36]{firstcourseSobolev}):
\begin{equation}
\semiBtr{g} := \left ( \int_{\partial \Omega} \int_{\partial \Omega \cap \ball{x}{\epsilon}} 
\frac{\seminorm{g(x) - g(y)}^2}{\norm{x - y}^{N}}dy dx \right )^{1/2}.
\end{equation}
\begin{proof}[Proof of theorem \ref{th:tracelift}]
We apply the above-mentioned theorem \cite[Theorem~18.40]{firstcourseSobolev} to $\Omega = T$ and $g$ a component of $\trnac \uvec{v}_{F}$.
Here $N = 3$ and 
the mesh regularity \cite[Definition~1.9]{hho} allows us to take an open cover of $\partial T$ 
making it uniformly Lipschitz continuous in the sense of
\cite[Definition~13.11]{firstcourseSobolev} with 
$L = 1$, $M \approx 1$ and $\epsilon \approx h_T$.

Let $\LiftNaF{\uvec{v}_T}$ be such that $\Tr(\LiftNaF{\uvec{v}_T}) = \trnac \uvec{v}_{F}$ and that $\LiftNaF{\uvec{v}_T}$ satisfies \eqref{eq:btr1} and \eqref{eq:btr2}.
Let $g$ be a component of $\trnac \uvec{v}_{F}$ and $u$ be given by \eqref{eq:btr1} and \eqref{eq:btr2}.
Without loss of generality we assume that $\int_{\partial T} g = 0$:
Else we take instead $\overline{g} = \int_{\partial T} g$ and $u = \overline{g}$ so that
$\GRAD u = 0$ and $\norm[\Ls{T}]{u} \approx h_T^\frac{3}{2} \seminorm{\overline{g}}$, 
$\norm[\Ls{\partial T}]{\overline{g}} \approx h_T \seminorm{\overline{g}}$
and $u$, $\overline{g}$ satisfy \eqref{eq:traceliftbound}.
Hence we reduce to the case $\int_{\partial T} g' = 0$ for $g' = g - \overline{g}$.

Equation \eqref{eq:btr1} with Lemma \ref{lemma:boundtrna} gives $\norm[T]{\LiftNaF{\uvec{v}_T}} \lesssim \opnNa{\uvec{v}_T}$
since $\epsilon \approx h_T$.
Let $A_F := \frac{1}{\seminorm{F}} \int_F g_{\vert F}$ and $A_T := \sum_{F \in \FT} \seminorm{F} A_F$.
By assumption $A_T = 0$, and for all $F \in \FT$, 
\begin{equation*}
\begin{aligned}
\sum_{F \in \FT} \norm[F]{g_{\vert F}} 
\leq&\; \norm[F]{g_{\vert F} - A_F} + \norm[F]{A_F - A_T}\\
\lesssim&\; h_F \norm[F]{\TGRAD \trnac \uvec{v}_F} + h_F^\frac{1}{2} \opnNa{\uvec{v}_T}.
\end{aligned}
\end{equation*}
We used the Poincaré-Wirtinger inequality on each face on the first term in the right-hand side and the same proof as Lemma \ref{lemma:normgrad} 
(see \cite[Equation~5.12]{ddr}).
Hence, Lemma \ref{lemma:normgrad} gives $\norm[\Ls{\partial T}]{g} \lesssim h_T^\frac{1}{2} \opnNa{\uvec{v}_T}$,
and 
\begin{equation*}
\begin{aligned}
\norm[\Ls{\Omega}]{\GRAD u} \lesssim&\; h_T^{-\frac{1}{2}} \norm[\Ls{\partial T}]{g} + \semiBtr{g} \\
\lesssim&\; \opnNa{\uvec{v}_T}.
\end{aligned}
\end{equation*}
We concluded with the estimate on the Besov seminorm Lemma \ref{lemma:besovestimate}.
\end{proof}

\begin{lemma} \label{lemma:besovestimate}
Keeping the notations of the proof of Theorem \ref{th:tracelift}, it holds 
\begin{equation}
\semiBtr{g} \lesssim \opnNa{\uvec{v}_T} .
\end{equation}
\end{lemma}
\begin{proof}
We know that $g \in C^1(\FT)$, 
hence by the mean value theorem,
$\forall y \in \ball{x}{\epsilon}$, $\exists c \in ]0,1[$ such that 
\begin{equation*}
g(\bvec{y}) = g(\bvec{x}) + \GRAD{g}((1 - c)\bvec{x} + c \bvec{y}) \cdot (\bvec{y} - \bvec{x}),
\end{equation*}
thus
\begin{equation*}
\frac{\seminorm{g(\bvec{x}) - g(\bvec{y})}^2}{\norm{\bvec{x} - \bvec{y}}^3} 
= \frac{\seminorm{\GRAD{g}((1 - c)\bvec{x} + c \bvec{y})}^2\norm{\bvec{x} - \bvec{y}}^2}{\norm{\bvec{x} - \bvec{y}}^3} 
\lesssim \norm[L^\infty(F)]{\GRAD{g}}^2 \seminorm{(\bvec{x} - \bvec{y})}^{-1}.
\end{equation*}
Switching to polar coordinates gives
\begin{equation*}
\begin{aligned}
\int_{\ball{x}{\epsilon}} \frac{\seminorm{g(x) - g(y)}^2}{\norm{x - y}^3} 
\lesssim\; \norm[L^\infty(F)]{\GRAD{g}}^2 \int_0^\epsilon \frac{1}{r} r \lesssim\; \norm[L^\infty(F)]{\GRAD{g}}^2 \epsilon.
\end{aligned}
\end{equation*}
Lemma \ref{lemma:discretepoincare} and a Poincaré-Wirtinger inequality show that
\begin{equation*}
\begin{aligned}
\norm[F]{\TGRAD \trnac \uvec{v}_F} 
\lesssim&\; h_F^{-1} \norm[F]{\trnac \uvec{v}_F - \frac{1}{\seminorm{F}} \int_F \trnac \uvec{v}_F} \\
\approx&\; h_F^{-1} \norm[F]{\trna \uvec{v}_F - \frac{1}{\seminorm{F}} \int_F \trna \uvec{v}_F} \\
\lesssim&\; \norm[F]{\TGRAD \trna \uvec{v}_F}.
\end{aligned}
\end{equation*}
Hence the discrete Lebesgue embedding \cite[Lemma~1.25]{hho} and Lemma \ref{lemma:normgrad} give
\begin{equation*}
\norm[L^\infty(F)]{\GRAD{g}}^2 \approx h_F^{-2} \norm[L^2(F)]{\GRAD{g}}^2 
\lesssim h_F^{-3} h_F \norm[F]{\TGRAD \trna \uvec{v}_F}^2
\lesssim h_T^{-3} \opnNa{\uvec{v}_T}^2.
\end{equation*}
Inferring $\epsilon \approx h_T$ and $\seminorm{\partial T} \approx h_T^2$ we find
\begin{equation*}
\begin{aligned}
\semiBtr{g} 
\lesssim&\; \left ( \int_{\partial T} \epsilon \norm[L^\infty(F)]{\GRAD g}^2 \right )^\frac{1}{2} \\
\lesssim&\; \left ( h_T h_T^{-3} \opnNa{\uvec{v}_T}^2 \int_{\partial T} 1 \right )^\frac{1}{2} \\
\lesssim&\; \opnNa{\uvec{v}_T}.
\end{aligned}
\end{equation*}
\end{proof}

\begin{theorem} \label{th:divlift}
If $p \in H^2_0(\Omega)$ then there is $\bvec{u} \in \bvec{H}^3(\Omega)$ such that $\DIV \bvec{u} = p$, 
$\norm[\bvec{H}^{3}]{\bvec{u}} \lesssim \norm[{H}^2]{p}$,
$\norm[\bvec{H}^{2}]{\bvec{u}} \lesssim \norm[{H}^1]{p}$ and $\norm[\bvec{H}^{1}]{\bvec{u}} \lesssim \norm[L^2]{p}$.
\end{theorem}
\begin{proof}
Let $B$ be a smooth bounded extension (at least $C^{3,1}$) of $\Omega$.
For all function $g \in H^{-1}(B)$, following \cite[Theorem~III.4.1]{MathTools}, there is a unique solution $f \in H^1_0(B)$ 
to the equation $\Delta f = g \text{ in } B$. 
Moreover this solution satisfies $\norm[H^1]{f} \lesssim \norm[H^{-1}]{g}$
and \cite[Theorem~III.4.2]{MathTools} shows that if $B$ is $C^{k+1,1}$, $k \ge 0$ and $g \in H^k(B)$ then 
$\norm[H^{k+2}]{f} \lesssim \norm[H^k]{g}$.
Since $p \in H^2_0(\Omega)$ we can extend $p$ by zero and define $\tilde{p} \in H^2_0(B)$ with $\norm[H^k(B)]{\tilde{p}} = \norm[H^k(\Omega)]{p}$ (\cite[Theorem~3.33]{McLean2000}).
Hence, since $\tilde{p} \in H^2(B)$,
if we take $f \in H^1_0(B)$ such that $\Delta f = \DIV \GRAD f = \tilde{p}$ we get $f \in H^4(B)$ with
$\norm[H^2(B)]{f} \lesssim \norm[L^2]{p}$,
$\norm[H^3(B)]{f} \lesssim \norm[H^1]{p}$,
and $\norm[H^4(B)]{f} \lesssim \norm[H^2]{p}$.
Let $\bvec{u} = \GRAD f_{\vert \Omega}$ then we have $\DIV \bvec{u} = p \text{ in } \Omega$ and the expected bounds.
\end{proof}

We can adapt the theorem to cover other boundary conditions.
However, we still need a smoother domain, and if we want to enforce a condition on the boundary of $\Omega$ 
we can no longer consider a larger domain.
Instead, we take a smaller domain, and we will have to do some work to recover the correct values on $\Omega$.
\begin{theorem} \label{th:divlift0}
If $B \subset \Omega$ is $C^{2,1}$, $p \in H^2(B)$ such that $\int_\Omega p = 0$ then there is $\bvec{u} \in \bvec{H}^3(B) \cap \bvec{H}^1_0(B)$ such that
$\DIV \bvec{u} = p$, $\norm[\bvec{H}^{3}]{\bvec{u}} \lesssim \norm[{H}^2]{p}$, $\norm[\bvec{H}^{2}]{\bvec{u}} \lesssim \norm[{H}^1]{p}$
and $\norm[\bvec{H}^{1}]{\bvec{u}} \lesssim \norm[L^2]{p}$.
\end{theorem}
\begin{proof}
Since $\int_B p = 0$ we can use \cite[Theorem~IV.5.2 and Theorem~IV.5.8]{MathTools} 
to get $\bvec{u} \in \bvec{H}^{k+1}(B)$ such that $\DIV \bvec{u} = p$ in $B$, $\bvec{u} = \bvec{0}$ on $\partial B$  
and $\norm[\bvec{H}^{k+1}]{\bvec{u}} \lesssim \norm[H^{k}]{p}$,
$k \in \lbrace 0, 1, 2 \rbrace$.
\end{proof}

\begin{remark}[Adaptation of Lemma \ref{lemma:rightinvdiv}] \label{rem:adaptinvdiv} 
The continuous (for the Sobolev norm) extension of a function in $H^1_0$ is not necessarily by zero.
However, we only need to take $L^2$-orthogonal projections,
hence we can afford to use a smaller domain $B$.
Let $B \subset \Omega$ be a $C^{2,1}$ domain 
containing all interior elements and most of the elements next to the boundary.
Formally we want $B$ such that $\forall X \in \Mh$, 
$\overline{X} \cap \partial \Omega = \emptyset \implies X \subset B$
and 
$\overline{X} \cap \partial \Omega \subsetneq \overline{X} \implies \seminorm{X \cap B} \geq \seminorm{X}/2$.

To adapt Lemma \ref{lemma:rightinvdiv} we take another mesh $\Mh'$ 
with the same interior elements as $\Mh$ 
and with elements crossing $\partial B$ collinear to those of $\Mh$ 
but cut before. 
This way, the domain on which $\Mh'$ is defined lies inside $B$.
Figure \ref{fig:subdomain} illustrates the construction.

Given $\ul{p}_h \in \uLshstar$ we construct 
$\ul{p}_h'$, a piecewise polynomial on each cell of $\Mh'$ 
with the same degree and moments as $\ul{p}_h$ has on the corresponding cell of $\Mh$.
We can now proceed as done in the proof of Lemma \ref{lemma:rightinvdiv}, 
with $\ul{p}_h'$ (extended by zero on $B$) instead of $\ul{p}_h$ 
and with a slightly different interpolator.
The modified interpolator is the same on interior elements,
and on elements $X \in \Mh$ crossing $\partial B$ of corresponding element $X' \in \Mh$ 
we substitute the $L^2$-orthogonal projector by $\pi$ such that 
\begin{equation*}
\int_X (\pi_P f) \, g := \int_{X'} f\, g ,\; \forall f \in L^2(X'), \forall g \in P.
\end{equation*}
One can check that the proof still works since $\bvec{u} \in \bvec{H}^1_0(B)$.
\end{remark}

\begin{figure}
\centering
\includegraphics[width=0.5\textwidth]{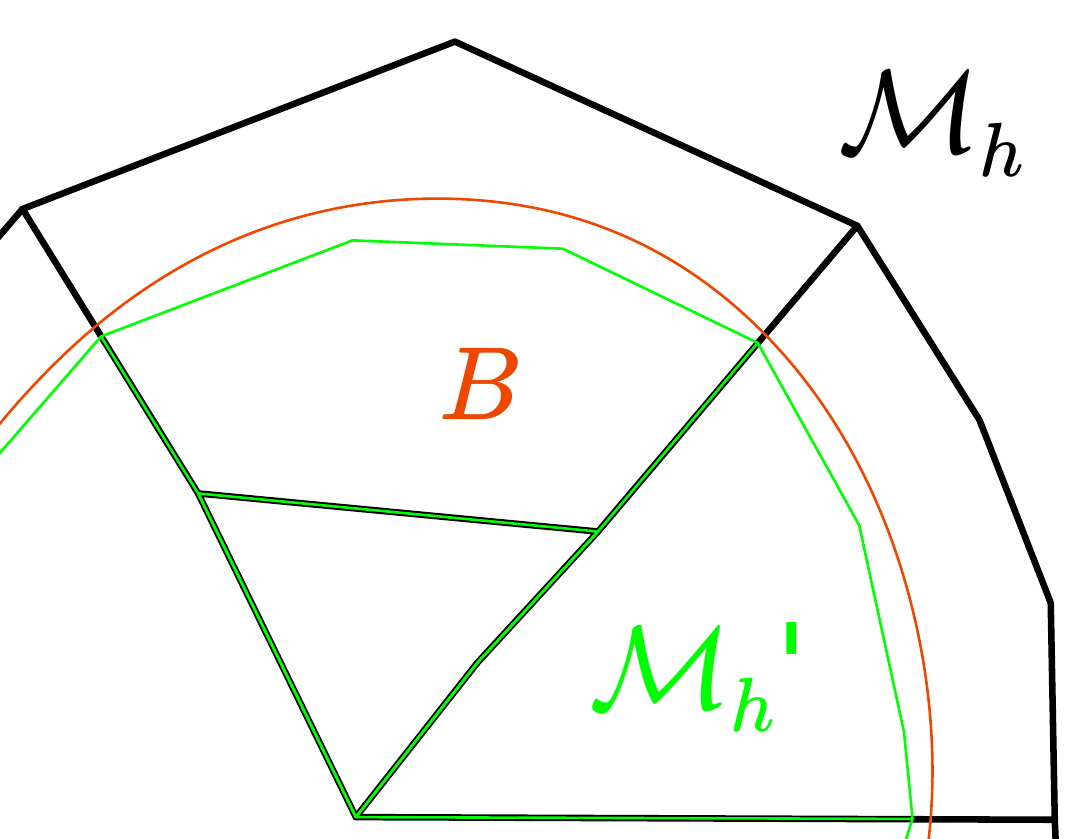}
\caption{Construction of the smooth subdomain.}
\label{fig:subdomain}
\end{figure}

\section{$2$-Dimensional complex.} \label{2DComplex}

The Stokes complex also exists in $2$ dimensions. 
Let $\Omega$ be a domain of $\Real^2$ instead of $\Real^3$.
The differential complex now reads:
\begin{equation} \label{cd:2d.L2deRham}
\begin{tikzcd}
\Real \arrow[r,"i_\Omega"] & H^1(\Omega) \arrow[r,"\ROT"] & \bvec{H}(\text{div}, \Omega) \arrow[r,"\DIV"] & L^2(\Omega) \arrow[r,"0"] & \lbrace 0 \rbrace.
\end{tikzcd}
\end{equation}

We can construct a discrete complex similar to the complex in Figure \ref{fig:tikzcddiff}.
However the construction is not merely the restriction of the $3$-dimensional complex to $2$-dimensional objects.
It is in fact much simpler.

\paragraph{Discrete spaces.}\mbox{}\\
The $2$-dimensionnal complex only needs four discrete spaces $\uHrot$, $\uHv$, $\uLt$ and $\uLs$.
They are defined by:
\begin{align}
\uHroth :=& \lbrace \ul{v}_h = ((\rdofV{v})_{V \in \Vh},(v_E ,\srdofE{v})_{E \in \Eh}, (v_F)_{F \in \Fh})
\st \nonumber \\
& \quad
\begin{aligned}[t]
&\rdofV{v} \in \Real^2(V), \forall V \in \Vh,
v_{E} \in \Poly[c]{k+1}(\Ech), 
\srdofE{v} \in \Poly{k}(E), \forall E \in \Eh, \\
& v_F \in \Poly{k-1}(F),
\forall F \in \Fh
\rbrace , 
\end{aligned}\displaybreak[1]\\
\uHvh :=& \lbrace \uvec{w}_h =  ((\bvec{w}_{E})_{E \in \Eh}, (\bgvec[F]{w},\bgcvec[F]{w})_{F \in \Fh})
\st \nonumber \\
& \quad
\begin{aligned}[t]
&\bvec{w}_{E} \in \bPoly[c]{k+2}(E), \forall E \in \Eh 
\bgvec[F]{w} \in \Goly{k}(F), \bgcvec[F]{w} \in \Golyb{c,k}(F), \forall F \in \Fh
\rbrace ,
\end{aligned}\displaybreak[1]\\
\uLth :=& \lbrace \uvec{W}_h = ((\bvec{W}_E)_{E \in \Eh}, (\bvec{W}_F)_{F \in \Fh}) \st \nonumber \\
& \quad 
\begin{aligned}[t]
& \bvec{W}_E \in \bPoly{k+1}(E), \forall E \in \Eh,
\bvec{W}_F \in \RTb{k+1}(F), \forall F \in \Fh
\rbrace , 
\end{aligned}\displaybreak[1]\\
\uLsh :=& \lbrace \ul{q}_h = ( (q_F)_{F \in \Fh}) \st 
\begin{aligned}[t]
q_F \in \Poly{k}(F), \forall F \in \Fh \rbrace .
\end{aligned}
\end{align}
Figure \ref{fig:2d.tikzcddiff} is the $2$-dimensional equivalent of Figure \ref{fig:tikzcddiff}.

\begin{figure}
\begin{tikzcd}
F: & \Poly{k-1}(F) \arrow[r,"\VROT"] & \Goly{k-1}(F) \times \Golyb{c,k}(F) \arrow[r,"\DIV"] \arrow[dr,blue,"\TGRAD"]&
\Poly{k}(F) \\
E: & \Poly{k}(E) \arrow[dr,"\text{Id}"] & & {\color{blue}\RTb{k+1}(F)}\\
& \Poly{k-1}(E) \arrow[r,"\VROT"{name=U}] & \bPoly{k}(E) \arrow[r,blue,"\TGRAD"{name=D}] & {\color{blue}\bPoly{k+1}(E)}\\
V: & \Real = \Poly{k+1}(V) \arrow[dash, to=U] & \bPoly{k+2}(V) \arrow[dash,blue,to=D] \\
 & \Real^2 = \bPoly{k+2}(V) \arrow[ur,"\text{Id}"]
\end{tikzcd}
\caption{Usage of the local degrees of freedom for the discrete differential operators in $2$ dimensions.}
\label{fig:2d.tikzcddiff}
\end{figure}
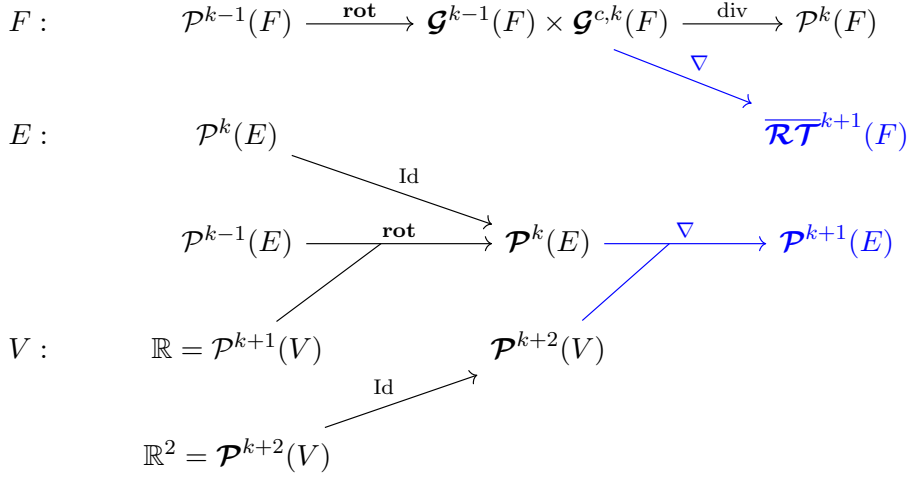

The interpolator on the space $\uHroth$ is defined for any $v \in C^1(\overline{\Omega})$ by
\begin{equation} \label{eq:defIrot}
\uIroth{v} = ((v_E, \lproj{k}{E}(\VROT v \cdot \nE ))_{E \in \Eh}, (\VROT v (V))_{V \in \Vh}, (\lproj{k-1}{F}(v))_{F \in \Fh}),
\end{equation}
where for any edge $E \in \Eh$, $v_E$ is such that $\lproj{k-1}{E}(v_E) = \lproj{k-1}{E}(v)$
and for any vertex $V \in \VE$, $v_{E} (\bvec{x}_V) = v(\bvec{x}_V)$.

The interpolator on the space $\uHvh$ is defined for any $\bvec{w} \in \bvec{C}^0(\overline{\Omega})$ by
\begin{equation} 
\uIHh{\bvec{w}} = ((\bvec{w}_{E})_{E \in \Eh}, (\Gproj{k-1}(\bvec{w}),\Gbcproj{k}(\bvec{w}))_{F \in \Fh}),
\end{equation}
where for any edge $E \in \Eh$, $\bvec{w}_E$ is such that $\vlproj{k}{E}(\bvec{w}_E) = \vlproj{k}{E}(\bvec{w})$
and for any vertex $V \in \VE$, $\bvec{w}_{E}(\bvec{x}_V) = \bvec{w}(\bvec{x}_V)$.

The interpolator on the space $\uLth$ is defined for any $\bvec{W} \in (\bvec{C}^0(\overline{\Omega})^\intercal)^2$ by
\begin{equation}
\uILh{\bvec{W}} = ((\vlproj{k+1}{E}(\bvec{W} \cdot \nE))_{E \in \Eh},(\RTbproj{k+1}(\bvec{W}))_{F \in \Fh}) .
\end{equation}

The interpolator on the space $\uLsh$ is $\lproj{k}{\Fh}$, the piecewise $L^2$-orthogonal projection on spaces $\Poly{k}(F), F \in \Fh$.

\paragraph{Operators and properties.}\mbox{}\\
The discrete operators are defined similarly to the faces operators in Section \ref{Discreteoperators}.
Thus all properties of the $3$-dimensional complex still appliable hold.
Furthermore, the complex property Theorem \ref{th:complexe} holds 
(barring the missing operator $\uGh$ and equation \eqref{eq:complexe1}).
The consistency results proven in Section \ref{Consistencyresults} also hold 
substituting the faces for the edges and the cells for the faces.

\section*{Acknowledgment}
This work has been realized with the support of MESO@LR-Platform at the University of Montpellier

\printbibliography

\end{document}